\documentclass[reqno,11pt]{amsart}

\usepackage[top=2.0cm,bottom=2.0cm,left=3cm,right=3cm]{geometry}
\usepackage{amsthm,amsmath,amssymb,dsfont}
\usepackage{mathrsfs,amsfonts,functan,extarrows,mathtools}
\usepackage[colorlinks]{hyperref}

\usepackage{marginnote}
\usepackage{xcolor}
\usepackage{stmaryrd}
\usepackage{esint}
\usepackage{graphicx}
\usepackage[english]{babel}

\newtheorem{theorem}{Theorem}[section]
\newtheorem{proposition}{Proposition}[section]

\newtheorem{remark}{Remark}[section]
\newtheorem{lemma}{Lemma}[section]

\numberwithin{equation}{section}
\allowdisplaybreaks

\def\p{\partial}

\def\d{\mathrm{d}}
\def\vr{\varrho}

\def\no{\nonumber}
\def\R{\mathbb{R}}
\def\eps{\epsilon}
\def\div{\mathrm{div}}
\def\u{\mathrm{u}}

\def\dr{\mathrm{d}}

\def\l{\left\langle}
\def\r{\right\rangle}

\def\A{\mathrm{A}}
\def\N{\mathrm{N}}
\def\g{\mathrm{g}}
\def\B{\mathrm{B}}
\def\J{\mathcal{J}}
\def\P{\mathrm{P}}
\def\Q{\mathrm{Q}}

\newcounter{wronumber}

\begin{document}
\title[Parabolic-hyperbolic compressible liquid crystal model]
			{On well-posedness of Ericksen-Leslie's parabolic-hyperbolic liquid crystal model in compressible flow}

\author[N. Jiang]{Ning Jiang}
\address[Ning Jiang]{\newline School of Mathematics and Statistics, Wuhan University, Wuhan, 430072, P. R. China}
\email{njiang@whu.edu.cn}

\author[Y.-L. Luo]{Yi-Long Luo}
\address[Yi-Long Luo]
		{\newline School of Mathematics and Statistics, Wuhan University, Wuhan, 430072, P. R. China}
\email{yl-luo@whu.edu.cn}

\author[S.-T. Tang]{Shaojun Tang$^*$}
\address[Shao-Jun Tang]
		{\newline School of Mathematics and Statistics, Wuhan University, Wuhan, 430072, P. R. China}
\email{shaojun.tang@whu.edu.cn}
\thanks{ $*$Corresponding author}

\maketitle

\begin{abstract}
   We study the Ericksen-Leslie's parabolic-hyperbolic liquid crystal model in compressible flow. Inspired by our study for incompressible case \cite{Jiang-Luo-arXiv-2017} and some techniques from compressible Navier-Stokes equations, we prove the local-in-time existence of the classical solution to the system with finite initial energy, under some constraints on the Leslie coefficients which ensure the basic energy law is dissipative. Furthermore, with an additional assumption on the coefficients which provides a damping effect, and the smallness of the initial energy, the global classical solution can be established.
\end{abstract}

\section{Introduction}

\subsection{Ericksen-Leslie's hyperbolic compressible model}
The hydrodynamic theory of liquid crystals was established by Ericksen \cite{Ericksen-1961-TSR, Ericksen-1962ARMA} and Leslie \cite{Leslie-1966-QJMAM, Leslie-1968-ARMA} during the period of 1958 through 1968, see also the survey paper \cite{Leslie-1979}. In this paper, we consider the general parabolic-hyperbolic system (derived by Leslie in \cite{Leslie-1968-ARMA}) of nematic liquid crystal in compressible flow:
\begin{equation}\label{EL-general}
  \left\{
    \begin{array}{l}
      \dot{\rho} + \rho \div \u = 0 \,, \\
      \rho \dot{\u} = \rho \mathrm{F} + \div \widehat{\sigma} \,, \\
      \rho \dot{\omega} = \rho \mathrm{G} + \widehat{g} + \div \widehat{\pi} \,.
    \end{array}
  \right.
\end{equation}

The above system is consisted of the equations of $(\rho, \u, \dr)$ with variables $ ( x ,t ) \in \R^N \times \R^+ $ $(N = 2 , 3)$. It is a macroscopic description of the time evolution of the material under influence of both the flow velocity field $\u=(\u_1,\cdots, \u_N)^\top$ and the macroscopic description of the microscopic orientation configurations $\dr= (\dr_1,\cdots, \dr_N)^\top$ (with geometric constraint $|\dr|=1$) of rod-like liquid crystals.  It represents the conservations laws of mass, linear momentum and angular momentum respectively. Here, $\rho$ is the fluid density. Furthermore, $\widehat{g}$ is the intrinsic force associated with $\dr$, $\widehat{\pi}$ is the director stress, $\mathrm{F}$ and $\mathrm{G}$ are external body force and external director body force, respectively. We use the notation $\partial_i = \partial_{x_i}$, and $\nabla = ( \partial_1, \cdots, \partial_N )^\top$ and $\div = \sum_{i = 1}^N \partial_i $. The superposed dot denotes the material derivative $\partial_t + \u\cdot\nabla$, and
$$\omega = \dot \dr = \partial_t \dr + \u \cdot \nabla \dr$$
represents the material derivative of $\dr$. Throughout this paper, the same index occurring in a monomial term means taking summation, for instance, $A_i B_i = \sum_{i = 1}^N A_i B_i$.

In the system \eqref{EL-general}, the constitutive relations for $\widehat{\sigma}$, $\widehat{\pi}$ and $\widehat{g}$ are given by:
\begin{equation}\label{Constitutive}
  \begin{aligned}
    \widehat{\sigma}_{ji}& = -p\delta_{ij} - \rho^2 \tfrac{\partial W}{ \partial \rho } \delta_{ij} - \rho \tfrac{\partial W}{\partial ( \partial_j \dr_{k}) } \partial_i \dr_{k} + \sigma_{ji}\,,\\
    \widehat{\pi}_{ji}& = \rho \tfrac{\partial W}{\partial ( \partial_j \dr_{i} ) } + \alpha_0 D_j \dr_i\,,\\
    \widehat{g}_{i} & = \Gamma \dr_i - \partial_j ( \alpha_0 D_j \dr_i ) - \rho \tfrac{\partial W}{\partial \dr_i} + g_i\,.
  \end{aligned}
\end{equation}
Here $p = p(\rho)$ is the pressure, the scalar function $\Gamma$ are Lagrangian multipliers for the constraint $|\dr|=1$ and $ D_j = \dr_j \partial_k \dr_{k} - \dr_k \partial_k \dr_{j} $. The terms involving the coefficient $\alpha_0$ are significant only if the director stress is specified as a boundary condition. $\rho W$ is the Oseen-Frank energy functional for the equilibrium configuration of a unit director field:
\begin{equation}\label{Oseen-Frank-energy}
  \begin{aligned}
    2 \rho W =& k_1 (\div \dr)^2 + k_2 |\dr \cdot (\nabla \times \dr)|^2 + k_3 |\dr \times (\nabla \times \dr)|^2 \\
    & + (k_2 + k_4) \left[ \mathrm{tr} (\nabla \dr)^2 - (\div \dr)^2 \right]\, ,
  \end{aligned}
\end{equation}
where the coefficients $k_1,\ k_2,\ k_3$, and $k_4$ are the measure of viscosity, depending on the material and the temperature.

The kinematic transport $\g$ is given by:
\begin{equation}\label{hat-g}
  g_i = \lambda_1 \N_i + \lambda_2 \dr_j \A_{ij}
\end{equation}
which represents the effect of the macroscopic flow field on the microscopic structure. The following notations
\begin{equation*}
\begin{aligned}
  \A = \tfrac{1}{2}(\nabla \u + \nabla^\top\u)\,,\quad \B= \tfrac{1}{2}(\nabla \u - \nabla^\top\u)\,, \quad \N = \omega + \B \dr\,,
\end{aligned}
\end{equation*}
represent the rate of strain tensor, skew-symmetric part of the strain rate and the rigid rotation part of director changing rate by fluid vorticity, respectively. Here $\A_{ij} = \tfrac{1}{2} (\partial_j \u_i + \partial_i \u_j)$, $ \B_{ij} = \tfrac{1}{2} (\partial_j \u_i - \partial_i \u_j) $, $(\B \dr)_i =\B_{ki} \dr_k$. The material coefficients $\lambda_1$ and $\lambda_2$ reflects the molecular shape and the slippery part between the fluid and the particles. The first term of \eqref{hat-g} represents the rigid rotation of the molecule, while the second term stands for the stretching of the molecule by the flow.

The stress tensor $\sigma$ has the following form:
\begin{equation}\label{Extra-Sress-sigma}
  \begin{aligned}
    \sigma_{ji}= \xi \A_{kk} \delta_{ij} +  \mu_1 \dr_k \A_{kp}\dr_p  \dr_i \dr_j + \mu_2  \dr_j \N_i  + \mu_3 \dr_i \N_j  + \mu_4 \A_{ij} + \mu_5 \A_{ik}\dr_k \dr_j   + \mu_6 \dr_i \A_{jk}\dr_k \,.
  \end{aligned}
\end{equation}
These coefficients $\mu_i (1 \leq i \leq 6)$ and $\xi$ which may depend on material and temperature, are usually called Leslie coefficients, and are related to certain local correlations in the fluid. Usually, the following relations are frequently introduced in the literature.
\begin{equation}\label{Coefficients-Relations}
  \lambda_1=\mu_2-\mu_3\,, \quad\lambda_2 = \mu_5-\mu_6\,,\quad \mu_2+\mu_3 = \mu_6-\mu_5\,.
\end{equation}
The first two relations are necessary conditions in order to satisfy the equation of motion identically, while the third relation is called {\em Parodi's relation}, which is derived from Onsager reciprocal relations expressing the equality of certain relations between flows and forces in thermodynamic systems out of equilibrium. Under Parodi's relation, we see that the dynamics of a compressible nematic liquid crystal flow involve six independent Leslie coefficients in \eqref{Extra-Sress-sigma}.

For simplicity, in this paper, we assume the external forces vanish, that is, $\mathrm{F}=0$, $\mathrm{G}=0$, and the coefficient $\alpha_0 = 0$, which reduces to
 \begin{equation}\label{d-Dissipative-term}
  \partial_j \pi_{ji} = \partial_j \left( \rho \tfrac{\partial W}{\partial(\partial_j \dr_i)} \right) = \partial_j (\partial_j \dr_i) = \Delta \dr_i\,.
\end{equation}
 Moreover, we take $k_1 = k_2 = k_3 = 1$, $k_4 = 0$ in \eqref{Oseen-Frank-energy}, then
$$ 2 \rho W = |\dr \cdot (\nabla \times \dr)|^2 + |\dr \times (\nabla \times \dr)|^2 + \mathrm{tr} (\nabla \dr)^2\, . $$
Since $|\dr| = 1$, this can be further simplified as $ 2W = |\nabla \dr|^2 $, which implies
\begin{equation}\label{Special-OF}
  \rho \tfrac{\partial W}{\partial \dr_i} = 0,\  \rho \tfrac{\partial W}{\partial(\partial_j \dr_i)} = \partial_j \dr_i\, .
\end{equation}
Thus the third evolution equation of \eqref{EL-general} is
\begin{equation}\label{d-equation-Hyper-Parab}
  \rho \ddot{\dr} = \Delta \dr + \Gamma \dr + \lambda_1 (\dot{\dr} + \mathrm{B} \dr) + \lambda_2 \A \dr \, .
\end{equation}
Since $|\dr|=1$, it is derived from multiplying $\dr$ in \eqref{d-equation-Hyper-Parab} that
\begin{equation}\label{Lagrangian}
  \Gamma \equiv \Gamma (\rho, \u, \dr, \dot{\dr}) = - \rho  |\dot{\dr}|^2 + |\nabla \dr|^2 - \lambda_2 \dr^\top \A \dr\, .
\end{equation}

Combining the first equality of \eqref{Constitutive} and \eqref{Special-OF}, one can obtain that
\begin{equation}\no
    \div \widehat{\sigma} = - \nabla p + \div ( \tfrac{1}{2} |\nabla \dr |^2 I -  \nabla \dr \odot \nabla \dr)  + \div \sigma\,,
\end{equation}
where the symbol $ I $ indicates the identity matrix and $(\nabla \dr \odot \nabla \dr)_{ij} = \partial_i \dr_k \partial_j \dr_k\,.$ On the other hand, it can yield that by \eqref{Extra-Sress-sigma}
\begin{equation}
  \no \div \sigma= \div \big( \tfrac{1}{2} \mu_4 ( \nabla \u + \nabla^\top \u ) + \xi \div \u I \big) + \div \tilde{\sigma}\,,
\end{equation}
where
\begin{equation}\label{tilde-sigma}
  \begin{aligned}
   \tilde{\sigma}_{ji} \equiv \big{(}\tilde{\sigma}(\u, \dr, \dot{\dr})\big{)}_{ji}  = & \mu_1 \dr_k \dr_p \A_{kp} \dr_i \dr_j + \mu_2 \dr_j (\dot{\dr}_i + \B_{ki} \dr_k)  \\
   &+ \mu_3 \dr_i (\dot{\dr}_j + \B_{kj} \dr_k)    +  \mu_5 \dr_j \dr_k \A_{ki} + \mu_6 \dr_i \dr_k \A_{kj} \, ,
  \end{aligned}
\end{equation}

As a consequence, the Ericksen-Leslie hyperbolic liquid crystal model for a compressible flow has the form:
\begin{equation}\label{compressible-Liquid-Crystal-Model}
  \begin{aligned}
    \left\{ \begin{array}{c}
      \p_t \rho + \div(\rho \u) = 0\, ,\\[2mm]
      \p_t ( \rho \u ) + \div ( \rho \u \otimes \u ) + \nabla p = \div ( \Sigma_1 + \Sigma_2 + \Sigma_3 ) \, ,\\[2mm]
      \rho \ddot{\dr} = \Delta \dr + \Gamma \dr + \lambda_1 (\dot{\dr} + \B \dr) + \lambda_2 \A \dr\, ,
    \end{array}\right.
  \end{aligned}
\end{equation}
on $ \R^N \times \R^+$ $(N=2 ,  3)$ with the geometric constraint $|\dr|=1$. The symbol $\otimes$ denotes the tensor product for two vectors with entries $ (a \otimes b)_{ij} = a_i b_j$ for $1 \leq i, j \leq 3$. For the simplicity, we assume that the pressure $ p $ obeys the $\gamma$-law, i.e. $ p(\rho) = a \rho^\gamma $ with $ \gamma \geq 1 $, $ a>1 $. And the notation $\Sigma_i$ (i = 1,2,3) are as follows:
\begin{align*}
  \Sigma_1 := & \tfrac{1}{2} \mu_4 ( \nabla \u + \nabla^\top \u ) + \xi \div \u I \, , \\[1.5mm]
  \Sigma_2 := & \tfrac{1}{2} |\nabla \dr|^2 I - \nabla \dr \odot \nabla \dr \, , \\[1.5mm]
  \Sigma_3 := &\tilde{\sigma} \, .
\end{align*}

\subsection{Previous results: incompressible model}
Since Ericksen and Leslie developed their continuum theory of liquid crystals, there have been remarkable research developments in liquid crystals from both theoretical and applied aspects. When the surrounding fluid is incompressible, Leslie (\cite{Leslie-1968-ARMA}) derived the corresponding incompressible version system of \eqref{compressible-Liquid-Crystal-Model} for $(\u,\dr)$:
\begin{equation}\label{PHLC}
  \begin{aligned}
    \left\{ \begin{array}{c}
      \partial_t \u + \u \cdot \nabla \u - \frac{1}{2} \mu_4 \Delta \u + \nabla p = - \div (\nabla \dr \odot \nabla \dr) + \div \tilde{\sigma}\, , \\
      \div \u = 0\, ,\\
     \rho_1 \ddot{\dr} = \Delta \dr + \gamma \dr + \lambda_1 (\dot{\dr} + \B \dr) + \lambda_2 \A \dr\, ,
    \end{array}\right.
  \end{aligned}
\end{equation}
where $\rho_1> 0$ is the inertia constant, and other notations are the same as before.  Let us remark that a particularly important special case of Ericksen-Leslie's model \eqref{PHLC} is that the term $\div \tilde{\sigma}$ vanishes. Namely, the coefficients $\mu_i's$, $(1\leq i \leq 6, i \neq 4)$ of $\div \tilde{\sigma}$ are chosen as 0, which immediately implies $\lambda_1 = \lambda_2 = 0$. Consequently, the system \eqref{PHLC} reduces to a model which is Navier-Stokes equations coupled with a wave map from $\mathbb{R}^N$ to $\mathbb{S}^{N-1}$:
\begin{align}\label{NS-WM}
  \left\{ \begin{array}{c}
    \partial_t \u + \u \cdot \nabla \u + \nabla p = \frac{1}{2}\mu_4 \Delta \u - \div (\nabla \dr \odot \nabla \dr )\, , \\
    \div \u =0\, ,\\
    \rho_1\ddot{\dr} = \Delta \dr + (-\rho_1 |\dot{\dr}|^2+|\nabla \dr|^2) \dr\, .
  \end{array}\right.
\end{align}

\subsubsection{$\rho_1=0, \lambda_1=-1$, incompressible parabolic model}

When the coefficients $\rho_1=0$ and $\lambda_1=-1$ in the third equation of \eqref{PHLC}, the system reduces to the parabolic type equations, which are also called Ericksen-Leslie's system in the literatures. This is the most well-developed case for analytical studies. If the fluid containing nematic liquid crystals is at rest, we have the well-known Oseen-Frank theory for static nematic liquid crystals, whose mathematical study was initialed from Hardt-Kinderlehrer-Lin \cite{Hardt-Kinderlehrer-Lin-CMP1986} on the existence and partial regularity from the point view of calculous of variations.. Since then there have been many works in this direction. In particular, the existence and regularity or partial regularity of the approximation (usually Ginzburg-Landau approximation as in \cite{Lin-Liu-CPAM1995}) dynamical Ericksen-Leslie's system was started by the work of Lin and Liu in \cite{Lin-Liu-CPAM1995}, \cite{Lin-Liu-DCDS1996} and \cite{Lin-Liu-ARMA2000}. For the simplest system preserving the basic energy law
\begin{align}\label{simplified-model}
  \left\{ \begin{array}{c}
    \partial_t \u + \u \cdot \nabla \u + \nabla p =  \Delta \u - \div (\nabla \dr \odot \nabla \dr )\, , \\
    \div \u =0\, ,\\
    \partial_t \dr + \u\cdot \nabla \dr = \Delta \dr + |\nabla \dr|^2 \dr\,,\quad |\dr|=1\,,
  \end{array}\right.
\end{align}
which can be obtained by neglecting the Leslie stress and specifying some elastic constants. In 2-D case, global weak solutions with at most a finite number of singular times was proved by Lin-Lin-Wang \cite{Lin-Lin-Wang-ARMA2010}. The uniqueness of weak solutions was later on justified by Lin-Wang \cite{Lin-Wang-CAMS2010} and Xu-Zhang \cite{Xu-Zhang-JDE2012}. Recently, Lin and Wang proved global existence of weak solution for 3-D case in \cite{Lin-Wang-CPAM2016}.

For the more general parabolic Ericksen-Leslie's system,  local well-posedness is proved by Wang-Zhang-Zhang in \cite{Wang-Zhang-Zhang-ARMA2013}, and in \cite{Huang-Lin-Wang-CMP2014} regularity and existence of global solutions in $\mathbb{R}^2$ was established by Huang-Lin-Wang. The existence and uniqueness of weak solutions, also in $\mathbb{R}^2$ was proved by Hong-Xin and Li-Titi-Xin in \cite{Hong-Xin-2012} \cite{Li-Titi-Xin} respectively. Similar result was also obtained by Wang-Wang in \cite{Wang-Wang-2014}. For more complete review of the works for the parabolic Ericksen-Leslie's system, please see the reference listed above.

\subsubsection{$\rho_1 >0$, incompressible parabolic-hyperbolic model}
If $\rho_1>0$,  \eqref{PHLC} is an incompressible Navier-Stokes equations coupled with a wave map type system for which the corresponding mathematical theories are far from well-developed, comparing the corresponding parabolic model, which is Navier-Stokes coupled with a heat flow. The only notable exception might be for the most simplified model, say, in \eqref{NS-WM}, taking $\u=0$, and the spacial dimension is $1$. For this case, the system \eqref{NS-WM} can be reduced to a so-called nonlinear variational wave equation. Zhang and Zheng (later on with Bressan and others) studied systematically the dissipative and energy conservative solutions in series work starting from late 90's \cite{Zhang-Zheng-AA1998, Zhang-Zheng-ActaC1999, Zhang-Zheng-ARMA2000, Zhang-Zheng-CAMS2001, Zhang-Zheng-CPDE2001, Zhang-Zheng-PRSE2002, Zhang-Zheng-ARMA2003, Zhang-Zheng-AIPA2005, Bressan-Zhang-Zheng-ARMA2007, Zhang-Zheng-ARMA2010, Zhang-Zheng-CPAM2012, Chen-Zhang-Zheng-ARMA2013}.

For the multidimensional case, to our best acknowledgement, there was very few mathematical work on the original hyperbolic Ericksen-Leslie's system \eqref{PHLC}. De Anna and Zarnescu \cite{DeAnna-Zarnescu-2016} considered the inertial Qian-Sheng model of liquid crystals which couples a hyperbolic type equation involving a second order derivative with a forced incompressible Navier-Stokes equations. It is a system describing the hydrodynamics of nematic liquid crystals in the Q-tensor framework. They proved global well-posedness and twist-wave solutions. Furthermore, for the inviscid version of the Qian-Sheng model, in \cite{FRSZ-2016}, Feireisl-Rocca-Schimperna-Zarnescu proved a global existence of the {\em dissipative solution} which is inspired from that of incompressible Euler equation defined by P-L. Lions \cite{Lions-1996}. Recently, in \cite{Jiang-Luo-arXiv-2017} the first two-named authors of the current papers proved the local well-posedness of \eqref{PHLC} under the mild coefficients under which the basic energy law is dissipative, and with additional damping effect assumption, i.e. $\lambda_1 <0$, global in time classical solutions with small initial data was also proved. When this assumption is not satisfied, even for the simplest case \eqref{NS-WM}, the global well-posedness is open. Furthermore, when the bulk velocity $\u$ is given, the zero inertia limit $\rho_1 \rightarrow 0$ from hyperbolic to parabolic system is rigorously justified in \cite{Jiang-Luo-Tang-Zarnescu}.

\subsection{Previous results: compressible model}
For the compressible liquid crystal model, there are also extensive studies for simplified parabolic system, i.e. the following system which is basically a coupling of compressible Navier-Stokes equations and parabolic heat flow: \begin{equation}\label{compressible-parabolic}
  \begin{aligned}
    \left\{ \begin{array}{c}
      \p_t \rho + \div(\rho \u) = 0\, ,\\[2mm]
      \p_t ( \rho \u ) + \div ( \rho \u \otimes \u ) + a\nabla \rho^\gamma = \mathcal{L}\u-\div ( \nabla\dr\odot \nabla \dr-\tfrac12|\nabla \dr|^2I) \, ,\\[2mm]
      \partial_t{\dr} +\u\cdot\nabla\dr= \Delta \dr  + |\nabla \dr|^2\dr\, ,
    \end{array}\right.
  \end{aligned}
\end{equation}
where $\mathcal{L}\u= \mu\Delta\u+(\mu+\lambda)\nabla\div \u$. In dimension one, the existence of global strong solutions and weak solutions to \eqref{compressible-parabolic} has been obtained by \cite{DLWW-DSDC2012, DWW-DSDC2011} respectively. In dimension two, the existence of global weak solutions, under the assumptions that the image of the initial data of $\dr$ is contained in $\mathbb{S}^2_+$, was obtained by \cite{JJW-JFA2013}. In dimension three, the local existence of strong solutions has been studied by \cite{HWW-JDE2012} and \cite{HWW-ARMA2012}. The incompressible limit of compressible nematic liquid crystal flow  has been studied by \cite{DHWZ-JFA2013}. When considering the compressible nematic liquid crystal flow  under the assumption that the director $\dr$ has variable degree of orientations, the global existence of weak solutions in dimension three has been obtained by \cite{Liu-Qing-DSDC2013} and \cite{Wang-Yu-ARMA2012}. Recently, inspired by the work of \cite{Lin-Wang-CPAM2016} for parabolic incompressible flow, the corresponding global finite energy weak solutions to \eqref{compressible-parabolic} was proved in \cite{Lin-Lai-Wang-SIMA2015}. Notice that for general parabolic liquid crystal model in compressible flow, i.e. the  system \eqref{compressible-Liquid-Crystal-Model} without $\rho\ddot{\dr}$ term and with $\lambda_1 = -1$, there are very few studies so far.

\subsection{Main results of this paper}
In this paper, we study the Cauchy problem of the most general parabolic-hyperbolic Ericksen-Leslie's system in compressible flow \eqref{compressible-Liquid-Crystal-Model} with the initial conditions
\begin{equation} \label{Initial-Data}
  (\rho,\ \u,\ \dr,\ \dot{\dr}) (x,t) |_{t=0} = (\rho^{in},\ \u^{in},\ \dr^{in},\ \tilde{\dr}^{in}) (x) \in \R \times \R^N \times \mathbb{S}^{N-1} \times \R^N \, ,
\end{equation}
where $\dr^{in}$ and $\tilde{\dr}^{in}$ satisfy the geometric constraint $|\dr^{in}| =1$ and compatibility condition $\dr^{in} \cdot \tilde{\dr}^{in} =0$.

To state our main results, we first introduce some notations.
We denote by $\l \cdot, \cdot \r$ the usual $L^2$-inner product in $\mathbb{R}^N$, by $|\cdot|_{L^2}$, $|\cdot|_{H^s}$ its corresponding $L^2$-norm and higher order derivatives $H^s$-norm , respectively, and $|\cdot|_{\dot{H}^s}$ the homogeneous $H^s$-norm. We define the following two Sobolev weighted-norms as
\begin{align*}
  |f|_{H^s_{\phi}} := \Big( \sum_{k=0}^{s} \int_{\R^N} \phi |\nabla^k f|^2 \d x \Big)^{\frac{1}{2}} , \quad
  |f|_{\dot{H}^s_{\phi}} := \Big( \sum_{k=1}^{s} \int_{\R^N} \phi |\nabla^k f|^2 \d x \Big)^{\frac{1}{2}}\,,
\end{align*}
and denote the quantity $\mathcal{N}_s (\rho)$ as
\begin{align*}
  \mathcal{N}_s (\rho) = |\rho|^2_{\dot{H}^s_{\frac{p'(\rho)}{\rho}}} + \tfrac{2 a}{\gamma -1} |\rho|^\gamma_{L^\gamma} \,.
\end{align*}
Furthermore, we define the following energy:
\begin{equation}\label{energy-norm}
  E(t)= \mathcal{N}_s (\rho) \!+\! |\u|_{H^s_\rho}^2 \!+\! |\dot{\dr}|_{H^s_\rho }^2 \!+\! |\nabla \dr|_{H^s}^2\,.
\end{equation}
In particular, the initial energy is
$$
  E^{in} = \mathcal{N}_s (\rho^{in}) + | \u^{in}|_{H^s_{\rho^{in}}}^2 + | \tilde{\dr}^{in} |_{H^s_{\rho^{in}}}^2 + |\nabla \dr^{in}|_{H^s}^2 \, .
  $$

At last, the notation $A \lesssim B$ will be used in this paper to indicate that there exists some constant $C>0$ such that $A \leq CB$.

The $L^2$-estimate which is the so-called basic energy law plays a fundamental role in our analysis for \eqref{compressible-Liquid-Crystal-Model}. However, it is not {\em dissipative} for all the Leslie coefficients. We first give some constraints on the Leslie coefficients such that the basic energy law is dissipative, which is presented in Section 2. Under these coefficients constraints, we prove the local-in-time existence of classical solutions to the system \eqref{compressible-Liquid-Crystal-Model} with large initial data. More precisely, the first main result is stated as follows:

\begin{theorem}[Local existence]\label{theorem-1}
  Let $s>\tfrac{N}{2} +1$. If $E^{in} \!<\! +\infty$, $\underline{\rho} \!\leq\! \rho^{in} \!\leq\! \bar{\rho}$ for some constant $\underline{\rho}, \bar{\rho} >0$ and the Leslie coefficients satisfy
  $$
  \mu_1 \geq 0\,,\ \mu_4 >0\,,\ \tfrac{1}{2} \mu_4 + \xi \geq 0 \,, \ \lambda_1 \leq 0\,,\ \mu_5 + \mu_6 + \tfrac{\lambda^2_2}{\lambda_1} \geq 0\,,
  $$
  then there exist $T, C_0>0$, depending only on $E^{in}$ and the Leslie coefficients, such that the system \eqref{compressible-Liquid-Crystal-Model}-\eqref{Initial-Data} admits the unique solution $(\rho, \u, \dr)$ satisfying
  \begin{align*}
    \rho \in & L^\infty \big( 0,T; \dot{H}^s_{\frac{p'(\rho)}{\rho}} (\R^N) \cap L^\gamma (\R^N) \big) \,,\\[2mm]
    \u \in & L^\infty \big( 0,T; H^s_\rho (\R^N) \big) \cap L^2 \big( 0, T; H^{s+1} (\R^N) \big) \,, \\[2mm]
    \dot{\dr} \in & L^\infty \big( 0,T; H^s_\rho (\R^N) \big) \,,\ \nabla \dr \in L^\infty \big( 0,T; H^s(\R^N) \big) \,.
  \end{align*}
  Moreover, the energy inequality
  \begin{align*}
    E(t)+\! \tfrac{1}{2}\mu_4\! \int_{0}^{t}\! |\nabla \u|_{H^s (\R^N)}^2 \d s \!\leq\! C_0
  \end{align*}
  holds for all $t \in [0,T]$.
\end{theorem}

We then study the global existence of classical solutions to the system \eqref{compressible-Liquid-Crystal-Model}. We rewrite the density function $\rho(x,t)$ as the form $$\rho(x,t) = 1 + \vr(x,t).$$
Consequently, the system \eqref{compressible-Liquid-Crystal-Model} is of a new form
\begin{equation}\label{Expan-System}
  \begin{aligned}
  \left \{ \begin{array} {c}
     \p_t \vr + \u \cdot \nabla \vr + (1 + \vr) \div \u = 0 \, , \\[1.5mm]
     \p_t \u + \u \cdot \nabla \u + \tfrac{p'(1+\vr)}{1+\vr} \nabla \vr = \tfrac{1}{1 + \vr} \div (\Sigma_1 + \Sigma_2 + \Sigma_3) \, , \\[1.5mm]
     \ddot{\dr} = \tfrac{1}{1 + \vr} \Delta \dr + \tfrac{1}{1 + \vr} \widetilde{\Gamma} \dr + \tfrac{1}{1 + \vr} \lambda_1 (\dot{\dr} + \B \dr) + \tfrac{1}{1 + \vr} \lambda_2 \A \dr \, ,
  \end{array} \right.
  \end{aligned}
\end{equation}
where $\widetilde{\Gamma} = - (1+\vr) |\dot{\dr}|^2 + |\nabla \dr|^2 - \lambda_2 \dr^\top \A \dr $.
The initial data of the system \eqref{Expan-System} is imposed on
\begin{align}\label{Expan-Initial-Data}
  \vr \big{|}_{t=0} = \vr^{in}(x), \ \u \big{|}_{t=0} = \u^{in}(x), \ \dr \big{|}_{t=0} = \dr^{in}(x), \ \dot{\dr} \big{|}_{t=0} = \tilde{\dr}^{in}(x) \,,
\end{align}
in which the function $\vr^{in}(x) = \rho^{in} (x) - 1$. Here the (vector) functions $\rho^{in} (x)$, $\u^{in} (x)$, $\dr^{in} (x)$ and $\tilde{\dr}^{in} (x)$ are given in \eqref{Initial-Data}.

We define the initial energy as
\begin{align}\label{Expan-Initial-Energ}
  \tilde{E}^{in} = |\u^{in}|_{H^s}^2 + |\vr^{in}|_{H^s}^2 + |\tilde{\dr}^{in}|_{H^s}^2 + |\nabla \dr^{in}|_{H^s}^2 \,.
\end{align}
We prove that if the Leslie coefficients relations in Theorem \ref{theorem-1} still hold, and, in addition, $\lambda_1 <0$, then the system \eqref{Expan-System} with the initial conditions \eqref{Expan-Initial-Data} has the unique global classical solution under the small size of the initial energy $\tilde{E}^{in}$.

Now, we precisely state our second main theorem of this paper as follows.
\begin{theorem}[Global existence]\label{theorem-2}
  Let $s > \tfrac{N}{2} +1$ and the Leslie coefficients satisfy
  $$
  \mu_1 \geq 0, \ \mu_4 > 0, \ \tfrac{1}{2} \mu_4 + \xi \geq 0 \,, \ \lambda_1 < 0, \ \mu_5 + \mu_6 + \tfrac{\lambda_2^2}{\lambda_1} \geq 0 \,.
  $$
   Then there exists an $\eps >0$, depending only on Leslie coefficients, such that if $\tilde{E}^{in} < \eps$, then the system \eqref{Expan-System}-\eqref{Expan-Initial-Data} admits the unique global solution $(\vr, \u, \dr)$ satisfying
  \begin{align*}
    & \vr,\ \dot{\dr},\ \nabla \dr \in L^\infty (\R^+; H^s(\R^N))\,, \\
    & \u \in L^\infty (\R^+; H^s(\R^N)) \cap L^2(\R^+; \dot{H}^{s+1}(\R^N)) \,.
  \end{align*}

  Furthermore, the energy bound
  \begin{align*}
  & \sup_{t \geq 0} \big( |\u|_{H^s}^2 + |\vr|_{H^s}^2 + |\dot{\dr}|_{H^s}^2 + |\nabla \dr|_{H^s}^2 \big) (t) \\
  & + \tfrac{1}{2} \mu_4 \int_{0}^{\infty} |\nabla \u|_{H^s}^2 \d t + (\tfrac{1}{2} \mu_4 + \xi) \int_{0}^{\infty} |\div \u|_{H^s}^2 \d t \leq C_1 \tilde{E}^{in} \,,
  \end{align*}
  holds for some constant $C_1>0$, depending only on Leslie coefficients.
\end{theorem}

We now sketch the main ideas and the novelties of the proof of the above theorems. It is well-known that the geometric constraint $|\dr|=1$ brings difficulties (particularly in higher order nonlinearities) on the Ericksen-Leslie's system, even in the parabolic case. We treat this difficulty as we did in \cite{Jiang-Luo-arXiv-2017} the incompressible hyperbolic case \eqref{PHLC}. More specifically, if the initial data $\dr^{in},\ {\tilde\dr}^{in}$ satisfy $|\dr^{in}|=1$ and the compatibility condition $ {\tilde\dr}^{in} \cdot \dr^{in} = 0$, then for the {\em solution} to \eqref{compressible-Liquid-Crystal-Model}, the constraint $|\dr|=1$ will be {\em forced} to hold. Hence, these constraints need only be given on the initial data, while in the system \eqref{compressible-Liquid-Crystal-Model}, we do not need the constraint $|\dr| =1$ explicitly any more.

Comparing to our previous work on the parabolic-hyperbolic Ericksen-Leslie system for incompressible flow \cite{Jiang-Luo-arXiv-2017}, in compressible case considered in the current paper, there appears some new difficulties, in particular on the estimates of the density. To overcome these difficulties we employ some techniques inspired by the studies in compressible Navier-Stokes equations. The first is the $L^\infty$-bound of the density $\rho (x,t) $, which is derived by the following inequalities from the mass conservation law under the initial density $\rho^{in} (x)$ with positive lower bound $\underline{\rho}$ and upper bound $\overline{\rho}$ (see Lemma \ref{rho-bound-Lemma}):
\begin{equation}\label{Bd-rho}
   \underline{\rho} \exp\Big\{ - \int_{0}^{t} |\div \u|_{L^\infty}(s) \d s\Big\} \leq \rho(x,t) \leq \overline{\rho} \exp\Big\{ \int_{0}^{t} |\div \u|_{L^\infty} (s) \d s\Big\} \,.
\end{equation}
As a result, in the derivation of the {\em a priori} estimates, the $L^\infty$-bounds of the density $\rho (x,t) $ and $\frac{1}{\rho (x,t)}$ can be bounded by the $L^\infty$-norm of $\div \u (x,t)$.

Our next step is that for a given smooth density $\rho(x,t)$ and velocity field $\u(x,t)$, the equation of $\dr(x,t)$ can be solved locally in time. We use the mollifier method to construct the approximate solutions {\em without} assuming the constraint $|\dr|=1$. The key point is that we can derive the uniform energy estimate also without this constraint. Then as mentioned above, the condition $|\dr|=1$ will be automatically obeyed in the time interval that the solution is smooth.

Now we carefully design an iteration scheme to construct approximate solutions $(\rho^k, \u^k, \dr^k)$: solve the mass conservation equation of $\rho^{k+1}$ with the velocity field $\u^k$, the Stokes type equation of $\u^{k+1}$ in terms of $\rho^k$, $\u^k$ and $\dr^k$, and the wave map type equation of $\dr^{k+1}$ in terms of $\rho^k$ and $\u^{k}$. Note that the two keys in this construction are as the follows: 1. the inequalities \eqref{Bd-rho} for $\rho^{k+1}$ are preserved in the iteration approximate systems, so that the norms $| \rho^{k+1}(\cdot, t) |_{L^\infty}$ and $\big| \frac{1}{\rho^{k+1}(\cdot, t)} \big|_{L^\infty}$ are dominated by $| \div \u^{k} ( \cdot, t ) |_{L^\infty} $; 2. the geometric constraint of $\dr^k$ will not be a trouble. So we can derive the uniform energy estimate by assuming $|\dr|=1$ which will significantly simply the process. We emphasize that in this step the basic energy law play a fundamental role: the main term of the higher order energy are kept in the form of the basic $L^2$-estimate. Thus the existence of the local-in-time smooth solutions can be proved.

To prove the global-in-time smooth solutions, the dissipation in the {\em a priori} estimate in the last step is not enough. We assume furthermore the coefficient $\lambda_1$ is strictly negative which will give us an additional dissipation on the direction field $\dr$, based on a new {\em a priori} estimate by carefully designing the energy and energy dissipation. Since the pressure $p(\rho)$, depending on the density function $\rho$, satisfies $p'(\rho) > 0$ for $\rho > 0$, the term $\nabla p (\rho)$ in the velocity equation of \eqref{compressible-Liquid-Crystal-Model} will give us a more dissipation on the density $\rho$ by multiplying $\nabla \rho$ (or $\nabla^k \rho$ for higher order estimates) in the $\u$-equation. Then we can show the existence of the global smooth solution with the small initial data.

This paper is organized as follows: we devote the Section \ref{Sect-2} to prove the basic energy law and give the conditions on the coefficients such that the energy is dissipative. In Section \ref{Sect-3}, the {\em a priori} estimate of the equations \eqref{compressible-Liquid-Crystal-Model} is obtained by using energy method. The local well-posedness of $\dr$ for a given density $\rho$ and bulk velocity field $\u$ is then constructed in Section \ref{Sect-4}, which will be used in constructing the approximate scheme of the system \eqref{compressible-Liquid-Crystal-Model}-\eqref{Initial-Data}. In Section \ref{Sect-5}, the local existence of the system \eqref{compressible-Liquid-Crystal-Model} with large initial data, namely, Theorem \ref{theorem-1}  is proved through obtaining uniform energy estimate of the approximate system \eqref{Approx-Equat}. Finally the proof of Theorem \ref{theorem-2} are presented in Section \ref{Sect-6}.

\section{Basic Energy law}\label{Sect-2}
The main purpose of this section is to deduce the basic energy law of the liquid crystal model \eqref{compressible-Liquid-Crystal-Model}, which will play an important role in the energy estimate. We shall assume first that the solutions $ (\rho, \u, \dr) $ to \eqref{compressible-Liquid-Crystal-Model}-\eqref{Initial-Data} are sufficiently smooth.

\begin{proposition}\label{Prop-Basic-Ener-Law}
  If $ (\rho, \u, \dr) $ is a smooth solution to the system \eqref{compressible-Liquid-Crystal-Model}, then the following energy identity holds:
  \begin{equation}\label{Basic-Ener-Law}
    \begin{aligned}
      & \tfrac{1}{2} \tfrac{\d}{\d t} \int_{\mathbb{R}^n} \Big{(} \tfrac{2a}{\gamma-1} \rho^\gamma + \rho (|\u|^2 + |\dot{\dr}|^2) + |\nabla \dr|^2 \Big{)} \d x + \tfrac{1}{2} \mu_4 |\nabla \u|^2_{L^2} + ( \tfrac{1}{2} \mu_4 + \xi ) |\div \u|^2_{L^2} \\
      & + \mu_1 |\dr^\top \A \dr|^2_{L^2} - \lambda_1 |\dot{\dr} + \B \dr + \tfrac{\lambda_2}{\lambda_1} \A \dr|^2_{L^2} + ( \mu_5 + \mu_6 + \tfrac{\lambda_2^2}{\lambda_1} ) |\A \dr|^2_{L^2} = 0\,.
    \end{aligned}
  \end{equation}
  Moreover, the basic energy law \eqref{Basic-Ener-Law} is dissipated if and only if Leslie coefficients satisfy
  $$ \mu_1 \geq 0\,, \mu_4 > 0\,, \tfrac{1}{2} \mu_4 + \xi \geq 0 \,, \lambda_1 \leq 0\,, \mu_5 + \mu_6 + \tfrac{\lambda_2^2}{\lambda_1} \geq 0\,. $$
\end{proposition}
\begin{proof}

  Multiplying the equation \eqref{compressible-Liquid-Crystal-Model}$_2$ (the second equation of \eqref{compressible-Liquid-Crystal-Model}) by $ \u $ and integrating on the whole space with respect to $x$, one obtains
  \begin{equation*}
    \begin{aligned}
      & \tfrac{1}{2} \tfrac{\d}{\d t} \int_{\mathbb{R}^n} \rho |\u|^2 \d x + \tfrac{1}{2} \mu_4 | \nabla \u |^2_{L^2} + \big{(} \tfrac{1}{2} \mu_4 + \xi \big{)} | \div \u |^2_{L^2} \\
      = & \l a \rho^\gamma, \div \u \r - \l \div ( \nabla \dr \odot \nabla \dr - \tfrac{1}{2} |\nabla \dr |^2 I ), \u \r + \l \div \tilde{\sigma}, \u \r \, .
    \end{aligned}
  \end{equation*}
  Taking $ L^2 $-inner product with $ \dot{\dr} $ in the third equation of \eqref{compressible-Liquid-Crystal-Model}, and making use of the fact that $ \dot{\dr} \cdot \dr = 0 $, we can deduce that
  \begin{equation*}
    \begin{aligned}
      & \tfrac{1}{2} \tfrac{\d}{\d t} \int_{\mathbb{R}^n} \rho | \dot{\dr} |^2 + | \nabla \dr |^2 \d x - \lambda_1 | \dot{\dr} |^2_{L^2} \\
      = & \l \Delta \dr, \u \cdot \nabla \dr \r + \lambda_1 \l \B \dr, \dot{\dr} \r + \lambda_2 \l \A \dr, \dot{\dr} \r + \tfrac{1}{2} \l | \nabla \dr |^2, \div \u \r\, .
    \end{aligned}
  \end{equation*}
  In order to cancel the term $ \l a \rho^\gamma, \div \u \r $, we should multiply the first equation of \eqref{compressible-Liquid-Crystal-Model} with $ \frac{ \gamma p(\rho) }{ ( \gamma-1 ) \rho } $, and then we can obtain
  \begin{equation*}
    \tfrac{a}{\gamma - 1} \tfrac{\d}{\d t}\int_{\mathbb{R}^n} \rho^\gamma \d x + a \l \rho^\gamma, \div u \r = 0\, .
  \end{equation*}
  Simple calculation tells us that
   $$ \l \Delta \dr, \u \cdot \nabla \dr \r - \l \div ( \nabla \dr \odot \nabla \dr - \tfrac{1}{2} |\nabla \dr |^2 I ), \u \r = 0\, . $$
  Hence we know
  \begin{equation}\label{BEL-1}
    \begin{aligned}
      & \tfrac{1}{2} \tfrac{\d}{\d t} \int_{\mathbb{R}^n} \tfrac{2a}{\gamma-1} \rho^\gamma + \rho ( | \u |^2 + | \dot{\dr} |^2 ) + | \nabla \dr |^2 \d x + \tfrac{1}{2} \mu_4 | \nabla \u |^2_{L^2} + \big{(} \tfrac{1}{2} \mu_4 + \xi \big{)} | \div \u |^2_{L^2} \\
      & = \l \div \tilde{\sigma}, \u \r + \lambda_1 \l \B \dr, \dot{\dr} \r + \lambda_2 \l \A \dr, \dot{\dr} \r \, .
    \end{aligned}
  \end{equation}

  Now we will deal with the right-hand side terms of the above equality \eqref{BEL-1}. Based on the observation of the structure of $ \tilde{\sigma} $, we can divide it into four parts as follows:
  \begin{equation*}
    \begin{aligned}
      \tilde{\sigma}_{ji} = & \mu_1 \dr_k \dr_p \A_{kp} \dr_i \dr_j + ( \mu_2 \dr_j \B_{ki} \dr_k + \mu_3 \dr_i \B_{kj} \dr_k ) \\
      & + ( \mu_2 \dr_j \dot{\dr}_i + \mu_3 \dr_i \dot{\dr}_j ) + ( \mu_5 \dr_j \A_{ki} \dr_k + \mu_6 \dr_i \A_{kj} \dr_k ) \, .
    \end{aligned}
  \end{equation*}

  For the first part, by the fact that $ \A_{ij} = \A_{ji} $, $ \B_{ij} = - \B_{ji} $ and $ \p_j \u_i = \A_{ij} + \B_{ij} $, we can easily get
  \begin{equation}\label{BEL-2}
    \begin{aligned}
      \l \partial_j ( \mu_1 \dr_k \dr_p \A_{kp} \dr_i \dr_j  ), \u_i \r =  - \mu_1 \l \dr_k \dr_p \A_{kp} \dr_i \dr_j , \partial_j \u_i \r = - \mu_1 |\dr^\top \A \dr|^2_{L^2}\,.
    \end{aligned}
  \end{equation}

  For the second part, straightforward calculation enable us to get
  \begin{equation}\label{BEL-3}
    \begin{aligned}
      & \l \p_j ( \mu_2 \dr_j \B_{ki} \dr_k + \mu_3 \dr_i \B_{kj} \dr_k ), \u_i \r
      = - \l \mu_2 \dr_j \B_{ki} \dr_k + \mu_3 \dr_i \B_{kj} \dr_k, \A_{ij} + \B_{ij} \r \\
      = & \mu_2 \l \B_{ki} \dr_k, \B_{ji} \dr_j \r - \mu_3 \l \B_{kj} \dr_k, \B_{ij} \dr_i \r
        - \mu_2 \l \B_{ki} \dr_k, \A_{ji} \dr_j \r - \mu_3 \l \B_{kj} \dr_k, \A_{ij} \dr_i \r \\
      = & ( \mu_2 - \mu_3 ) |\B \dr|^2_{L^2} - (\mu_2 + \mu_3) \l \B \dr, \A \dr \r
      = \lambda_1 |\B \dr|^2_{L^2} + \lambda_2 \l \B \dr, \A \dr \r \, .
    \end{aligned}
  \end{equation}
  Here we have used the symmetry of the tensor $ \A $ and the skew symmetry of the tensor $ \B $, and the coefficients relation \eqref{Coefficients-Relations}.

  For the last two parts, by using the same argument as \eqref{BEL-3}, we arrive at
  \begin{equation}\label{BEL-4}
    \begin{aligned}
      & \l \p_j (\mu_2 \dr_j \dot{\dr}_i + \mu_3 \dr_i \dot{\dr}_j), \u_i \r \\
      = & - (\mu_2 + \mu_3) \l \dot{\dr}, \A \dr \r + ( \mu_2 - \mu_3 ) \l \dot{\dr}, \B \dr \r \\
      = & \lambda_2 \l \dot{\dr}, \A \dr \r + \lambda_1 \l \dot{\dr}, \B \dr \r
    \end{aligned}
  \end{equation}
  and
  \begin{equation}\label{BEL-5}
    \begin{aligned}
      & \l \p_j ( \mu_5 \dr_j \A_{kj} \dr_k + \mu_6 \dr_i \A_{kj} \dr_k ), \u_i \r \\
      = & - ( \mu_5 + \mu_6 ) | \A \dr |^2_{L^2} + ( \mu_5 - \mu_6 ) \l \A \dr, \B \dr \r \\
      = & - ( \mu_5 + \mu_6 ) | \A \dr |^2_{L^2} + \lambda_2 \l \A \dr, \B \dr \r\, ,
    \end{aligned}
  \end{equation}
  respectively.

  Therefore, the above equalities \eqref{BEL-2}, \eqref{BEL-3}, \eqref{BEL-4}, \eqref{BEL-5} along with the equality \eqref{BEL-1} imply that
  \begin{equation*}
    \begin{aligned}
      & \tfrac{1}{2} \tfrac{\d}{\d t} \int_{\mathbb{R}^n} \tfrac{2a}{\gamma-1} \rho^\gamma + \rho ( | \u |^2 + | \dot{\dr} |^2 ) + | \nabla \dr |^2 \d x + \tfrac{1}{2} \mu_4 | \nabla \u |^2_{L^2} + \big{(} \tfrac{1}{2} \mu_4 + \xi \big{)} | \div \u |^2_{L^2} \\[2mm]
      & + \mu_1 | \dr^\top \A \dr |^2_{L^2} - \lambda_1 | \dot{\dr} + \B \dr |^2_{L^2} - 2 \lambda_2 \l \dot{\dr} + \B \dr, \A \dr \r + ( \mu_5 + \mu_6 ) | \A \dr |^2_{L^2} = 0 \, .
    \end{aligned}
  \end{equation*}

  Finally, by the square method we can get the basic energy law \eqref{Basic-Ener-Law}, this completes the proof of Proposition \ref{Prop-Basic-Ener-Law}.
\end{proof}

\section{{\em A priori} estimate} \label{Sect-3}

This section is devoted to the {\em a priori} estimate for the compressible Ericksen--Leslie liquid crystal model \eqref{compressible-Liquid-Crystal-Model}-\eqref{Initial-Data}. Besides the energy norm defined in \eqref{energy-norm}, we introduce the following energy dissipation rate:
\begin{align} \label{Dissp-Energy}
 \nonumber D(t) = & \tfrac{1}{2} \mu_4 |\nabla \u|^2_{H^s} + (\tfrac{1}{2} \mu_4 + \xi) |\div \u|^2_{H^s} + (\mu_5 + \mu_6 + \tfrac{\lambda_2^2}{\lambda_1}) \sum_{k=0}^{s} |(\nabla^k \A) \dr|^2_{L^2} \\
  & - \lambda_1 \sum_{k=0}^{s} |\nabla^k \dot{\dr} + (\nabla^k \B) \dr + \tfrac{\lambda_2}{\lambda_1} (\nabla^k \A) \dr|^2_{L^2}+ \mu_1 \sum_{k=0}^{s} |\dr^\top (\nabla^k \A) \dr|^2_{L^2} \, .
\end{align}

The following three lemmas are needed in the derivation of the {\em a priori} estimate. The first is about the bound of the density $\rho$ in terms of the initial data $\rho^{in}$ and $| \div \u |_{L^\infty}$.

\begin{lemma}\label{rho-bound-Lemma}
  Assume that $ \underline{\rho} \leq \rho^{in} (x) \leq \overline{\rho} $ for some constants $\underline{\rho}, \bar{\rho} > 0$, and the density $\rho (x,t)$ satisfies the first equation of \eqref{compressible-Liquid-Crystal-Model}, then the following inequalities hold
  \begin{equation}\label{rho-bound}
    \underline{\rho} \exp\Big\{ - \int_{0}^{t} |\div \u|_{L^\infty}(s) \d s\Big\} \leq \rho(x,t) \leq \overline{\rho} \exp\Big\{ \int_{0}^{t} |\div \u|_{L^\infty} (s) \d s\Big\}\, .
  \end{equation}
\end{lemma}

\begin{proof}

  Rewriting the first equation of \eqref{compressible-Liquid-Crystal-Model} as $ \p_t \rho + \u \cdot \nabla{\rho} = - \rho \div{\u} $, then multiplying it by $ \rho^{-1} $, one has
  $$ \p_t \ln{\rho} + \u \cdot \nabla \ln{\rho} = - \div{\u}\, . $$
  By the transport property of the above equation, we will use the characteristic method to derive the bound of $\rho$. Let
  \begin{equation}\label{characteristic}
    \begin{aligned}
      \left\{ \begin{array}{c}
        \frac{\d X(t,x)}{\d t} = \u (t, X(t,x))\, , \\[1.5mm]
        X(0, x) = x \, .
      \end{array}
      \right.
    \end{aligned}
  \end{equation}
  Hence we have $ \tfrac{\d}{\d t} \ln \rho (t, X(t, x)) = -\div \u(t, X(t, x))$, then integrating on $(0, t)$ with respect to $t$ infers that
  $$  \ln \rho (t, X(t, x)) = \ln \rho^{in}(x) - \int_{0}^{t} \div \u (s, X(s, x)) \d s\, .$$
  Therefore $\rho(t, X(t, x)) = \rho^{in} (x) \exp\Big\{- \int_{0}^{t} \div \u(s, X(s, x)) \d s\Big\}$. Then by the assumption $\underline{\rho} \leq \rho^{in} (x) \leq \overline{\rho}$ we obtain the bound \eqref{rho-bound} of the density, this completes the proof.
\end{proof}

The second lemma is about the $L^2$-bound of the derivative of the function $f(\rho)$, which will be play an important role in our calculation.
  \begin{lemma} \label{rho-Hk-Lemma}
    Let $ f(\rho) $ be a smooth function, then for any positive integer $k>0$ and $\rho \in H^k \cap L^\infty $, we have
    \begin{equation}\label{rho-Hk-1}
      \nabla^k f(\rho) = \sum_{ i=1 }^{k} f^{(i)}(\rho) \sum_{\substack{\sum_{l=1}^{i} k_l = k, \\ k_l \geq 1}} \prod_{l=1}^{i} \nabla^{k_l} \rho \, .
    \end{equation}
    In particular, when $\rho$ satisfies the assumption stated in Lemma \ref{rho-bound-Lemma}, and $f(\rho) = \frac{1}{\rho}$, we then have
     \begin{align}\label{rho-H-k-2}
      \no | \nabla^k f(\rho) |_{L^2}
      & \leq \sum_{ i=1 }^{k} \frac{i !}{\underline{\rho}^{i+1}} \exp\Big\{(i+1) \int_{0}^{t} |\div \u|_{L^\infty}(s) \d s\Big\} \sum_{\substack{\sum_{l=1}^{i} k_l = k, \\ k_l \geq 1}} | \prod_{l=1}^{i} \nabla^{k_l} \rho |_{L^2} \\
      \no & \leq C( \underline{\rho}, k ) \exp\Big\{ (k+1) \int_{0}^{t} |\div \u|_{L^\infty} (s) \d s \Big\} \P_k ( |\rho|_{\dot{H}^s} ) \\
      & \leq \Q (\u) \P_k ( |\rho|_{\dot{H}^s} ) \,,
     \end{align}
    where $\Q(\u) = \kappa_1 \exp\Big\{ \kappa_2 \int_{0}^{t} |\div \u|_{L^\infty} (s) \d s \Big\}$ with generic constants $\kappa_1, \kappa_2$, and $ \P_k (x) = \sum_{j=1}^{k} x^j $.
  \end{lemma}

  \begin{proof}
    It is easy to prove the lemma by the induction, so we omit it for simplicity.
  \end{proof}

  The next lemma so-called Moser-type inequality formulated by Moser \cite{Moser-ASNSPCS-1966} concerning the estimate of commutator.

  \begin{lemma}[Moser-type inequality] \label{Moser}
    For functions $ f, g \in H^m \cap L^\infty $, $ m \in \mathbb{Z}_+ \cup \{ 0 \}$, and $ | \alpha | \leq m $, we have
    \begin{equation}\label{Moser-type}
      | \nabla^\alpha_x (f g) - f \nabla^\alpha_x g |_{L^2} \leq C ( | \nabla_x f |_{L^\infty} | \nabla^{m-1} g |_{L^2} + | \nabla^m f |_{L^2} | g |_{L^\infty})\, ,
    \end{equation}
    with the constant depends only on $s$.
  \end{lemma}

The system \eqref{compressible-Liquid-Crystal-Model} can be rewritten as
\begin{equation}\label{compressible-Liquid-Crystal-Model-1}
  \begin{aligned}
    \left\{ \begin{array}{c}
      \p_t \rho + \div(\rho u) = 0\, ,\\[2mm]
      \p_t \u  + \u \cdot \nabla \u  + \tfrac{p'(\rho)}{\rho} \nabla \rho = \tfrac{1}{\rho} \div ( \Sigma_1 + \Sigma_2 + \Sigma_3 )\, ,\\[2mm]
      \ddot{\dr} = \tfrac{1}{\rho} \Delta \dr + \tfrac{1}{\rho} \Gamma \dr + \tfrac{\lambda_1}{\rho} (\dot{\dr} + \B \dr) + \tfrac{\lambda_2}{\rho} \A \dr \, .
    \end{array}\right.
  \end{aligned}
\end{equation}
We now turn to deal with the higher order estimates of system \eqref{compressible-Liquid-Crystal-Model-1} and give the following lemma about the a prior estimate of the system \eqref{compressible-Liquid-Crystal-Model}-\eqref{Initial-Data}:

\begin{lemma}\label{A-Priori-Estimate}
  Let $s>\tfrac{N}{2} +1$, and assume that $(\rho, \u, \dr)$ is a sufficiently smooth solution to the system \eqref{compressible-Liquid-Crystal-Model}-\eqref{Initial-Data} on $[0,T]$. Then there is a constant $C>0$, depending only on Leslie coefficients and $s$, such that for all $t \in [0,T]$, we have
  \begin{align}\label{A-Priori-Estimate-0}
    \tfrac{1}{2} \tfrac{\d}{\d t} E (t) \!+\! D (t) \leq C \Big\{ \Q (\u) \sum_{l=2}^{s+6} E^{\frac{l}{2}} (t) + D^{\frac{1}{2}} (t) E (t) \big[ 1 \!+\! \Q (\u) \sum_{l=0}^{s+3} E^{\frac{l}{2}} (t) \big] \Big\} \, .
  \end{align}
\end{lemma}

\begin{proof}
For all $ 1 \leq k \leq s \ (s \geq \tfrac{N}{2} +1) $, firstly, acting $\nabla^k$ on the first equation of \eqref{compressible-Liquid-Crystal-Model-1} and taking $L^2$-inner product with $\frac{p'(\rho)}{\rho} \nabla^k \rho$ yield
  \begin{align}\label{rho-k}
    \no & \tfrac{1}{2} \tfrac{\d}{\d t} \int_{\mathbb{R}^N} \tfrac{p'(\rho)}{\rho} | \nabla^k \rho|^2 \d x +\l p'(\rho) \nabla^k \div \u, \nabla^k \rho \r + \l \u \cdot \nabla^{k+1} \rho, \tfrac{p'(\rho)}{\rho} \nabla^k \rho \r \\
    = & \tfrac{1}{2} \int_{\mathbb{R}^N} \p_t \big{(} \tfrac{p'(\rho)}{\rho}\big{)} | \nabla^k \rho |^2 \d x - \l [\nabla^k, \rho \div] \u, \tfrac{p'(\rho)}{\rho} \nabla^k \rho \r - \l [\nabla^k, \u \cdot \nabla] \rho, \tfrac{p'(\rho)}{\rho} \nabla^k \rho \r \, .
  \end{align}

Secondly, applying the multi-derivative operator $\nabla^k$ to the second equation \eqref{compressible-Liquid-Crystal-Model-1}, and taking $L^2$-inner product with $\rho \nabla^k \u$ enable us to derive the following equality:
\begin{align}\label{u-k}
    \no & \tfrac{1}{2} \tfrac{\d}{\d t} \int_{\mathbb{R}^N} \rho |\nabla^k \u|^2 \d x + \tfrac{1}{2} \mu_4 |\nabla^{k+1} \u|^2_{L^2} + ( \tfrac{1}{2} \mu_4 + \xi ) |\nabla \div \u|^2_{L^2} + \l p'(\rho) \nabla^{k+1} \rho, \nabla^k \u \r \\
    \no = & - \l [\nabla^k, \u \cdot \nabla] \u, \rho \nabla^k \u \r - \l [\nabla^k, \tfrac{p'(\rho)}{\rho} \nabla] \rho, \rho \nabla^k \u \r + \l [\nabla^k, \tfrac 1 \rho \div] \Sigma_1, \rho \nabla^k \u \r
    \\
    \no & + \l [\nabla^k, \tfrac{1}{\rho} \div] \Sigma_2, \rho \nabla^k \u \r + \l [\nabla^k, \tfrac{1}{\rho} \div] \Sigma_3, \rho \nabla^k \u \r + \l \nabla^k \div \Sigma_2, \nabla^k \u \r \\
    & + \l \nabla^k \div \Sigma_3, \nabla^k \u \r \, .
\end{align}

Thirdly, via acting $\nabla^k$ on \eqref{compressible-Liquid-Crystal-Model-1}$_3$, then multiplying by $\rho \nabla^k \dot{\dr}$ and integrating the resulting identity on $\mathbb{R}^N$ with respect to $x$, one obtains
  \begin{align}\label{d-k}
    \no & \tfrac{1}{2} \tfrac{\d}{\d t} \int_{\mathbb{R}^N} \rho |\nabla^k \dot{\dr}|^2 + |\nabla^{k+1} \dr|^2 \d x - \l \nabla^k \Delta \dr, \nabla^k (\u \cdot \nabla \dr) \r \\
    \no = & \l \nabla^k (\Gamma \dr), \nabla^k \dot{\dr} \r + \l \lambda_1 \nabla^k (\B \dr) + \lambda_2 \nabla ^k(\A \dr), \nabla^k \dot{\dr} \r + \l [\nabla^k, \tfrac{1}{\rho} \Delta] \dr, \rho \nabla^k \dot{\dr} \r \\
    \no & - \l [\nabla^k, \u \cdot \nabla] \dot{\dr}, \rho \nabla^k \dot{\dr} \r + \sum_{\substack{a+b=k, \\ a \geq 1}} \l \nabla^a ( \tfrac{1}{\rho} ) \nabla^b ( \lambda_1 \B \dr + \lambda_2 \A \dr), \rho \nabla^k \dot{\dr} ) \r \\
    & + \sum_{\substack{a+b=k, \\ a \geq 1}} \l \lambda_1 \nabla^a (\tfrac{1}{\rho}) \nabla^b \dot{\dr}, \rho \nabla^k \dot{\dr} \r + \sum_{\substack{a+b=k, \\ a \geq 1}} \l \nabla^a (\tfrac{1}{\rho}) \nabla^b (\Gamma \dr), \rho \nabla^k \dot{\dr} \r \, .
  \end{align}

Therefore, combining the equalities \eqref{rho-k}, \eqref{u-k} and \eqref{d-k} leads to
\begin{equation}\label{APE-1}
  \begin{aligned}
    & \tfrac{1}{2} \tfrac{\d}{\d t} \int_{\mathbb{R}^N} \tfrac{p'(\rho)}{\rho} | \nabla^k \rho |^2 + \rho ( |\nabla^k \u|^2 + |\nabla^k \dot{\dr}|^2 ) + |\nabla^{k+1} \dr|^2 \d x \\[2mm]
    & \quad \quad + \tfrac{1}{2} \mu_4 |\nabla^{k+1} \u|^2_{L^2} + ( \tfrac{1}{2} \mu_4 + \xi ) |\nabla^k \div \u|^2_{L^2} \\[2mm]
    & = I + J \, ,
  \end{aligned}
\end{equation}
where
  \begin{align*}
    I = & \underbrace{- \l p'(\rho) \nabla^{k+1} \rho, \nabla^k \u \r - \l p'(\rho) \nabla^k \div \u, \nabla^k \rho \r}_{I_1} - \underbrace{ \l \u \cdot \nabla^{k+1} \rho, \tfrac{p'(\rho)}{\rho} \nabla^k \rho \r }_{I_2} \\
    & + \underbrace{ \l \nabla^k \Delta \dr, \nabla^k (\u \cdot \nabla \dr) \r - \l \nabla^k \div \Sigma_2, \nabla^k \u \r}_{I_3} + \underbrace{\l \nabla^k ( \Gamma \dr ), \nabla^k \dot{\dr} \r}_{I_4} \\
    & + \underbrace{\l \lambda_1 \nabla^k ( \dot{\dr} + \B \dr ) + \lambda_2 \nabla^k ( \A \dr ), \nabla^k \dot{\dr} \r + \l \nabla^k \tilde{\sigma}, \nabla^k \u \r}_{I_5}
  \end{align*}
and
  \begin{align*}
    J = & - \l [\nabla^k, \u \cdot \nabla] \u, \rho \nabla^k \u \r - \l [\nabla^k, \tfrac{p'(\rho)}{\rho} \nabla] \rho, \rho \nabla^k \u \r - \l [\nabla^k, \u \cdot \nabla] \dot{\dr}, \rho \nabla^k \dot{\dr} \r \\[1.5mm]
    & + \l [\nabla^k, \tfrac{1}{\rho} \Delta] \dr, \rho \nabla^k \dot{\dr} \r - \l [\nabla^k, \u \cdot \nabla] \rho, \tfrac{p'(\rho)}{\rho} \nabla^k \rho \r - \l [\nabla^k, \rho \div] \u, \tfrac{p'(\rho)}{\rho} \nabla^k \rho \r \\[1.5mm]
    & + \tfrac{1}{2} \int_{\mathbb{R}^N} \p_t \big( \tfrac{p'(\rho)}{\rho} \big) |\nabla^k \rho|^2 \d x + \l [\nabla^k, \tfrac 1 \rho \div] \Sigma_1, \rho \nabla^k u \r + \l [\nabla^k, \tfrac 1 \rho \div] \Sigma_2, \rho \nabla^k u \r \\[1.5mm]
    & + \l [\nabla^k, \tfrac 1 \rho \div] \Sigma_3, \rho \nabla^k u \r + \sum_{\substack{a+b=k,\\ a \geq 1}} \l \nabla^a (\tfrac{1}{\rho}) \nabla^b (\Gamma \dr), \rho \nabla^k \dot{\dr} \r \\[1.5mm]
    & + \sum_{\substack{a+b=k,\\ a \geq 1}} \l \lambda_1 \nabla^a (\tfrac{1}{\rho}) \nabla^b \dot{\dr}, \rho \nabla^k \dot{\dr} \r + \sum_{\substack{a+b=k,\\ a \geq 1}} \l \nabla^a (\tfrac{1}{\rho}) \nabla^b (\lambda_1 \B \dr + \lambda_2 \A \dr), \rho \nabla^k \dot{\dr} \r \\[1.5mm]
    \equiv & \sum_{l=1}^{13} J_{l} \, .
  \end{align*}

We now turn to deal with $I_i(1 \leq i \leq 5)$ term by term. It is obvious that the terms in $I_1$ have a cancellation relation, so we can get by employing the H\"older inequality and Sobolev embedding theory that
  \begin{align}\label{APE-I-1}
    \no I_1 = & - \l p'(\rho) \nabla^{k+1} \rho - \nabla(p'(\rho) \nabla^k \rho), \nabla^k \u \r
     = \l p''(\rho) \nabla \rho \nabla^k \rho, \nabla^k \u \r \\
    \lesssim & Q(\u) |\nabla \rho|_{L^\infty} |\nabla^k \rho|_{L^2} |\nabla^k \u|_{L^2}
    \lesssim Q(\u) |\rho|_{\dot{H}^s}^2 |\u|_{\dot{H}^s}\, .
  \end{align}

As to $I_2$, it is easy to deduce that
  \begin{align}\label{APE-I-2}
    \no I_2 = & \l \u_l \p_l \nabla^k \rho, \tfrac{p'(\rho)}{\rho} \nabla^k \rho \r
            = -\tfrac{1}{2} \l \p_l \u_l \tfrac{p'(\rho)}{\rho} + \u_l \p_l (\tfrac{p'(\rho)}{\rho}), |\nabla^k \rho|^2 \r \\
    \no    \lesssim & Q(\u) (|\div \u|_{L^\infty} + |\u|_{L^\infty} |\nabla \rho|_{L^\infty}) |\nabla^k \rho|^2_{L^2} \\
          \lesssim & Q(\u) |\rho|^2_{\dot{H}^s} (1 + |\rho|_{\dot{H}^s}) |\u|_{H^s} \, .
  \end{align}

For the third term $I_3$, we also want to find some cancellation to reduce the order of the derivation. Based on such idea, we can infer that
  \begin{align}\label{APE-I-3}
    \no I_3 = & - \l \nabla^k \p_j \dr_p, \p_j \u_i \nabla^k \p_i \dr_p \r + \tfrac 1 2 \l \p_i \u_i, |\nabla^{k+1} \dr|^2 \r - \l \Delta \dr \nabla^k \nabla \dr, \nabla^k \u \r \\
    \no & - \sum_{\substack{a+b=k,\\ 1 \leq a \leq k-1}} \l \nabla^{a+2} \dr \nabla^b \u + \nabla^{a+1} \dr \nabla^{b+1} \u, \nabla^k \nabla \dr \r - \sum_{\substack{a+b=k,\\ 1 \leq a \leq k-1}} \l \nabla^a \Delta \dr \nabla^b \nabla \dr, \nabla^k \u \r \\
    \no \lesssim & |\nabla^{k+1} \dr|^2_{L^2}|\nabla \u|_{L^\infty} + \sum_{\substack{a+b=k,\\ 1 \leq a \leq k-1}} |\nabla^{k+1} \dr|_{L^2} ( |\nabla^a \Delta \dr|_{L^2} |\nabla^b \u|_{L^\infty} + |\nabla^{a+1} \dr|_{L^4} |\nabla^{b+1} \u|_{L^4} ) \\
    \no & + |\nabla^{k+1} \dr|_{L^2} |\Delta \dr|_{L^4} |\nabla^k \u|_{L^4} + \sum_{\substack{a+b=k,\\ 1 \leq a \leq k-1}} |\nabla^a \Delta \dr|_{L^2} |\nabla^b \nabla \dr|_{L^4} |\nabla^k \u|_{L^4} \\
    \lesssim & |\nabla \dr|^2_{\dot{H^s}} |\nabla \u|_{H^s} \, ,
  \end{align}
by using the H\"older inequality and Sobolev embedding theory.

We now turn to estimate $I_4$. To do so, we first have
\begin{align*}
  I_4 = \l \nabla^k \Gamma, \dr \cdot \nabla^k \dot{\dr} \r + \sum_{\substack{a+b=k, \\ b \geq 1}} \l \nabla^a \Gamma \nabla^b \dr, \nabla^k \dot{\dr} \r \, .
\end{align*}
Then, keeping the structure of the Lagrangian $\Gamma$ in mind, we know that the estimate of the second terms on the right-hand side of the forgoing equality can be divided into three parts, for the first part
\begin{align*}\label{APE-I-41}
  & \sum_{\substack{a+b=k, \\ b \geq 1}} \l \nabla^a (-\rho |\dot{\dr}|^2) \nabla^b \dr, \nabla^k \dot{\dr} \r \leq \sum_{\substack{a+b+c+e = k, \\ e \geq 1}} \l |\nabla^a \rho| |\nabla^b \dot{\dr}| |\nabla^c \dot{\dr}| |\nabla^e \dr|, \nabla^k \dot{\dr} \r \\
  \lesssim & |\rho|_{L^\infty} \sum_{\substack{a+b+c=k, \\ c\geq 1}} \l |\nabla^a \dot{\dr}| |\nabla^b \dot{\dr}| |\nabla^c\dr|, |\nabla^k \dot{\dr}| \r + \sum_{\substack{a+b+c+e=k, \\ a, e\geq 1}} \l |\nabla^a \rho| |\nabla^b \dot{\dr}| |\nabla^c \dot{\dr}| |\nabla^e \dr|, |\nabla^k \dot{\dr}| \r \\
  \lesssim & |\rho|_{L^\infty} |\dot{\dr}|_{L^\infty} \!\sum_{\substack{a+b=k, \\ b \geq 1}} \!|\nabla^a \dot{\dr}|_{L^4} |\nabla^b \dr|_{L^4} |\nabla^k \dot{\dr}|_{L^2} \! + \! |\rho|_{L^\infty} \!\sum_{\substack{a+b+c=k, \\ a, b, c\geq 1}} \!|\nabla^a \dot{\dr}|_{L^\infty} |\nabla^b \dot{\dr}|_{L^4} |\nabla^c \dr|_{L^4} |\nabla^k \dot{\dr}|_{L^2} \\
  & + \sum_{\substack{a+b+c+e=k, \\ a, e\geq 1}} |\nabla^a \rho|_{L^4} |\nabla^b \dot{\dr}|_{L^\infty} |\nabla^c \dot{\dr}|_{L^\infty} |\nabla^e \dr|_{L^4} |\nabla^k \dot{\dr}|_{L^2}\\
  \lesssim & ( |\rho|_{L^\infty} + |\rho|_{\dot{H}^s} ) |\nabla \dr|_{H^s} |\dot{\dr}|^3_{H^s} \, .
\end{align*}
Similar as the above estimates, we can infer that
\begin{align*}
  & \sum_{\substack{a+b=k, \\ b \geq 1}} \l \nabla^a (|\nabla \dr|^2) \nabla^b \dr, \nabla^k \dot{\dr} \r \\
   = & - \l \div (|\nabla \dr|^2 \nabla^k \dr), \nabla^{k-1} \dot{\dr} \r + \sum_{\substack{a+b+c=k, \\ 1 \leq c \leq k-1}} \l \nabla^{a+1} \dr \nabla^{b+1} \dr \nabla^c \dr, \nabla^k \dot{\dr} \r \\
  \lesssim & |\nabla \dr|^2_{H^s} |\nabla \dr|_{\dot{H}^s} |\dot{\dr}|_{H^s}
\end{align*}
for the second part. Finally, by the same arguments in the above estimates one can deduce that
\begin{align*}
  & \sum_{\substack{a+b=k, \\ b \geq 1}} \l -\lambda_2 \nabla^a (\dr^\top \A \dr) \nabla^b \dr, \nabla^k \dot{\dr} \r \\
  \lesssim & |\lambda_2| \sum_{\substack{a+b+c=k, \\ c \geq 1}} \sum_{\substack{a_1+a_2=a, \\ a_1, a_2 \geq 1 }} \l |\nabla^{a_1} \dr| |\nabla^{a_2} \dr| |\nabla^{b+1} \u| |\nabla^c \dr|, |\nabla^k \dot{\dr}| \r \\
  & + 2 |\lambda_2| \sum_{\substack{a+b+c=k, \\ c \geq 1}} \l |\nabla^{a} \dr| |\nabla^{b+1} \u| |\nabla^c \dr|, |\nabla^k \dot{\dr}| \r + |\lambda_2| \sum_{\substack{a+b=k, \\ b \geq 1}} \l |\nabla^{a+1} \u| |\nabla^b \dr|, |\nabla^k \dot{\dr}| \r \\
  \lesssim & |\lambda_2| |\dot{\dr}|_{H^s} |\nabla \u|_{H^s} \sum_{l=1}^{3} |\nabla \dr|^l_{H^s} \, .
\end{align*}

With the above three estimates in hand, one has
\begin{align*}
  & \sum_{\substack{a+b=k, \\ b \geq 1}} \l \nabla^a \Gamma \nabla^b \dr, \nabla^k \dot{\dr} \r \\
  \lesssim & ( |\rho|_{L^\infty} + |\rho|_{\dot{H}^s} ) |\nabla \dr|_{H^s} |\dot{\dr}|^3_{H^s} + |\nabla \dr|^2_{H^s} |\nabla \dr|_{\dot{H}^s} |\dot{\dr}|_{H^s} + |\lambda_2| |\dot{\dr}|_{H^s} |\nabla \u|_{H^s} \sum_{l=1}^{3} |\nabla \dr|^l_{H^s} \, .
\end{align*}

Next we will control the term $\l \nabla^k \Gamma, \dr \cdot \nabla^k \dot{\dr} \r$. Noticing the fact that $\dr \cdot \dot{\dr} = 0$, one obtains
\begin{align*}
  \dr \cdot \nabla^k \dot{\dr} = \nabla^k( \dr \cdot \dot{\dr}) - \sum_{\substack{a+b=k,\\ a \geq 1}} \nabla^a \dr \nabla^b \dot{\dr} = - \sum_{\substack{a+b=k,\\ a \geq 1}} \nabla^a \dr \nabla^b \dot{\dr} \, .
\end{align*}
Then
\begin{align*}
  |\dr \cdot \nabla^k \dot{\dr}|_{L^2} \lesssim \sum_{\substack{a+b=k,\\ a \geq 1}} |\nabla^a \dr|_{L^4} |\nabla^b \dot{\dr}|_{L^4} \lesssim |\dot{\dr}|_{H^s} |\nabla \dr|_{H^s}\, .
\end{align*}
In addition,
  \begin{align*}
    |\nabla^k \Gamma|_{L^2} \lesssim & \sum_{a+b+c=k} |\nabla^a \rho \nabla^b \dot{\dr} \nabla^c \dot{\dr}|_{L^2} + \sum_{a+b=k} |\nabla^{a+1} \dr \nabla^{b+1} \dr|_{L^2} \\
    & + |\lambda_2| \sum_{a+b+c=k} |\nabla^a \dr^\top \nabla^{b+1} \u \nabla^c \dr|_{L^2} \\
    \lesssim & ( |\rho|_{L^\infty} + |\rho|_{\dot{H}^s} ) |\dot{\dr}|^2_{H^s} + |\nabla \dr|_{H^s} |\nabla \dr|_{\dot{H}^s} + |\lambda_2| \sum_{l=0}^{2} |\nabla \dr|^l_{H^s} |\nabla \u|_{H^s}\, .
  \end{align*}
Therefore, we have the estimate of $I_4$ as
\begin{equation}\label{APE-I-4}
  I_4 \lesssim ( |\rho|_{L^\infty} \!+\! |\rho|_{\dot{H}^s} ) |\nabla \dr|_{H^s} |\dot{\dr}|^3_{H^s} \!+\! |\nabla \dr|^2_{H^s} |\nabla \dr|_{\dot{H}^s} |\dot{\dr}|_{H^s} \!+\! |\lambda_2| \!\sum_{l=1}^{3}\! |\nabla \dr|^l_{H^s} |\dot{\dr}|_{H^s} |\nabla \u|_{H^s} \, .
\end{equation}

For the terms of $I_5$, we anticipate that they can give us some dissipation similar as the derivation of the basic energy law. Actually, $I_5$ can be rewritten as the endpoint derivative terms part $I^e_5$ and the intermediate terms part $I^m_5$, and the former can give us some dissipation if the coefficients satisfy some relations. That is
\begin{align*}
  I_5 = I^e_5 + I^m_5 \, ,
\end{align*}
where
  \begin{equation*}
    \begin{aligned}
      I_5^{e} & = - \mu_1 \l  \dr_p \dr_q \nabla^k \A_{pq} \dr_i \dr_j, \nabla^k \partial_j \u_i \r - \l \mu_2 \dr_j \nabla^k \dot{\dr}_i + \mu_3 \dr_i \nabla^k \dot{\dr}_j, \nabla^k \partial_j \u_i \r \\
      & - \l \mu_2 \dr_j \nabla^k \B_{pi} \dr_p + \mu_3 \dr_i \nabla^k \B_{pj} \dr_p , \nabla^k \partial_j \u_i \r - \l \mu_5 \dr_j \dr_p \nabla^k \A_{pi} + \mu_6 \dr_i \dr_p \nabla^k \A_{pj}, \nabla^k \partial_j \u_i \r \\
      & + \lambda_1 |\nabla^k \dot{\dr}|^2_{L^2} + \lambda_1 \l (\nabla^k \B ) \dr, \nabla^k \dot{\dr} \r + \lambda_2 \l (\nabla^k \A ) \dr, \nabla^k \dot{\dr} \r
    \end{aligned}
  \end{equation*}
and
  \begin{align*}
    I^m_5 = & - \mu_1 \sum_{\substack{a+b=k,\\ a \geq 1}} \l \nabla^a ( \dr_i \dr_j \dr_p \dr_q ) \nabla^b \A_{pq}, \nabla^k \p_j \u_i \r \\
    & - \sum_{\substack{a+b=k,\\ a \geq 1}} \l \mu_2 \nabla^a \dr_j \nabla^b \dot{\dr}_i + \mu_3 \nabla^a \dr_i \nabla^b \dot{\dr}_j, \nabla^k \p_j \u_i \r \\
    & - \sum_{\substack{a+b=k,\\ a \geq 1}} \l \mu_2 \nabla^a (\dr_j \dr_p) \nabla^b \B_{pi} + \mu_3 \nabla^a (\dr_i \dr_p) \nabla^b \B_{pj}, \nabla^k \p_j \u_i \r \\
    & - \sum_{\substack{a+b=k,\\ a \geq 1}} \l \mu_5 \nabla^a (\dr_j \dr_p) \nabla^b \A_{pi} + \mu_6 \nabla^a (\dr_i \dr_p) \nabla^b \A_{pj}, \nabla^k \p_j \u_i \r \\
    & + \sum_{\substack{a+b=k,\\ b \geq 1}} \l \lambda_1 \nabla^a \B \nabla^b \dr + \lambda_2 \nabla^a \A \nabla^b \dr, \nabla^k \dot{\dr} \r \\
    \equiv & I^m_{51} + I^m_{52} + I^m_{53} + I^m_{54} + I^m_{55} \, .
  \end{align*}

Similar as the derivation of the basic energy law, one can get the dissipation as
  \begin{align*}
    I^e_5 = - \mu_1 |\dr^\top (\nabla^k \A) \dr|^2_{L^2} \!+\! \lambda_1 |\nabla^k \dot{\dr} \!+\! (\nabla^k \B) \dr \!+\! \tfrac{\lambda_2}{\lambda_1} (\nabla^k \A) \dr |^2_{L^2} \!-\! (\mu_5 \!+\! \mu_6 \!+\! \tfrac{\lambda_2^2}{\lambda_1}) |(\nabla^k \A) \dr|^2_{L^2} \, .
  \end{align*}

Next we will estimate the intermediate terms one by one, which should be controlled by the free energy or/and dissipation energy. Thanks to H\"older inequality, Sobolev embedding theory and $|\dr| = 1$, one arrives at
  \begin{align*}
    I^m_{51} & \lesssim \mu_1 \sum_{\substack{a+b=k,\\ a \geq 1}} \sum_{\substack{a_1 + a_2 +a_3 + a_4 = a, \\ a_1, a_2, a_3, a_4 \geq 1}} \l |\nabla^{a_1} \dr| |\nabla^{a_2} \dr| |\nabla^{a_3} \dr| |\nabla^{a_4} \dr| |\nabla^{b+1} \u|, |\nabla^{k+1} \u| \r \\
    & + \mu_1 \sum_{\substack{a+b=k, \\ a \geq 1}} \sum_{\substack{a_1 +a_2 + a_3 = a, \\ a_1, a_2, a_3 \geq 1}} \l |\nabla^{a_1} \dr| |\nabla^{a_2} \dr| |\nabla^{a_3} \dr| |\nabla^{b+1} \u|, |\nabla^{k+1} \u| \r \\
    & + \mu_1 \! \sum_{\substack{a+b=k, \\ a \geq 1}} \sum_{\substack{a_1 +a_2 = a, \\ a_1, a_2 \geq 1}} \! \l |\nabla^{a_1} \dr| |\nabla^{a_2} \dr| |\nabla^{b+1} \u|, \nabla^{k+1} \u \r + \mu_1 \! \sum_{\substack{a+b=k, \\ a \geq 1}} \! \l |\nabla^{a} \dr| |\nabla^{b+1} \u|, |\nabla^{k+1} \u| \r \\
    \lesssim & \mu_1 |\u|_{\dot{H}^s} |\nabla \u|_{\dot{H}^s} \sum_{l=1}^{4} |\nabla \dr|^l_{H^s}
  \end{align*}
and
\begin{equation*}\label{APE-I-m-2}
  \begin{aligned}
    I^m_{52} \lesssim & (|\mu_2| + |\mu_3|) \sum_{\substack{a+b=k, \\ a \geq 1}}  |\nabla^a \dr|_{L^4} |\nabla^b \dot{\dr}|_{L^4} |\nabla^{k+1} \u|_{L^2} \\
    \lesssim & (|\mu_2| + |\mu_3|) |\nabla \u|_{\dot{H}^s} |\nabla \dr|_{H^s} |\dot{\dr}|_{H^s} \, ,
  \end{aligned}
\end{equation*}
and for the control of $I^m_{53}$ and $I^m_{54}$, using the same calculation of $I^m_{51}$ to yield
\begin{equation*}\label{APE-I-m-3}
  \begin{aligned}
    I^m_{53} \lesssim (|\mu_2| + |\mu_3|) |\u|_{\dot{H}^s} |\nabla \u|_{\dot{H}^s} (|\nabla \dr|_{H^s} + |\nabla \dr|^2_{H^s})
  \end{aligned}
\end{equation*}
and
\begin{equation*}\label{APE-I-m-4}
  \begin{aligned}
    I^m_{54} \lesssim (|\mu_5| + |\mu_6|) |\u|_{\dot{H}^s} |\nabla \u|_{\dot{H}^s} (|\nabla \dr|_{H^s} + |\nabla \dr|^2_{H^s}) \, .
  \end{aligned}
\end{equation*}
Meanwhile, it is easy to get
\begin{equation*}\label{APE-I-m-5}
  \begin{aligned}
    I^m_{55} \lesssim & (|\lambda_1| \!+\! |\lambda_2|)( |\nabla^k \u|_{L^2} |\nabla \dr|_{L^\infty} \!+\! |\nabla \u|_{L^\infty} |\nabla^k \dr|_{L^2} \!+\! \sum_{\substack{a+b=k, \\ 2 \leq b \leq k-1}} \! |\nabla^{a+1} \u|_{L^4} |\nabla^b \dr|_{L^4} )|\nabla^k \dot{\dr}|_{L^2} \\
    \lesssim & (|\lambda_1| + |\lambda_2|) |\u|_{\dot{H}^s} |\nabla \dr|_{H^s} |\dot{\dr}|_{H^s} \, .
  \end{aligned}
\end{equation*}
Thus, we have the following estimate
  \begin{align} \label{APE-I-m}
    I^m_5 \lesssim \mu \big[ \sum_{l=1}^{4} |\nabla \dr|^l_{H^s} ( |\u|_{\dot{H}^s} + |\dot{\dr}|_{H^s} ) |\nabla \u|_{\dot{H}^s} + |\u|_{\dot{H}^s} |\nabla \dr|_{H^s} |\dot{\dr}|_{H^s} \big] \ \, .
  \end{align}
where $\mu = |\mu_1| \!+\! |\mu_2| \! +\! |\mu_3| \!+\! |\mu_4| \!+\! |\mu_5| \!+\! |\mu_6|$. As a result, combining with the above estimates of $ I_l ( 1 \leq l \leq 5 ) $, we arrive at the following estimate:
\begin{align}\label{APE-I}
     \no I \leq & - \mu_1 |\dr^\top (\nabla^k \A) \dr|^2_{L^2} + \lambda_1 |\nabla^k \dot{\dr} + (\nabla^k \B) \dr + \tfrac{\lambda_2}{\lambda_1} (\nabla^k \A) \dr |^2_{L^2} \\
     & - (\mu_5 + \mu_6 + \tfrac{\lambda_2^2}{\lambda_1}) |(\nabla^k \A) \dr|^2_{L^2} + C I^{'}\, ,
\end{align}
where
\begin{align*}
  I^{'} = & \Q (\u) |\rho|^2_{\dot{H}^s} ( 1 + |\rho|_{\dot{H}^s} ) |\u|_{H^s} + |\nabla \dr|^2_{\dot{H}^s} |\nabla \u|_{H^s} + ( |\rho|_{L^\infty} + |\rho|_{\dot{H}^s} ) |\nabla \dr|_{H^s} |\dot{\dr}|^3_{H^s} \\[1.5mm]
  & + |\nabla \dr|^2_{H^s} |\nabla \dr|_{\dot{H}^s} |\dot{\dr}|_{H^s} + \mu \sum_{l=1}^{4} |\nabla \dr|^l_{H^s} ( |\u|_{\dot{H}^s} + |\dot{\dr}|_{H^s} ) |\nabla \u|_{H^s} \, .
\end{align*}

We now turn to deal with $J$ term by term. By using the H\"older inequality, Sobolev embedding theory and the Moser-type inequality \eqref{Moser-type}, we can infer that
  \begin{align}\label{APE-J-1}
    J_1 \lesssim |[\nabla^k, \u \cdot \nabla] \u|_{L^2} |\rho \nabla^k \u|_{L^2}
    \lesssim |\nabla \u|_{L^\infty} |\nabla^k \u|_{L^2} |\rho|_{L^\infty} |\nabla^k \u|_{L^2}
    \lesssim |\rho|_{L^\infty} |\u|^3_{\dot{H}^s}
  \end{align}
and
  \begin{align}\label{APE-J-2}
    J_2 \lesssim & \big( |\nabla ( \tfrac{p'(\rho)}{\rho} )|_{L^\infty} |\nabla^k \rho|_{L^2} \!+\! |\nabla^k ( \tfrac{p'(\rho)}{\rho} )|_{L^2} |\nabla \rho|_{L^\infty} \big) |\rho|_{L^\infty} |\nabla^k \u|_{L^2}
    \!\lesssim\! \Q (\u) |\rho|_{\dot{H}^s} \P_s (|\rho|_{\dot{H}^s}) |\u|_{\dot{H}^s}
  \end{align}
and
  \begin{align}\label{APE-J-3}
    J_3 \lesssim ( |\nabla \u|_{L^\infty} |\nabla^k \dot{\dr}|_{L^2} + |\nabla^k \u|_{L^2} |\nabla \dot{\dr}|_{L^\infty} ) |\rho|_{L^\infty} |\nabla^k \dot{\dr}|_{L^2}
    \lesssim |\rho|_{L^\infty} |\u|_{\dot{H}^s} |\dot{\dr}|^2_{H^s} \, .
  \end{align}

Similarly, one has
  \begin{align}\label{APE-J-4}
    \no J_4 \lesssim & \big( |\nabla (\tfrac{1}{\rho})|_{L^\infty} |\nabla^{k+1} \dr|_{L^2} + |\nabla^k (\tfrac{1}{\rho})|_{L^2} |\Delta \dr|_{L^\infty} \big) |\rho|_{L^\infty} |\nabla^k \dot{\dr}|_{L^2} \\
    \lesssim & \Q(\u) \P_k(|\rho|_{\dot{H}^s}) |\nabla \dr|_{\dot{H}^s} |\dot{\dr}|_{H^s} \, .
  \end{align}

And it is easy to deduce that $J_5$ and $J_6$ have the same estimate:
  \begin{align}\label{APE-J-5}
    J_5 \lesssim & \Q(\u) |\nabla^k \rho|_{L^2} ( |\nabla \u|_{L^\infty} |\nabla^k \rho|_{L^2} + |\nabla^k \u|_{L^2} |\nabla \rho|_{L^\infty} )
   \lesssim \Q (\u) |\rho|^2_{\dot{H}^s} |\u|_{\dot{H}^s}
  \end{align}
and
  \begin{align}\label{APE-J-6}
    J_6 \lesssim \Q(\u) |\nabla^k \rho|_{L^2} ( |\nabla \rho|_{L^\infty} |\nabla^k \u|_{L^2} + |\nabla^k \rho|_{L^2} |\div \u|_{L^\infty} )
    \lesssim \Q (\u) |\rho|^2_{\dot{H}^s} |\u|_{\dot{H}^s} \, .
  \end{align}

For the estimate of $J_7$, using the first equation of system \eqref{compressible-Liquid-Crystal-Model} infers that
  \begin{align}\label{APE-J-7}
    J_7 \lesssim & |(\tfrac{p'(\rho)}{\rho})'|_{L^\infty} |\div (\rho \u)|_{L^\infty} |\nabla^k \rho|^2_{L^2}
    \lesssim \Q (\u) (|\rho|_{L^\infty} + | \rho|_ {\dot{H}^s} ) |\u|_{H^s} |\rho|^2_{\dot{H}^s} \, .
  \end{align}

Based on the observation of the structure of $\Sigma_1$, we only need to estimate the one part $ | [\nabla^k, \tfrac{1}{\rho} \div] \nabla \u |_{L^2} $ for the estimate of $J_8$. It is easy to obtain
\begin{align*}
  | [\nabla^k, (\tfrac{1}{\rho}) \div] \nabla \u |_{L^2}
  \lesssim |\nabla (\tfrac{1}{\rho})|_{L^\infty} |\nabla^{k+1} \u|_{L^2} + |\nabla^k (\tfrac{1}{\rho})|_{L^2} |\Delta \u|_{L^\infty}
  \lesssim \Q (\u) \P_s (|\rho|_{\dot{H}^s}) |\nabla \u|_{\dot{H}^s} \, ,
\end{align*}
so we have
\begin{equation}\label{APE-J-8}
  \begin{aligned}
    J_8 \lesssim (\mu_4 + \xi) \Q (\u) \P_s (|\rho|_{\dot{H}^s}) |\nabla \u|_{\dot{H}^s} |\u|_{\dot{H}^s} \, .
  \end{aligned}
\end{equation}

We now turn to estimate $J_9$, we should first get the estimate of $ | [\nabla^k, (\tfrac{1}{\rho}) \div] ( \nabla \dr \odot \nabla \dr) |_{L^2} $, it infers that
\begin{align*}
  & |[\nabla^k, (\tfrac{1}{\rho}) \div] ( \nabla \dr \odot \nabla \dr) |_{L^2}
  \lesssim |\nabla (\tfrac{1}{\rho})|_{L^\infty} |\nabla^k ( \nabla \dr \odot \nabla \dr )|_{L^2} + |\nabla^k (\tfrac{1}{\rho})|_{L^2} |\div ( \nabla \dr \odot \nabla \dr )|_{L^\infty} \\
  & \lesssim \Q (\u) \P_s (|\rho|_{\dot{H}^s}) \big( |\nabla^{k+1} \dr|_{L^2} |\nabla \dr|_{L^\infty} + \sum_{\substack{a+b=k,\\a,b \geq 1}} |\nabla^{a+1} \dr|_{L^4} |\nabla^{b+1} \dr|_{L^4} + |\nabla \dr|_{L^\infty} |\nabla^2 \dr|_{L^2} \big) \\
  & \lesssim \Q (\u) \P_s (|\rho|_{\dot{H}^s}) |\nabla \dr|_{H^s} |\nabla \dr|_{\dot{H}^s} \, ,
\end{align*}
hence we obtain
\begin{equation}\label{APE-J-9}
  \begin{aligned}
    J_9 \lesssim \Q (\u) \P_k (|\rho|_{\dot{H}^s}) |\u|_{\dot{H}^s} |\nabla \dr|_{H^s} |\nabla \dr|_{\dot{H}^s} \, .
  \end{aligned}
\end{equation}

By the analysis of the construction of $J_{10}$, we only need to estimate the following three parts:
\begin{align*}
  & | [\nabla^k, (\tfrac 1 \rho) \div] (\dr_i \dr_j \dr_p \dr_q \A_{pq}) |_{L^2} \\
  \lesssim & |\nabla^k (\tfrac{1}{\rho})|_{L^2} |\div (\dr_i \dr_j \dr_p \dr_q \A_{pq} )|_{L^\infty} + |\nabla (\tfrac 1 \rho)|_{L^\infty} |\nabla^{k-1} \div (\dr_i \dr_j \dr_p \dr_q \A_{pq})|_{L^2} \\
  \lesssim & \Q (\u) \P_k (|\rho|_{\dot{H}^s}) |\nabla \u|_{H^s} \sum_{l=0}^{4} |\nabla \dr|^l_{H^s} \, ,
\end{align*}
and
\begin{align*}
   | [\nabla^k, (\tfrac 1 \rho) \div] (\dr_j \dot{\dr}_i) |_{L^2}
   \lesssim &  | \nabla^k (\tfrac{1}{\rho}) |_{L^2} |\div (\dr \dot{\dr})|_{L^\infty} + |\nabla (\tfrac{1}{\rho})|_{L^\infty} |\nabla^{k-1} \div (\dr \dot{\dr})|_{L^2} \\
   \lesssim & \Q (\u) \P_k (|\rho|_{\dot{H}^s}) |\dot{\dr}|_{H^s} (1 + |\nabla \dr|_{H^s})
\end{align*}
and
\begin{align*}
  |[\nabla^k, (\tfrac{1}{\rho}) \div] (\dr_j \B_{pi} \dr_i)|_{L^2}
  \lesssim & |\nabla^k (\tfrac{1}{\rho})|_{L^2} |\div (\dr_j \B_{pi} \dr_i)|_{L^\infty} + |\nabla (\tfrac{1}{\rho})|_{L^\infty} |\nabla^{k-1} \div (\dr_j \B_{pi} \dr_i)|_{L^2} \\
  \lesssim & \Q (\u) \P_k (|\rho|_{\dot{H}^s}) |\nabla \u|_{H^s} \sum_{l=0}^{2} |\nabla \dr|^l_{H^s} \, .
\end{align*}
As a result, one has
\begin{equation}\label{APE-J-10}
  \begin{aligned}
    J_{10} \lesssim & \mu \Q (\u) \P_k (|\rho|_{\dot{H}^s}) |\u|_{\dot{H}^s} ( |\nabla \u|_{H^s} + |\dot{\dr}|_{H^s} ) \sum_{l=0}^{4} |\nabla \dr|^l_{H^s} \, .
  \end{aligned}
\end{equation}

As for the estimate of $J_{11}$, from the representation of Lagrangian $\Gamma$ we can divide it into three parts
\begin{align*}
  & \sum_{\substack{a+b=k,\\ a \geq 1}}|\nabla^a (\tfrac{1}{\rho}) \nabla^b (\rho |\dot{\dr}|^2 \dr)|_{L^2} \\
  \lesssim & |\nabla^k (\tfrac{1}{\rho})|_{L^2} |\rho|_{L^\infty} |\dot{\dr}|^2_{L^\infty} + \sum_{\substack{a+b+c+e = k, \\ 1 \leq a \leq k-1}} \sum_{c_1 + c_2 =c}|\nabla^a (\tfrac{1}{\rho}) \nabla^b \rho \nabla^{c_1} \dot{\dr} \nabla^{c_2} \dot{\dr} \nabla^e \dr|_{L^2} \\
  \lesssim & \Q (\u) \P_k (|\rho|_{\dot{H}^s}) |\dot{\dr}|^2_{H^s} (1 + |\nabla \dr|_{H^s})
\end{align*}
and
\begin{align*}
  & \sum_{\substack{a+b=k,\\ a \geq 1}} |\nabla^a (\tfrac{1}{\rho}) \nabla^b ( |\nabla \dr|^2 \dr)|_{L^2} \\
  \lesssim & |\nabla^k (\tfrac{1}{\rho})|_{L^2} |\nabla \dr|^2_{L^\infty} + \sum_{\substack{a+b+c = k, \\ 1 \leq a \leq k-1}} \sum_{b_1 + b_2 =b}|\nabla^a (\tfrac{1}{\rho}) \nabla^{b_1+1} \dr \nabla^{b_2+1} \dr \nabla^c \dr|_{L^2} \\
  \lesssim & \Q (\u) \P_k (|\rho|_{\dot{H}^s}) |\nabla \dr|^2_{H^s} (1 + |\nabla \dr|_{H^s})
\end{align*}
and similarly
\begin{align*}
  & \sum_{\substack{a+b=k,\\ a \geq 1}} |\nabla^a (\tfrac{1}{\rho}) \nabla^b ( (\dr^\top \A \dr) \dr)|_{L^2} \\
  \lesssim & |\nabla^k (\tfrac{1}{\rho})|_{L^2} | (\dr^\top \A \dr) \dr |_{L^\infty} + \sum_{\substack{a+b+c = k, \\ 1 \leq a \leq k-1}} \sum_{c_1 + c_2 + c_3 = c}|\nabla^a (\tfrac{1}{\rho}) \nabla^{b+1} \u \nabla^{c_1} \dr \nabla^{c_2} \dr \nabla^{c_3} \dr |_{L^2} \\
  \lesssim & \Q (\u) \P_k (|\rho|_{\dot{H}^s}) |\u|_{\dot{H}^s} \sum_{l=0}^{3} |\nabla \dr|^l_{H^s} \, .
\end{align*}
Based on the above three estimates, we have that
  \begin{align}\label{APE-J-11}
    J_{11} \lesssim \Q (\u) \P_k (|\rho|_{\dot{H}^s}) |\u|_{\dot{H}^s} \big\{ ( |\dot{\dr}|^2_{H^s} \!+\! |\nabla \dr|^2_{H^s} ) (1 \!+\! |\nabla \dr|_{H^s}) \!+\!  |\lambda_2| |\u|_{\dot{H}^s} \sum_{l=0}^{3} |\nabla \dr|^l_{H^s} \big\}  \, .
  \end{align}

Now we turn to control the last two terms of $J$, it is easy to deduce that
\begin{align*}
  \sum_{\substack{a+b=k, \\ a \geq 1 }} |\nabla^a (\tfrac{1}{\rho}) \nabla^b \dot{\dr}|_{L^2}
  \lesssim & |\nabla^k (\tfrac{1}{\rho})|_{L^2} |\dot{\dr}|_{L^\infty} + \sum_{\substack{a+b=k, \\ 1 \leq a \leq k-1 }} |\nabla^a (\tfrac{1}{\rho})|_{L^4} |\nabla^b \dot{\dr}|_{L^4} \\
  \lesssim & \Q (\u) \P_k (|\rho|_{\dot{H}^s}) |\dot{\dr}|_{H^s} \, ,
\end{align*}
and
\begin{align*}
  \sum_{\substack{a+b=k, \\ a \geq 1 }} |\nabla^a (\tfrac{1}{\rho}) \nabla^b (\B \dr)|_{L^2}
  \lesssim & |\nabla^k (\tfrac{1}{\rho})|_{L^2} |\B \dr|_{L^\infty} + \sum_{\substack{a+b+c=k, \\ 1 \leq a \leq k-1 }} | \nabla^a (\tfrac{1}{\rho}) \nabla^{b+1} \u \nabla^c \dr|_{L^2} \\
  \lesssim & \Q (\u) \P_k (|\rho|_{\dot{H}^s}) |\u|_{\dot{H}^s} ( 1 + |\nabla \dr|_{H^s} ) \, ,
\end{align*}
then we have
\begin{align}\label{APE-J-12}
  J_{12} \lesssim |\lambda_1| \Q (\u) \P_k (|\rho|_{\dot{H}^s}) |\dot{\dr}|^2_{H^s}
\end{align}
and
\begin{align}\label{APE-J-13}
  J_{13} \lesssim ( |\lambda_1| + |\lambda_2| ) \Q (\u) \P_k (|\rho|_{\dot{H}^s}) |\u|_{\dot{H}^s} |\dot{\dr}|_{H^s} ( 1 + |\nabla \dr|_{H^s} ) \, .
\end{align}

Therefore, the combination of the estimates of $J_l (1 \leq l \leq 13)$ leads to
  \begin{align}\label{APE-J}
    \no J \lesssim & \Q(\u) |\u|_{\dot{H}^s} ( |\u|^2_{\dot{H}^s} \!+\! |\dot{\dr}|^2_{H^s} ) \!+\! \Q(\u) \P_s(|\rho|_{\dot{H}^s}) ( |\rho|_{\dot{H}^s} |\u|_{\dot{H}^s} \!+\! |\nabla \dr|_{\dot{H}^s} |\dot{\dr}|_{H^s}) \\
    \no + & \Q (\u) ( |\rho|_{L^\infty} + |\rho|_{\dot{H}^s} ) |\rho|^2_{\dot{H}^s} |\u|_{H^s} + (\mu_4 + \xi) \Q (\u) \P_s (|\rho|_{\dot{H}^s}) |\nabla \u|_{\dot{H}^s}  |\u|_{\dot{H}^s} \\
    \no + & \Q (\u) \P_s (|\rho|_{\dot{H}^s}) |\u|_{\dot{H}^s} (|\nabla \dr|_{\dot{H}^s} \!+\!  \mu |\nabla \u|_{H^s} \!+\!  \mu |\dot{\dr}|_{H^s} ) \!\sum_{l=0}^{4}\! |\nabla \dr|^l_{H^s} \!+\! |\lambda_1| \Q (\u) \P_s (|\rho|_{\dot{H}^s}) |\dot{\dr}|^2_{H^s}\\
    + & \Q (\u) \P_s (|\rho|_{\dot{H}^s}) |\u|_{\dot{H}^s} \big\{ ( |\dot{\dr}|^2_{H^s} \!+\! |\nabla \dr|^2_{H^s} ) (1 \!+\! |\nabla \dr|_{H^s}) \!+\! |\lambda_2| |\u|_{\dot{H}^s} \!\sum_{l=0}^{3}\! |\nabla \dr|^l_{H^s} \big\} \, .
  \end{align}

Substituting the estimate \eqref{APE-I} and \eqref{APE-J} into \eqref{APE-1} and summing up for all $1 \leq k \leq s$, then combining the result with the basic energy law \eqref{Basic-Ener-Law} infers that
  \begin{align}\label{APE}
    \no & \tfrac{1}{2} \tfrac{\d}{\d t} \Big( \mathcal{N}_s (\rho) + | \u |^2_{\rho, H^s} + | \dot{\dr} |^2_{\rho, H^s} + |\nabla \dr|^2_{H^s} \Big) + \Big( \tfrac{1}{2} \mu_4 |\nabla \u|^2_{H^s} + ( \tfrac{1}{2} \mu_4 + \xi ) |\div \u|^2_{H^s} \Big) \\[1.5mm]
    \no + & \!\mu_1 \!\sum_{k=0}^{s}\! |\dr^\top \!(\nabla^k \!\A) \dr|^2_{L^2} \!-\! \lambda_1 \!\sum_{k=0}^{s}\! |\nabla^k \dot{\dr} \!+\! (\nabla^k \! \B) \dr \!+\! \tfrac{\lambda_2}{\lambda_1} (\nabla^k \! \A) \dr |^2_{L^2} \!+\! (\mu_5 \!+\! \mu_6 \!+\! \tfrac{\lambda_2^2}{\lambda_1}) \!\sum_{k=0}^{s}\! |(\nabla^k \! \A) \dr|^2_{L^2}\\
    \no \lesssim & \Q (\u) |\rho|^2_{\dot{H}^s} ( 1 + |\rho|_{\dot{H}^s} ) |\u|_{H^s} + ( |\rho|_{L^\infty} + |\rho|_{\dot{H}^s} ) |\nabla \dr|_{H^s} |\dot{\dr}|^3_{H^s} + |\lambda_1| \Q (\u) \P_s (|\rho|_{\dot{H}^s}) |\dot{\dr}|^2_{H^s} \\[1.5mm]
    \no + & |\nabla \dr|^2_{\dot{H}^s} |\nabla \u|_{H^s} + |\nabla \dr|^2_{H^s} |\nabla \dr|_{\dot{H}^s} |\dot{\dr}|_{H^s} + \mu \sum_{l=1}^{4} |\nabla \dr|^l_{H^s} ( |\u|_{\dot{H}^s} + |\dot{\dr}|_{H^s} ) |\nabla \u|_{H^s}\\
    \no + & \Q(\u) |\u|_{\dot{H}^s} ( |\u|^2_{\dot{H}^s} \!+\! |\dot{\dr}|^2_{H^s} ) \!+\! \Q(\u) \P_s(|\rho|_{\dot{H}^s}) ( |\rho|_{\dot{H}^s} |\u|_{\dot{H}^s} \!+\! |\nabla \dr|_{\dot{H}^s} |\dot{\dr}|_{H^s}) \\[1.5mm]
    \no + & \Q (\u) \P_s (|\rho|_{\dot{H}^s}) |\u|_{\dot{H}^s} (|\nabla \dr|_{\dot{H}^s} +  (\mu + \xi) |\nabla \u|_{H^s} +  \mu |\dot{\dr}|_{H^s} ) \sum_{l=0}^{4} |\nabla \dr|^l_{H^s} \\
    + & \Q (\u) \P_s (|\rho|_{\dot{H}^s}) |\u|_{\dot{H}^s} \!\big\{ ( |\dot{\dr}|^2_{H^s} \!+\! |\nabla \dr|^2_{H^s} ) (1 \!+\! |\nabla \dr|_{H^s}) \!+\! |\lambda_2| |\u|_{\dot{H}^s} \!\sum_{l=0}^{3}\! |\nabla \dr|^l_{H^s} \big\} \, .
  \end{align}

By Lemma \ref{rho-bound-Lemma}, one can infer that
\begin{align*}
  |\tfrac{\rho}{p'(\rho)}|_{L^\infty} \leq C (\underline{\rho}, \overline{\rho}, \gamma) \exp\Big\{ |\gamma-2| \int_{0}^{t} |\div \u|_{L^\infty} (\tau) \d \tau \Big\} \, ,
\end{align*}
then we have
\begin{align*}
  |\rho|^2_{\dot{H}^s} \leq & C (\underline{\rho}, \overline{\rho}, \gamma) \exp \Big\{ |\gamma-2| \int_{0}^{t} |\div \u|_{L^\infty} (\tau) \d \tau \Big\} \sum_{k=1}^{s}\int_{\mathbb{R}^N} \tfrac{p'(\rho)}{\rho} |\nabla^k \rho|^2 \d x \\
  \lesssim & \Q (\u) E(t) \, .
\end{align*}
Similarly,
\begin{align*}
  |\u|_{H^s}^2 \lesssim \Q (\u) E(t)\, , \ \ |\dot{\dr}|_{H^s}^2 \lesssim \Q (\u) E(t)\, .
\end{align*}
Meanwhile, noticing the fact that
\begin{align*}
  \P_s ( |\rho|_{\dot{H}^s} ) \lesssim \Q (\u) \sum_{l=1}^{s} E^{\frac{l}{2}} (t) \, ,
\end{align*}
then combining with inequality \eqref{APE}, we can deduce that
\begin{align}\label{Energy-Estimate}
   & \tfrac{1}{2} \tfrac{\d}{\d t} E(t) + D(t)
   \lesssim \Q (\u) (1+\mu) \sum_{l=2}^{s+6} E^{\frac{l}{2}} (t) + \tfrac{1}{\sqrt{\mu_4}} D^{\frac{1}{2}} (t) E (t) + \tfrac{\mu + \xi}{\sqrt{\mu_4}} \Q (\u) D^{\frac{1}{2}} (t) \sum_{l=1}^{s+5} E^{\frac{l}{2}} (t) \, .
\end{align}

Then, the inequality \eqref{A-Priori-Estimate-0} holds for all $t \in [0,T]$, and this completes the proof of Lemma \ref{A-Priori-Estimate}.
\end{proof}

\section{well-posedness of system \eqref{compressible-Liquid-Crystal-Model} for a given  $(\rho, \u)$} \label{Sect-4}

In this section, we focus on the well-posedness of the following hyperbolic system
\begin{equation}\label{rho-u-given}
  \begin{aligned}
  \left \{ \begin{array} {c}
  \rho \ddot{\dr} = \Delta \dr + \Gamma (\rho, \u ,\dr, \dot{\dr}) \dr + \lambda_1 (\dot{\dr} + \B \dr) + \lambda_2 \A \dr \, , \\[1.5mm]
  \dr \in \mathbb{S}^{N-1} \, ,
   \end{array}
   \right.
  \end{aligned}
\end{equation}
with the initial conditions
\begin{equation}\label{rho-u-given-Ini}
  \dr(x,0) = \dr^{in} (x) \in \mathbb{S}^{N-1},\ \dot{\dr}(x,0) = \tilde{\dr}^{in}(x) \in \R^N \, ,
\end{equation}
where the density $\rho \in \mathbb{R}$ and bulk velocity $\u \in \mathbb{R}^N$ are given satisfying the relation $\p_t \rho + \div(\rho \u) =0$, and $\dot{\dr}$ is the first order material derivative of the vector field $\dr$, the Lagrangian multiplier $\Gamma(\rho, \u, \dr, \dot{\dr})$ is of the form \eqref{Lagrangian}, and the initial data satisfy the compatibility
$$\dr^{in} \cdot \tilde{\dr}^{in} = 0\, .$$

Before getting the well-posedness of the system \eqref{rho-u-given}-\eqref{rho-u-given-Ini}, we should give the following lemma which is concerned with the relation between the Lagrangian multiplier $\Gamma$ and the geometric constraint $|\dr|=1$ and play an important role in our analysis.

\begin{lemma} \label{lemma-d}
Let $s>\tfrac{N}{2}+1$ and $T>0$. Assume the given function $\rho, \u$ in the equation \eqref{rho-u-given} satisfy $\rho, \tfrac{1}{\rho} \in L^1 (0,T; L^\infty)$, $\u \in L^1 (0,T; W^{1,\infty})$ and $\dr$ is a classical solution to \eqref{rho-u-given}-\eqref{rho-u-given-Ini} satisfying $\dot{\dr} \in L^\infty (0,T; H^s)$, $\nabla \dr \in L^\infty (0,T; H^s_{\rho^{-1/2}})$.

If the constraint $|\dr|=1$ is required, then the Lagrangian multiplier $\Gamma$ is
  \begin{equation}\label{gamma}
    \Gamma = - \rho |\dot{\dr}|^2 + |\nabla \dr|^2 - \lambda_2 \dr^\top \A \dr \, .
  \end{equation}

  Conversely, if $\Gamma$ is of the form \eqref{gamma} and the initial data $\dr^{in}$ and $\tilde{\dr}^{in}$ are such that $\tilde{\dr}^{in} \cdot \dr^{in} = 0, |\dr^{in}|=1$, then $|\dr|=1$ holds for all $t \in [0,T]$.
\end{lemma}
%
%
\begin{proof}
  We can proof this lemma similar as Lemma 4.1 of the reference \cite{Jiang-Luo-arXiv-2017}. Here we omit the detail of the proof for simplicity.
\end{proof}


Then we state the well-posedness of the system \eqref{rho-u-given}-\eqref{rho-u-given-Ini} in the following proposition. Here we point out that the proposition will play an important role in constructing the iterating approximate system of the system \eqref{compressible-Liquid-Crystal-Model}-\eqref{Initial-Data}.

\begin{proposition}\label{Proposition-1}
  Let $s>\tfrac{N}{2}+1$ and $T_0>0$. Assume the given functions $\rho, \u$ in \eqref{rho-u-given} satisfy $\u \in L^1 (0,T_0; H^{s+1})$, $\rho \in L^\infty (0,T_0; \dot{H}^s)$ and $\underline{\rho} \leq \rho (x,0) \leq \bar{\rho}$ for some $\underline{\rho}, \bar{\rho} >0$.

  If $\tilde{\dr}^{in} \in H^s$, $\nabla \dr^{in} \in H^s_{\rho(\cdot, 0)^{-1/2}}$, then there exist $0 < T \leq T_0$ and $C_2>0$, depending only on $\dr^{in}, \tilde{\dr}^{in}, \rho$ and $\u$, such that \eqref{rho-u-given}-\eqref{rho-u-given-Ini} admits the unique classical solution $\dr$ satisfying $\dot{\dr} \in C (0,T; H^s)$ and $ \nabla \dr \in C(0,T; H^s_{\rho^{-1/2}})$. Moreover, the following energy bound holds:
  \begin{align*}
    |\dot{\dr}|_{L^\infty(0,T;H^s (\R^N))}^2 + |\nabla \dr|_{L^\infty(0,T;H^s_{\rho^{-1/2}}(\R^N))} \leq C_2 \,.
  \end{align*}
\end{proposition}
\begin{proof}
  Proof of the proposition can be divided into three steps. Firstly, we construct an approximate scheme for the system \eqref{rho-u-given}-\eqref{rho-u-given-Ini} by standard mollifier methods. Secondly, we derive a uniform energy bound of the approximate system. Lastly, we take limits in the approximate system by using the arguments of compactness.

  {\it{Step 1. Construct the approximate scheme.}} We can define the mollifier operator $\mathcal{J}_{\eps}$ as
  $$ \mathcal{J}_{\eps} f = \mathcal{F}^{-1} (\mathbf{1}_{|\xi| \leq \frac{1}{\eps}} \mathcal{F}(f)) \, ,$$
  where the operator $\mathcal{F}$ is the standard Fourier transform and $\mathcal{F}^{-1}$ is the inverse Fourier transform. It is easy to check that the mollifier has the proposition that $\mathcal{J}^2_\eps = \mathcal{J}_{\eps}$. We then turn to construct the approximate scheme of \eqref{rho-u-given}-\eqref{rho-u-given-Ini} as follows:
  \begin{equation}\label{ODE}
    \begin{aligned}
      \left \{ \begin{array} {l}
        \p_t \dot{\dr}^{\eps} = - \J_\eps (\u \cdot \nabla \J_\eps \dot{\dr}^\eps) + \J_\eps (\tfrac{1}{\rho} \Delta \J_\eps \dr^\eps) + \J_\eps (\tfrac{1}{\rho} \Gamma(\rho, \u, \J_\eps \dr^\eps, \J_\eps \dot{\dr}^\eps) \J_\eps \dr^\eps) \\
        \qquad \qquad \qquad \qquad \qquad + \lambda_1 \J_\eps (\tfrac{1}{\rho} (\B \J_\eps \dr^\eps + \J_\eps \dot{\dr}^\eps) ) + \lambda_2 \J_\eps (\tfrac 1 \rho (\A \J_\eps \dr^\eps) ) \, ,  \\
        \p_t \dr^\eps = \dot{\dr}^\eps - \J_\eps (\u \cdot \nabla \J_\eps \dr^\eps) \, , \\
        (\dr^\eps, \dot{\dr}^\eps)|_{t=0} = ( \J_\eps \dr^{in}, \J_\eps \tilde{\dr}^{in}) \, .
      \end{array}
      \right.
    \end{aligned}
  \end{equation}
Then by the ODE theory, we can have that there exists a maximal time $T_\eps > 0$, depending only on $\dr^{in}$, $\tilde{\dr}^{in}$, $\rho$,  $\u$ and $T_0$, such that the approximate system \eqref{ODE} admits a unique solution $\dr^\eps \in C([0, T_\eps); H^{s+1}(\mathbb{R}^N))$ and $\dot{\dr}^\eps \in C([0, T_\eps); H^{s}(\mathbb{R}^N))$. We remark that $T_\eps < T_0$ for all $\eps>0$, which is determined by the regularity of $\rho$ and $\u$. Noticing the fact that $\J^2_\eps = \J_\eps$, we have $(\J_\eps \dr^\eps, \J_\eps \dot{\dr}^\eps)$ is also a solution to the system \eqref{ODE}. Then we can immediately have
\begin{equation}\label{Pro-Step-1-1}
  (\J_\eps \dr^\eps, \J_\eps \dot{\dr}^\eps) = (\dr^\eps, \dot{\dr}^\eps)
\end{equation}
by the uniqueness of solution. Thus substituting the relation \eqref{Pro-Step-1-1} into the system \eqref{ODE}, we have
\begin{equation}\label{ODE-2}
  \begin{aligned}
    \left \{ \begin{array}{l}
      \p_t \dot{\dr}^{\eps} = - \J_\eps (\u \cdot \nabla \dot{\dr}^\eps) + \J_\eps (\tfrac{1}{\rho} \Delta \dr^\eps) + \J_\eps (\tfrac{1}{\rho} \Gamma(\rho, \u, \dr^\eps, \dot{\dr}^\eps) \dr^\eps) \\[1.5mm]
      \qquad \qquad \qquad \qquad \qquad + \lambda_1 \J_\eps (\tfrac{1}{\rho} (\B \dr^\eps + \dot{\dr}^\eps) ) + \lambda_2 \J_\eps (\tfrac 1 \rho (\A \dr^\eps) ) \, ,  \\[1.5mm]
      \p_t \dr^\eps = \dot{\dr}^\eps - \J_\eps (\u \cdot \nabla \dr^\eps) \, , \\[1.5mm]
      (\dr^\eps, \dot{\dr}^\eps)|_{t=0} = ( \J_\eps \dr^{in}, \J_\eps \tilde{\dr}^{in}) \, .
    \end{array}
    \right.
  \end{aligned}
\end{equation}

{\it{Step 2. Uniform energy estimate.}}  With the relation \eqref{Pro-Step-1-1} in hand, and by using H\"older inequality and Sobolev embedding theory, we want to derive the uniform energy bounds for the system \eqref{ODE-2}.

Firstly, we calculate the $L^2$-estimate of the approximate system \eqref{ODE-2}. Multiplying \eqref{ODE-2}$_1$ with $\dot{\dr}^\eps$ and integrating respect with $x$ over $\mathbb{R}^N$, we have
  \begin{align}\label{eps-1}
     \no \tfrac{1}{2} \tfrac{\d}{\d t} |\dot{\dr}^\eps|^2_{L^2} = & \l \u \cdot \nabla \dot{\dr}^\eps, \dot{\dr}^\eps \r + \l \tfrac{1}{\rho} \Delta \dr^\eps, \dot{\dr}^\eps \r + \l \tfrac{1}{\rho} \Gamma(\rho, \u, \dr^\eps, \dot{\dr}^\eps) \dr^\eps, \dot{\dr}^\eps \r \\
     \no & \qquad + \l \tfrac{1}{\rho} (\lambda_1 \B \dr^\eps + \lambda_2 \A \dr^\eps), \dot{\dr}^\eps \r + \l \tfrac{1}{\rho} (\lambda_1 \dot{\dr}^\eps), \dot{\dr}^\eps \r \\
     \equiv & I^\eps_1 + I^\eps_2 + I^\eps_3 + I^\eps_4 + I^\eps_5 \, .
  \end{align}
We now estimate the terms on the righthand side of \eqref{eps-1} term by term.

Integrating by parts, we can easily deduce
\begin{equation}\label{eps-I-1}
  I^\eps_1 = \tfrac 1 2 \l \div \u, |\dr^\eps|^2 \r \leq \tfrac 1 2 |\nabla \u|_{L^\infty} |\dr^\eps|^2_{L^2} \, .
\end{equation}

For the second term $I^\eps_2$, by using the fact that $\dot{\dr}^\eps = \p_t \dr^\eps + \u \cdot \nabla \dr^\eps $ and the first equation of \eqref{compressible-Liquid-Crystal-Model}, we can infer that
  \begin{align}\label{eps-I-2}
    \no I^\eps_2 = & - \l \nabla \dr^\eps, \nabla (\tfrac 1 \rho) \dot{\dr}^\eps + (\tfrac 1 \rho) \nabla \dot{\dr}^\eps \r \\
    \no = & - \tfrac 1 2 \tfrac{\d}{\d t} \int_{\mathbb{R}^N} (\tfrac 1 \rho)|\nabla \dr^\eps|^2 \d x + \tfrac 1 2 \l (\tfrac 1 \rho)_t + \div (\tfrac \u \rho), |\nabla \dr^\eps|^2 \r - \l \nabla \dr^\eps, \nabla (\tfrac 1 \rho) \dot{\dr}^\eps + \tfrac{\nabla \u}{\rho} \nabla \dr^\eps \r \\
    \no = & - \tfrac 1 2 \tfrac{\d}{\d t} \int_{\mathbb{R}^N} (\tfrac 1 \rho)|\nabla \dr^\eps|^2 \d x +  \l \tfrac{\div \u}{\rho}, |\nabla \dr^\eps|^2 \r - \l \nabla \dr^\eps, \nabla (\tfrac 1 \rho) \dot{\dr}^\eps + \tfrac{\nabla \u}{\rho} \nabla \dr^\eps \r \\
    \leq & - \tfrac 1 2 \tfrac{\d}{\d t} \int_{\mathbb{R}^N} (\tfrac 1 \rho)|\nabla \dr^\eps|^2 \d x + C |\nabla (\tfrac 1 \rho)|_{L^\infty} |\nabla \dr^\eps|_{L^2} |\dot{\dr}^\eps|_{L^2} + |\tfrac{\nabla \u + \div \u}{\rho}|_{L^\infty} |\nabla \dr^\eps|^2_{L^2} \, .
  \end{align}

Recalling the structure of Lagrangian $\Gamma(\rho, \u, \dr^\eps, \dot{\dr}^\eps)$, we have
  \begin{align}\label{eps-I-3}
    \no I^\eps_3 = &  \l -|\dot{\dr}^\eps|^2 \dr^\eps + \tfrac 1 \rho |\nabla \dr^\eps|^2 \dr^\eps - \tfrac{\lambda_2}{\rho} ({\dr^\eps}^\top \A \dr^\eps) \dr^\eps, \dot{\dr}^\eps \r \\
    \leq & C |\dr^\eps|_{L^\infty} |\dot{\dr}^\eps|_{L^\infty} (|\tfrac 1 \rho|_{L^\infty} |\nabla \dr^\eps|^2_{L^2} + |\dot{\dr}^\eps|^2_{L^2}) + |\lambda_2| |\tfrac 1 \rho|_{L^\infty} |\nabla \u|_{L^2} |\dr^\eps|^3_{L^\infty} |\dot{\dr}^\eps|_{L^2}
  \end{align}
by using H\"older inequality.

For the term $I^\eps_4$ and $I^\eps_5$, it is easy to derive that
\begin{equation}\label{eps-I-4}
  \begin{aligned}
    I^\eps_4 \leq ( |\lambda_1| + |\lambda_2| ) |\tfrac 1 \rho|_{L^\infty} |\nabla \u|_{L^2} |\dr^\eps|_{L^\infty} |\dot{\dr}^\eps|_{L^2}
  \end{aligned}
\end{equation}
and
\begin{equation}\label{eps-I-5}
  I^\eps_5 = \lambda_1 \int_{\mathbb{R}^N} \tfrac 1 \rho |\dot{\dr}^\eps|^2 \d x \, .
\end{equation}

Combining with the above estimates \eqref{eps-I-1}-\eqref{eps-I-5} infers that
\begin{align}\label{eps-Basic-Energy}
  \no & \tfrac 1 2 \tfrac{\d}{\d t} \int_{\mathbb{R}^N} |\dot{\dr}^\eps|^2 + \tfrac 1 \rho |\nabla \dr^\eps|^2 \d x - \lambda_1 \int_{\mathbb{R}^N} \tfrac 1 \rho |\dot{\dr}^\eps|^2 \d x \\
  \no \leq & \tfrac 1 2 |\nabla \u|_{L^\infty} |\dr^\eps|^2_{L^2} + (|\lambda_1| + |\lambda_2|) |\tfrac 1 \rho|_{L^\infty} |\nabla \u|_{L^2} |\dr^\eps|_{L^\infty} |\dot{\dr}^\eps|_{L^2} \\
  \no & + C |\nabla (\tfrac 1 \rho)|_{\infty} |\nabla \dr^\eps|_{L^2} |\dot{\dr}^\eps|_{L^2} + |\lambda_2| |\tfrac 1 \rho|_{L^\infty} |\nabla \u|_{L^2} |\dr^\eps|^3_{L^\infty} |\dot{\dr}^\eps|_{L^2} \\
  & + C |\dr^\eps|_{L^\infty} |\dot{\dr}^\eps|_{L^\infty} (|\tfrac 1 \rho|_{L^\infty} |\nabla \dr^\eps|^2_{L^2} + |\dot{\dr}^\eps|^2_{L^2}) + |\tfrac{1}{\rho}|_{L^\infty} |\nabla \u|_{L^\infty} |\nabla \dr^\eps|^2_{L^2}
\end{align}

Secondly, we concern with the higher order energy estimate of the approximate system \eqref{ODE-2}. Acting the $k$-order $(1 \leq k \leq s)$ derivative operator $\nabla^k$ on \eqref{ODE-2}$_1$ and taking $L^2$-inner product by multiplying $\nabla^k \dot{\dr}^\eps$, then integrating by parts, we can obtain
\begin{align}\label{eps-k}
  \no & \tfrac 1 2 \tfrac{\d}{\d t} |\nabla^k \dot{\dr}^\eps|^2_{L^2}
  = \l \nabla^k (\u \cdot \dot{\dr}^\eps), \nabla^k \dot{\dr}^\eps \r +\l \nabla^k (\tfrac{1}{\rho} \Delta \dr^\eps), \nabla^k \dot{\dr}^\eps \r \\
  \no & + \l \nabla^k (\tfrac{1}{\rho} \Gamma(\rho, \u, \dr^\eps, \dot{\dr}^\eps) \dr^\eps), \nabla^k \dot{\dr}^\eps \r + \l \nabla^k (\tfrac 1 \rho (\lambda_1 \B \dr^\eps + \lambda_2 \A \dr^\eps)), \nabla^k \dot{\dr}^\eps \r \\
  & + \l \nabla^k (\tfrac{\lambda_2}{\rho} \A \dr^\eps), \nabla^k \dot{\dr}^\eps \r
  \equiv \tilde{I}^\eps_1 + \tilde{I}^\eps_2 + \tilde{I}^\eps_3 + \tilde{I}^\eps_4 + \tilde{I}^\eps_5 \, .
\end{align}
We now turn to estimate the terms $\tilde{I}^\eps_j (j=1,...,5)$ one by one.

For the term of $\tilde{I}^\eps_1$, we have
\begin{align}\label{eps-k-I-1}
  \no \tilde{I}^\eps_1 = & - \tfrac{1}{2} \l \div \u, |\nabla^k \dot{\dr}^\eps|^2 \r + \l \nabla \u \cdot \nabla^k \dot{\dr}^\eps, \nabla^k \dot{\dr}^\eps \r + \sum_{\substack{a+b=k,\\ a \geq 2}} \l \nabla^a \u \nabla^{b+1} \dot{\dr}^\eps, \nabla^k \dot{\dr}^\eps \r \\
  \no \lesssim & |\nabla \u|_{L^\infty} |\nabla^k \dot{\dr}^\eps|^2_{L^2} + \sum_{\substack{a+b=k,\\ a \geq 2}} |\nabla^a \u|_{L^4} |\nabla^{b+1} \dot{\dr}^\eps|_{L^4} |\nabla^k \dot{\dr}^\eps|_{L^2} \\
  \lesssim & |\nabla \u|_{H^s} |\dot{\dr}^\eps|^2_{H^s}
\end{align}
by taking advantage of H\"older inequality and Sobolev embedding theory.

We then turn to deal with the term $\tilde{I}^\eps_2$. We can divide $\tilde{I}^\eps_2$ into the following two parts
\begin{align*}
  \tilde{I}^\eps_2 = \l \tfrac{1}{\rho} \nabla^k \Delta \dr^\eps, \nabla^k \dot{\dr}^\eps \r + \sum_{\substack{a+b=k,\\ a \geq 1}} \l \nabla^a (\tfrac{1}{\rho}) \nabla^b \Delta \dr^\eps, \nabla^k \dot{\dr}^\eps \r = \tilde{I}^\eps_{21} + \tilde{I}^\eps_{22} \, .
\end{align*}
We now calculate $\tilde{I}^\eps_{21}$ and $\tilde{I}^\eps_{22}$, respectively. By using H\"older inequality, Sobolev embedding theory and the first equation of \eqref{compressible-Liquid-Crystal-Model}, we can derive that
\begin{align}\label{eps-k-I-21}
   \no & \tilde{I}^\eps_{21} = - \tfrac 1 2 \tfrac{\d}{\d t} \int_{\mathbb{R}^N} (\tfrac{1}{\rho}) |\nabla^{k+1} \dr^\eps|^2 \d x + \tfrac 1 2 \l \p_t (\tfrac{1}{\rho}) + \div (\tfrac{\u}{\rho}), |\nabla^{k+1} \dr^\eps|^2 \r - \l \nabla^{k+1} \dr^\eps, \tfrac{\nabla \u}{\rho} \nabla^{k+1} \dr^\eps \r \\
   & \qquad - \l \nabla^{k+1} \dr^\eps, \nabla (\tfrac 1 \rho) \nabla^k \dot{\dr}^\eps \r + \sum_{\substack{a+b=k+1,\\ a \geq 2}} \l (\tfrac 1 \rho) \nabla^a \u  \nabla^{b+1} \dr^\eps, \nabla^{k+1} \dr^\eps \r \\
   \no & \leq - \tfrac 1 2 \tfrac{\d}{\d t} \int_{\mathbb{R}^N} (\tfrac{1}{\rho}) |\nabla^{k+1} \dr^\eps|^2 \d x + 2 |\tfrac 1 \rho|_{L^\infty} |\nabla \u|_{L^\infty} |\nabla^{k+1} \dr^\eps|^2_{L^2} + |\nabla(\tfrac 1 \rho)|_{L^\infty} |\nabla^{k+1} \dr^\eps|_{L^2} |\nabla^k \dot{\dr}^\eps|_{L^2} \\
   \no & \quad + |\tfrac 1 \rho|_{L^\infty} |\nabla \dr^\eps|_{L^\infty} |\nabla^{k+1} \dr^\eps|_{L^2} |\nabla^{k+1} \u|_{L^2} + \sum_{\substack{a+b=k+1,\\ 2 \leq a \leq k}} |\tfrac 1 \rho|_{L^\infty} |\nabla^{k+1} \dr^\eps|_{L^2} |\nabla^{a+1} \u|_{L^4} |\nabla^{b+1} \dr^\eps|_{L^4} \\
   \no & \leq - \tfrac 1 2 \tfrac{\d}{\d t} \int_{\mathbb{R}^N} (\tfrac{1}{\rho}) |\nabla^{k+1} \dr^\eps|^2 \d x + C |\tfrac 1 \rho|_{L^\infty} |\nabla \u|_{H^s} |\nabla \dr^\eps|^2_{H^s} + C  |\nabla (\tfrac 1 \rho)|_{L^\infty} |\nabla \dr^\eps|_{H^s} |\dot{\dr}^\eps|_{H^s}
\end{align}
and
\begin{align}\label{eps-k-I-22}
  \no \tilde{I}^\eps_{22} = & \l \nabla(\tfrac 1 \rho) \nabla^{k-1} \Delta \dr^\eps, \nabla^k \dot{\dr}^\eps \r + \l \nabla^k (\tfrac 1 \rho) \Delta \dr^\eps, \nabla^k \dot{\dr}^\eps \r + \sum_{\substack{a+b=k, \\ 2 \leq a \leq k-1}} \l \nabla^a (\tfrac 1 \rho) \nabla^b \Delta \dr^\eps, \nabla^k \dot{\dr}^\eps \r \\
  \no \leq & |\nabla (\tfrac 1 \rho)|_{L^\infty} |\nabla^{k+1} \dr^\eps|_{L^2} |\nabla^k \dot{\dr}^\eps|_{L^2} + |\nabla^k (\tfrac 1 \rho)|_{L^2} |\Delta \dr^\eps|_{L^\infty} |\nabla^k \dot{\dr}^\eps|_{L^2} \\
  \no & + \sum_{\substack{a+b=k, \\ 2 \leq a \leq k-1}} |\nabla^a (\tfrac 1 \rho)|_{L^4} |\nabla^{b+2} \dr^\eps|_{L^4} |\nabla^k \dot{\dr}^\eps|_{L^2} \\
  \lesssim & |\tfrac 1 \rho|_{\dot{H}^s} |\nabla \dr^\eps|_{H^s} |\dot{\dr}^\eps|_{H^s} \, .
\end{align}
Combining with the estimates \eqref{eps-k-I-21} and \eqref{eps-k-I-22}, we obtain
\begin{equation}\label{eps-k-I-2}
  \begin{aligned}
    \tilde{I}^\eps_2 \leq - \tfrac 1 2 \tfrac{\d}{\d t} \int_{\mathbb{R}^N} (\tfrac{1}{\rho}) |\nabla^{k+1} \dr^\eps|^2 \d x + C |\tfrac 1 \rho|_{L^\infty} |\nabla \u|_{H^s} |\nabla \dr^\eps|^2_{H^s} + C |\tfrac 1 \rho|_{\dot{H}^s} |\nabla \dr^\eps|_{H^s} |\dot{\dr}^\eps|_{H^s} \, .
  \end{aligned}
\end{equation}

For the estimate of the term $\tilde{I}^\eps_3$, by the structure of $\Gamma(\rho, \u, \dr^\eps, \dot{\dr}^\eps)$, we know that
\begin{align*}
  \tilde{I}^\eps_3 = & \l \nabla^k (- |\dot{\dr}^\eps|^2 \dr^\eps), \nabla^k \dot{\dr}^\eps \r + \l \nabla^k (\tfrac 1 \rho |\nabla \dr^\eps|^2 \dr^\eps), \nabla^k \dot{\dr}^\eps \r - \l \nabla^k ( \tfrac{\lambda_2}{\rho} ( {\dr^\eps}^\top \A \dr^\eps ) \dr^\eps), \nabla^k \dot{\dr}^\eps \r \\
  \equiv & \tilde{I}^\eps_{31} + \tilde{I}^\eps_{32} + \tilde{I}^\eps_{33} \, .
\end{align*}
Taking advantage of H\"older inequality and Sobolev embedding theory infers that
\begin{align}\label{eps-k-I-31}
  \no \tilde{I}^\eps_{31} \leq & |\dr^\eps|_{L^\infty} \sum_{a+b=k} \l |\nabla^a \dot{\dr}^\eps| |\nabla^b \dot{\dr}^\eps|, |\nabla^k \dot{\dr}^\eps| \r + \sum_{\substack{a+b+c=k,\\ c \geq 1}} \l |\nabla^a \dot{\dr}^\eps| |\nabla^b \dot{\dr}^\eps| |\nabla^c \dr^\eps|, |\nabla^k \dot{\dr}^\eps| \r \\
  \no \leq & 2 |\dr^\eps|_{L^\infty} |\dot{\dr}^\eps|_{L^\infty} |\nabla^k \dot{\dr}^\eps|^2_{L^2} + |\dr^\eps|_{L^\infty} \sum_{\substack{a+b=k, \\ a, b \geq 1}} |\nabla^a \dot{\dr}^\eps|_{L^4} |\nabla^b \dot{\dr}^\eps|_{L^4} |\nabla^k \dot{\dr}^\eps|_{L^2} \\
  \no & + |\dot{\dr}^\eps|_{L^\infty} |\dot{\dr}^\eps|_{L^4} |\nabla^k \dr^\eps|_{L^4} |\nabla^k \dot{\dr}^\eps|_{L^2} + \sum_{\substack{a+b+c=k, \\ 1 \leq c \leq k-1}} |\nabla^c \dr^\eps|_{L^\infty} |\nabla^a \dot{\dr}^\eps|_{L^4} |\nabla^b \dot{\dr}^\eps|_{L^2} |\nabla^k \dot{\dr}^\eps|_{L^2} \\
  \lesssim & ( |\dr^\eps|_{L^\infty} + |\nabla \dr^\eps|_{H^s} ) |\dot{\dr}^\eps|^3_{H^s} \, .
\end{align}
Similar as the above calculation, we have
\begin{align}\label{eps-k-I-32}
  \no \tilde{I}^\eps_{32} & \lesssim |\nabla^{k-1} (\tfrac{1}{\rho})|_{L^4} (|\nabla^2 \dr^\eps|_{L^4} |\nabla \dr^\eps|_{L^\infty} |\dr^\eps|_{L^\infty} + |\nabla \dr^\eps|_{L^4} |\nabla \dr^\eps|_{L^\infty}^2) |\nabla^k \dot{\dr}^\eps|_{L^2} \\
  \no & + \sum_{\substack{a+b+c=k,\\ 1 \leq c \leq k-2}} \sum_{a_1 + a_2 =a} |\nabla^{a_1+1} \dr^\eps|_{L^4} |\nabla^{a_2+1} \dr^\eps|_{L^4} |\nabla^b \dr^\eps|_{L^\infty} |\nabla^c (\tfrac{1}{\rho})|_{L^\infty} |\nabla^k \dot{\dr}^\eps|_{L^2} \\
  \no & + |\tfrac 1 \rho|_{L^\infty} ( |\nabla \dr^\eps|^2_{L^\infty} |\nabla^k \dr^\eps|_{L^2} + |\dr^\eps|_{L^\infty} |\nabla \dr^\eps|_{L^\infty} |\nabla^{k+1} \dr^\eps|_{L^2} ) |\nabla^k \dot{\dr}^\eps|_{L^2} \\
  \no & +\! |\nabla^k (\tfrac 1 \rho)|_{L^2} |\nabla \dr^\eps|^2_{L^\infty} |\dr^\eps|_{L^\infty} |\nabla^k \dot{\dr}^\eps|_{L^2} \!+\! \!\sum_{\substack{a+b=k, \\ 1 \leq a \leq k-1}}\!\! |\tfrac 1 \rho|_{L^\infty} |\dr^\eps|_{L^\infty} |\nabla^{a+1} \dr^\eps|_{L^4} |\nabla^{b+1} \dr^\eps|_{L^4} |\nabla^k \dot{\dr}^\eps|_{L^2}  \\
  \no & + \sum_{\substack{a+b+c=k, \\1 \leq a, b, c \leq k-1}} |\tfrac 1 \rho|_{L^\infty} |\nabla^{a+1} \dr^\eps|_{L^4} |\nabla^{b+1} \dr^\eps|_{L^4} |\nabla^c \dr^\eps|_{L^\infty} |\nabla^k \dot{\dr}^\eps|_{L^2} \\
  \lesssim & ( |\tfrac 1 \rho|_{L^\infty} + |\tfrac 1 \rho|_{\dot{H}^s} ) (|\dr^\eps|_{L^\infty} + |\nabla \dr^\eps|_{H^s}) |\nabla \dr^\eps|^2_{H^s} |\dot{\dr}^\eps|_{H^s} \, ,
\end{align}
and
\begin{align}\label{eps-k-I-33}
  \no \tilde{I}^\eps_{33} \leq & |\lambda_2| |\tfrac 1 \rho|_{L^\infty} \sum_{a+b+c+e=k} \l |\nabla^a \dr^\eps \nabla^b \dr^\eps \nabla^c \dr^\eps| |\nabla^{e+1} \u|, |\nabla^k \dot{\dr}^\eps| \r \\
  \no & + |\lambda_2| |\nabla^k (\tfrac 1 \rho)|_{L^2} |\dot{\dr}^\eps|^3_{L^\infty} |\nabla \u|_{L^\infty} |\nabla^k \dot{\dr}^\eps|_{L^2} \\
  \no & + |\lambda_2| \sum_{\substack{a+b+c=k, \\ 1 \leq c \leq k-1}} \sum_{a_1 + a_2 +a_3 = a} \l |\nabla^c (\tfrac 1 \rho)| |\nabla^{a_1} \dr^\eps| |\nabla^{a_2} \dr^\eps| |\nabla^{a_3} \dr^\eps| |\nabla^{b+1} \u|, |\nabla^k \dot{\dr}^\eps| \r \\
  \lesssim & |\lambda_2| (|\tfrac 1 \rho|_{L^\infty} + |\tfrac 1 \rho|_{\dot{H}^s}) |\nabla \u|_{H^s} (|\dr^\eps|_{L^\infty} + |\nabla \dr^\eps|_{H^s})^3 |\dot{\dr}^\eps|_{H^s} \, .
\end{align}
With the above estimates \eqref{eps-k-I-31}-\eqref{eps-k-I-33} in hand, we have
\begin{align}\label{eps-k-I-3}
  \no & \l \nabla^k (\tfrac 1 \rho \Gamma (\rho, \u, \dr^\eps, \dot{\dr}^\eps) \dr^\eps), \nabla^k \dot{\dr}^\eps \r
  \lesssim ( |\dr^\eps|_{L^\infty} + |\nabla \dr^\eps|_{H^s} ) |\dot{\dr}^\eps|^3_{H^s} \\
  \no & \qquad \quad \quad +  ( |\tfrac 1 \rho|_{L^\infty} + |\tfrac 1 \rho|_{\dot{H}^s} ) (|\dr^\eps|_{L^\infty} + |\nabla \dr^\eps|_{H^s}) |\nabla \dr^\eps|^2_{H^s} |\dot{\dr}^\eps|_{H^s} \\
  & \qquad \quad \quad +  |\lambda_2| (|\tfrac 1 \rho|_{L^\infty} + |\tfrac 1 \rho|_{\dot{H}^s}) |\nabla \u|_{H^s} (|\dr^\eps|_{L^\infty} + |\nabla \dr^\eps|_{H^s})^3 |\dot{\dr}^\eps|_{H^s} \, .
\end{align}

Then we estimate the last two parts of the right-hand side of \eqref{eps-k}, by using H\"older inequality and Sobolev embedding theory, we can infer that
\begin{align} \label{eps-k-I-4}
  \no \tilde{I}^\eps_4 \leq & (|\lambda_1| + |\lambda_2|) |\nabla^k (\tfrac 1 \rho)|_{L^2} |\nabla \u|_{L^\infty} |\dr^\eps|_{L^\infty} |\nabla^k \dot{\dr}^\eps|_{L^2} \\
  \no & + (|\lambda_1| + |\lambda_2|) |\tfrac 1 \rho|_{L^\infty} \sum_{a+b=k} \l |\nabla^{a+1} \u| |\nabla^b \dr^\eps|, |\nabla^k \dot{\dr}^\eps| \r \\
  \no & + (|\lambda_1| + |\lambda_2|) \sum_{\substack{a+b+c=k,\\ 1 \leq a \leq k-1}} \l |\nabla (\tfrac 1 \rho)| |\nabla^{b+1} \u| |\nabla^c \dr^\eps|, |\nabla^k \dot{\dr}^\eps| \r \\
  \lesssim & (|\lambda_1| + |\lambda_2|) (|\tfrac 1 \rho|_{L^\infty} + |\tfrac 1 \rho|_{\dot{H}^s}) |\nabla \u|_{H^s} (|\dr^\eps|_{L^\infty} + |\nabla \dr^\eps|_{H^s}) |\dot{\dr}^\eps|_{H^s} \, ,
\end{align}
and
\begin{align} \label{eps-k-I-5}
  \no \tilde{I}^\eps_5 \leq & \int_{\mathbb{R}^N} \tfrac{\lambda_1}{\rho} |\nabla^k \dot{\dr}^\eps|^2 \d x + |\lambda_1| |\nabla^k (\tfrac 1 \rho)|_{L^2} |\dot{\dr}^\eps|_{L^\infty} |\nabla^k \dot{\dr}^\eps|_{L^2} \\
  \no & + |\lambda_1| \sum_{\substack{a+b=k,\\a,b \geq 1}} |\nabla^a (\tfrac 1 \rho)|_{L^4} |\nabla^b \dot{\dr}^\eps|_{L^4} |\nabla^k \dot{\dr}^\eps|_{L^2} \\
  \leq & \int_{\mathbb{R}^N} \tfrac{\lambda_1}{\rho} |\nabla^k \dot{\dr}^\eps|^2 \d x + C |\lambda_1| |\tfrac 1 \rho|_{\dot{H}^s} |\dot{\dr}^\eps|^2_{H^s} \, .
\end{align}

Consequently, combining with the estimates \eqref{eps-k-I-1}, \eqref{eps-k-I-2}, \eqref{eps-k-I-3}, \eqref{eps-k-I-4} and \eqref{eps-k-I-5}, we have the following higher order estimate:
\begin{align} \label{eps-Higher-Energy}
  \no & \tfrac 1 2 \tfrac{\d}{\d t} \int_{\mathbb{R}^N} |\nabla^k \dot{\dr}^\eps|^2 + \tfrac 1 \rho |\nabla^{k+1} \dr^\eps|^2 \d x - \int_{\mathbb{R}^N} \tfrac{\lambda_1}{\rho} |\nabla^k \dot{\dr}^\eps|^2 \d x \\
  \no \lesssim & |\nabla \u|_{H^s} |\dot{\dr}^\eps|^2_{H^s} + |\tfrac 1 \rho|_{\dot{H}^s} (|\nabla \dr^\eps|_{H^s} + |\lambda_1| |\dot{\dr}^\eps|_{H^s}) |\dot{\dr}^\eps|_{H^s} + ( |\dr^\eps|_{L^\infty} + |\nabla \dr^\eps|_{H^s} ) |\dot{\dr}^\eps|^3_{H^s} \\
  \no & + |\tfrac{1}{\rho}|_{L^\infty} |\nabla \u|_{H^s} |\nabla \dr^\eps|^2_{H^s} \!+\! (1 \!+\! |\lambda_2|) ( |\tfrac 1 \rho|_{L^\infty} \!+\! |\tfrac 1 \rho|_{\dot{H}^s} ) (1 \!+\! |\nabla \u|_{H^s}) (|\dr^\eps|_{L^\infty} \!+\! |\nabla \dr^\eps|_{H^s})^3 |\dot{\dr}^\eps|_{H^s} \\
  & + (|\lambda_1| + |\lambda_2|) (|\tfrac 1 \rho|_{L^\infty} + |\tfrac 1 \rho|_{\dot{H}^s}) |\nabla \u|_{H^s} (|\dr^\eps|_{L^\infty} + |\nabla \dr^\eps|_{H^s}) |\dot{\dr}^\eps|_{H^s} \, .
\end{align}
Then the basic energy estimate \eqref{eps-Basic-Energy} and higher order estimate \eqref{eps-Higher-Energy} give us
  \begin{align}\label{eps-Energy-Est}
    \no & \tfrac 1 2 \tfrac{\d}{\d t} \sum_{k=0}^{s} \int_{\mathbb{R}^N} |\nabla^k \dot{\dr}^\eps|^2 + \tfrac 1 \rho |\nabla^{k+1} \dr^\eps|^2 \d x - \lambda_1 \sum_{k=0}^{s} \int_{\mathbb{R}^N} \tfrac 1 \rho |\nabla^k \dot{\dr}^\eps|^2 \d x \\
    \no \lesssim & |\nabla \u|_{H^s} |\dot{\dr}^\eps|^2_{H^s} + |\tfrac 1 \rho|_{\dot{H}^s} (|\nabla \dr^\eps|_{H^s} + |\lambda_1| |\dot{\dr}^\eps|_{H^s}) |\dot{\dr}^\eps|_{H^s} + ( |\dr^\eps|_{L^\infty} + |\nabla \dr^\eps|_{H^s} ) |\dot{\dr}^\eps|^3_{H^s} \\
  \no & + |\tfrac{1}{\rho}|_{L^\infty} |\nabla \u|_{H^s} |\nabla \dr^\eps|^2_{H^s} \!+\! (1 \!+\! |\lambda_2|) ( |\tfrac 1 \rho|_{L^\infty} \!+\! |\tfrac 1 \rho|_{\dot{H}^s} ) (1 \!+\! |\nabla \u|_{H^s}) (|\dr^\eps|_{L^\infty} \!+\! |\nabla \dr^\eps|_{H^s})^3 |\dot{\dr}^\eps|_{H^s} \\
  & + (|\lambda_1| + |\lambda_2|) (|\tfrac 1 \rho|_{L^\infty} + |\tfrac 1 \rho|_{\dot{H}^s}) |\nabla \u|_{H^s} (|\dr^\eps|_{L^\infty} + |\nabla \dr^\eps|_{H^s}) |\dot{\dr}^\eps|_{H^s} \, .
  \end{align}

 Noticing that $|\dr^\eps|_{L^\infty}$ in the right-hand side of the $H^s$-energy estimate \eqref{eps-Energy-Est} is uncontrolled by now. We now try to control it. Obviously, it holds that
 \begin{align*}
   |\dr^\eps|_{L^\infty} \leq & |\dr^\eps - \J_\eps \dr^{in}|_{L^\infty} + |\J_\eps \dr^{in}|_{L^\infty} \leq |\dr^\eps - \J_\eps \dr^{in}|_{H^2} + 1 \\
   \leq & |\dr^\eps - \J_\eps \dr^{in}|_{L^2} + |\nabla \dr^\eps|_{H^1} + |\nabla \dr^{in}|_{H^1} + 1 \, .
 \end{align*}
 And by using the second equation of system \eqref{ODE-2} and Sobolev embedding theory, we can derive that
 \begin{align*}
   & \tfrac 1 2 \tfrac{\d}{\d t} |\dr^\eps - \J_\eps \dr^{in}|^2_{L^2} = \l \dr^\eps - \J_\eps \dr^{in}, \dot{\dr}^\eps - \J_\eps (\u \cdot \nabla \dr^\eps) \r \\
   \leq & |\dr^\eps - \J_\eps \dr^{in}|_{L^2} |\dot{\dr}^\eps|_{L^2} + |\u|_{H^2} |\dr^\eps - \J_\eps \dr^{in}|_{L^2} |\nabla \dr^\eps|_{L^2} \, .
 \end{align*}

 The $L^2$-norm estimate \eqref{eps-Basic-Energy} and higher order estimate \eqref{eps-Higher-Energy} tell us that we can define the free energy functional as
 \begin{equation}
   E_\eps (t) = |\dot{\dr}^\eps|^2_{H^s} + \sum_{k=0}^{s} |\tfrac{1}{\sqrt{\rho}} \nabla^{k+1} \dr^\eps|^2_{L^2} + |\dr^\eps - \J_\eps \dr^{in}|^2_{L^2} \, ,
 \end{equation}
 and dissipative energy functional as
 \begin{equation}
   D_\eps (t) = - \lambda_1 \sum_{k=0}^{s} |\tfrac{1}{\sqrt{\rho}} \nabla^k \dot{\dr}^\eps|^2_{L^2} \, .
 \end{equation}
 Then we can easy get
 \begin{align}\label{dr-eps-0}
   \tfrac 1 2 \tfrac{\d}{\d x} |\dr^\eps - \J_\eps \dr^{in}|^2_{L^2} \leq (1 + |\u|_{H^2} |\rho|^{\frac{1}{2}}_{L^\infty} ) E_\eps (t) \, ,
 \end{align}
 meanwhile, the bound
 \begin{equation}\label{dr-eps}
   |\dr^\eps|_{L^\infty} \leq (1 + |\rho|^{\frac{1}{2}}_{L^\infty}) E_\eps^{\frac{1}{2}} (t) + E_\eps^{\frac{1}{2}} (0) + 1 .
 \end{equation}
Then the inequalities \eqref{eps-Energy-Est}, \eqref{dr-eps-0} and \eqref{dr-eps} imply that there exists a positive constant $C_2$, depending only on $\lambda_1$, $\lambda_2$ and $\dr^{in}$, such that for all $\eps > 0$ the inequality
\begin{equation}\label{eps-Energy-Est-0}
  \tfrac{\d}{\d t} E_\eps (t) + D_\eps (t) \leq C_2 (1 + |\tfrac 1 \rho|_{L^\infty} + |\tfrac 1 \rho|_{\dot{H}^s}) (1 + |\rho|_{L^\infty}^{\frac{1}{2}})^3 (1 + |\u|_{H^{s+1}}) \sum_{l=1}^{4} E_\eps^{\frac{l}{2}} (t)
\end{equation}
holds for all $t \in[0, T_\eps)$.

Finally, based on the estimate \eqref{eps-Energy-Est-0}, we then try to get the uniform bounds of the energy functional $E_{\eps} (t)$ by using Gronwall arguments. Noticing that
$$
E_{\eps} (0) = | \tilde{\dr}^{in} |_{H^s}^2 + \sum_{k=0}^{s} | \tfrac{1}{\sqrt{\rho(\cdot,0)}} \nabla^{k+1} \dr^{in} |_{L^2}^2 \equiv \hat{E}^{in} < +\infty \,,
$$
defining the time $T^*_{\eps}$ as
$$
T^*_{\eps} = \big\{ \tau \in [0, T_{\eps} ); \sup_{t \in [0, \tau]} E_{\eps}^{\frac{1}{2}} (t) \leq 2 \sqrt{\hat{E}^{in}} \big\}
$$
one can immediately derive that $T^*_{\eps} >0$ by using the continuity of the energy functional $E_{\eps} (t)$. Then the inequality \eqref{eps-Energy-Est-0} implies that
\begin{equation*}
  \tfrac{\d}{\d t} E_{\eps}^{\frac{1}{2}} (t) \leq \Lambda (t) ( 1 + E_{\eps}^{\frac{1}{2}} (t) + E_{\eps} (t) + E_{\eps}^{\frac{3}{2}} (t) ) \,,
\end{equation*}
where the non-negative function $\Lambda(t)$ is
 $$
 \Lambda(t) = C_1 (1 + |\tfrac 1 \rho|_{L^\infty} + |\tfrac 1 \rho|_{\dot{H}^s}) (1 + |\rho|_{L^\infty}^{\frac{1}{2}})^3 (1 + |\u|_{H^{s+1}}) \in L^1(0, T_0)\,.
 $$
Integrating with respect to $t \ (t \leq T^*_{\eps} )$ over $(0, t)$, so we have
\begin{equation*}
  \begin{aligned}
    E_{\eps}^{\frac{1}{2}} (t) \leq & F(t) := E_{\eps}^{\frac{1}{2}} (0) + \int_{0}^{t} \Lambda (\tau) \d \tau \sup_{\tau \in [0,t]} ( 1 + E_{\eps}^{\frac{1}{2}} (\tau) + E_{\eps} (\tau) + E_{\eps}^{\frac{3}{2}} (\tau) ) \\
    \leq & \sqrt{\hat{E}^{in}} + \int_{0}^{t} \Lambda (\tau) \d \tau \sup_{\tau \in [0,t]} ( 1 + \sqrt{\hat{E}^{in}} )^3
  \end{aligned}
\end{equation*}
for all $\eps>0$. Since $F(0)=\sqrt{\hat{E}^{in}}$, and noticing that $\int_{0}^{t} \Lambda (\tau) \d \tau$ is continuous in $t$ and is independent of $\eps>0$, then there is a time $T < T^*_{\eps}$ independent of $\eps>0$ such that $F(t) \leq 2 \sqrt{\hat{E}^{in}}$ for all $t \in (0, T)$. Namely, for all $\eps >0$ and $t \in [0, T]$, we have $E_{\eps}(t) \leq 4 \hat{E}^{in}$. Consequently, we get the following uniform energy bound
\begin{equation}\label{Uniform-Bound-E-eps}
  |\dot{\dr}^\eps|^2_{H^s} + \sum_{k=0}^{s} |\tfrac{1}{\sqrt{\rho}} \nabla^{k+1} \dr^\eps|^2_{L^2} + |\dr^\eps - \J_\eps \dr^{in}|^2_{L^2} \leq 4 \hat{E}^{in}
\end{equation}
for all $\eps>0$ and $t \in [0, T]$.

{\it{ Step 3. Pass to the limits.}} Based on the bound \eqref{dr-eps} and \eqref{Uniform-Bound-E-eps}, and taking limit in the approximate system \eqref{ODE-2} as $\eps \to 0$, we can obtain that there exists a $\dr \in L^{\infty} ([0,T] \times \R^N)$ satisfying $\nabla \dr \in C(0,T; H^s_{\rho^{-1/2}} (\R^N))$, $\dot{\dr} \in C(0,T; H^s(\R^N))$, meanwhile, $\dr$ satisfies the first equation of the system \eqref{rho-u-given} with the initial data \eqref{rho-u-given-Ini}. Then we know that
\begin{equation*}
  \begin{aligned}
    \left\{
    \begin{array}{l}
      \rho \ddot{\dr} = \Delta \dr + \Gamma (\rho, \u, \dr, \dot{\dr}) \dr + \lambda_1 (\dot{\dr} + \B \dr) + \lambda_2 \A \dr \, ,\\[2mm]
      (\dr, \dot{\dr}) \big{|}_{t=0} = ( \dr^{in}(x), \tilde{\dr}^{in}(x)) \in \mathbb{S}^{N+1} \times \R^{N}
    \end{array}
    \right.
  \end{aligned}
\end{equation*}
with $\dr \in L^{\infty} ([0,T] \times \R^{N})$, where
$$
\Gamma(\rho, u, \dr, \dot{\dr}) = - \rho |\dot{\dr}|^2 + |\nabla \dr|^2 - \lambda_2 {\dr}^\top \A \dr \,.
$$
Then Lemma \ref{lemma-d} yields that $\dr \in \mathbb{S}^{N-1}$. Therefore we finish the proof of the Proposition \ref{Proposition-1}
\end{proof}

\section{Local existence: Proof of Theorem \ref{theorem-1}} \label{Sect-5}
The aim of this section is to prove the local well-posedness of the compressible Ericksen-Leslie liquid crystal flow \eqref{compressible-Liquid-Crystal-Model}-\eqref{Initial-Data} with large initial data under the Leslie coefficients constraint $\mu_1 \geq 0,\ \mu_4 >0,\ \lambda_1 \leq 0,\ \mu_5 + \mu_6 + \tfrac{\lambda_2^2}{\lambda_1} \geq 0$, namely, Theorem \ref{theorem-1}. Firstly, we construct the approximate system of \eqref{compressible-Liquid-Crystal-Model} by iteration. More precisely, the iterating approximating sequences are constructed as follows:
\begin{equation}\label{Approx-Equat}
  \begin{aligned}
    \left\{ \begin{array}{c}
      \p_t \rho^{n+1} + \u^n \cdot \nabla \rho^{n+1} + \rho^{n+1} \div \u^n = 0\, ,\\ [2mm]
      \p_t \u^{n+1}  + \u^n \cdot \nabla \u^{n+1}  + \tfrac{p'(\rho^n)}{\rho^n} \nabla \rho^{n+1} = \tfrac{1}{\rho^n} \div ( \Sigma_1^{n+1} + \Sigma_2^n + \Sigma_3^{n+1} )\, ,\\ [2mm]
      \ddot{\dr}^{n+1} = \tfrac{1}{\rho^n} \Delta \dr^{n+1} + \tfrac{1}{\rho^n} \Gamma^{n+1} \dr^{n+1} + \tfrac{\lambda_1}{\rho^n} (\dot{\dr}^{n+1} + \B^n \dr^{n+1}) + \tfrac{\lambda_2}{\rho^n} \A^n \dr^{n+1} \, , \\ [2mm]
      (\rho^{n+1}, \u^{n+1}, \dr^{n+1}, \dot{\dr}^{n+1})(x, t) |_{t=0} = (\rho^{in}, \u^{in}, \dr^{in}, \tilde{\dr}^{in}) (x) \, ,
    \end{array}\right.
  \end{aligned}
\end{equation}
where
\begin{align*}
  \Sigma_1^{n+1} := & \tfrac{1}{2} \mu_4 ( \nabla \u^{n+1} + \nabla^\top \u^{n+1} ) + \xi \div \u^{n+1} I \, , \\
  \Sigma_2^n := & \tfrac{1}{2} |\nabla \dr^n|^2 I - \nabla \dr^n \odot \nabla \dr^n \, , \\
  \Sigma_3^{n+1} := & \tilde{\sigma}(\u^{n+1}, \dr^n, \dot{\dr}^n) \, ,
\end{align*}
and
\begin{align*}
  \A^n = & \tfrac{1}{2} (\nabla \u^n + \nabla^\top \u^n), \ \B^n = \tfrac{1}{2} (\nabla \u^n - \nabla^\top \u^n) \, , \\
  \dot{\dr}^{n+1} = & \p_t \dot{\dr}^{n+1} + \u^n \cdot \nabla \dr^{n+1} \, , \\
  \Gamma^{n+1} = & \Gamma (\rho^n, \u^n, \dr^{n+1}, \dot{\dr}^{n+1}) \, ,
\end{align*}
$\dot{\dr}^{n+1}$ is the iterating approximate material derivatives. And we start the approximating system with $$(\rho^0, \u^0, \dr^0, \dot{\dr}^0)(x, t)|_{t=0} = (\rho^{in}, \u^{in}, \dr^{in}, \tilde{\dr}^{in}) (x) \in \mathbb{R} \times \mathbb{R}^N \times \mathbb{S}^{N-1} \times \mathbb{R}^N \, .$$

Before proving the local well-posedness of \eqref{compressible-Liquid-Crystal-Model}-\eqref{Initial-Data}, we should obtain the existence conclusion of the approximate system \eqref{Approx-Equat}, which is presented in the following lemma.

\begin{lemma} \label{lemma-5-1}
  Suppose that $s>\tfrac{N}{2}+1$ and the initial data $(\u^{in}, \dr^{in}, \tilde{\dr}^{in}) \in \R^N \times \mathbb{S}^{N-1} \times \R^N$ satisfy $\u^{in}, \nabla \dr^{in}, \tilde{\dr}^{in} \in H^s$. Then there is a maximal number $T^{*}_{n+1} > 0$ such that the approximate system \eqref{Approx-Equat} admits a unique solution $(\u^{n+1}, \dot{\dr}^{n+1}, \dr^{n+1})$ satisfying $\u^{n+1} \in C(0, T^{*}_{n+1}; H^s) \cap L^2 (0, T^{*}_{n+1}; H^{s+1})$, and $\nabla \dr^{n+1}, \dot{\dr}^{n+1} \in C(0, T^{*}_{n+1}; H^s)$.
\end{lemma}
\begin{proof}
  For the case $n+1$, the unknown vectors of the approximate system are $\rho^{n+1}$, $\u^{n+1}$, $\dr^{n+1}$, $\dot{\dr}^{n+1}$. The first equation of the approximate system is a linear equation about $\rho^{n+1}$, which have a unique solution $\rho^{n+1} \in C(0,\hat{T}_{n+1}; H^s) \cap L^2 (0, \hat{T}_{n+1}; H^{s+1})$ on the maximal time interval $[0, \hat{T}_{n+1})$. Substituting the solution $\rho^{n+1}$ of the first equation of the approximate system into the velocity equation of $\u^{n+1}$, which then to be a linear Stokes type system, therefore there exists a time $\check{T}_{n+1} \leq \hat{T}_{n+1}$, such that the second equation of the approximate system has a unique solution $\u^{n+1} \in C(0, \check{T}_{n+1}; H^s) \cap L^2(0, \check{T}_{n+1}; H^{s+1})$. Since $\rho^n, \u^n$ are the known, by Proposition \ref{Proposition-1}, we know that the orientation equation of $\dr^{n+1}$ admits the unique solution $\dr^{n+1}$ satisfying $\nabla \dr^{n+1}, \dot{\dr}^{n+1} \in C (0, \tilde{T}_{n+1}; H^s)$ on the maximal time interval $[0, \tilde{T}_{n+1})$. Denote the $T^{*}_{n+1} = \min \{\check{T}_{n+1}, \tilde{T}_{n+1} \}>0$, then we complete the proof of Lemma \ref{lemma-5-1}.
\end{proof}
\begin{remark}
  We remark that $T^{*}_{n+1} \leq T^{*}_{n}$.
\end{remark}

Secondly, we should get the uniform energy bound of the iterating approximate system \eqref{Approx-Equat} and the uniform positive lower bound $T$ of the time sequence $\{ T^{*}_{n+1} \}$ in the Lemma \ref{lemma-5-1}. Then by using the standard compactness arguments and Lemma \ref{lemma-d}, we can take limit for the system \eqref{Approx-Equat} as $n \!\to \!+\infty$. To do that, we can get the local well-posedness of the initial system \eqref{compressible-Liquid-Crystal-Model}-\eqref{Initial-Data}, and the time interval of the solution existence is $[0,T]$ . Before stating the lemma about the uniform energy bound of the system \eqref{Approx-Equat}, we define the energy functionals as follows:
\begin{align*}
  E_{n+1}(t) = & \mathcal{N}_s (\rho^{n+1}) + | \u^{n+1} |^2_{H^s_{\rho^n}} + | \dot{\dr}^{n+1} |^2_{H^s_{\rho^n}} + |\nabla \dr^{n+1}|^2_{H^s} \, , \\
  D_{n+1}(t) = & \tfrac {1}{2} \mu_4 |\nabla \u^{n+1}|^2_{H^s} + (\tfrac{1}{2} \mu_4 + \xi) |\div \u^{n+1}|^2_{H^s} + \mu_1 \sum_{k=0}^{s} |{\dr^n}^\top \nabla^k \A^{n+1} \dr^n|^2_{L^2} \\[1mm]
  & - \lambda_1 \sum_{k=0}^{s} (|\nabla^k \dot{\dr}^{n+1}|^2_{L^2} + |\nabla^k \B^{n+1} \dr^n + \tfrac{\lambda_2}{\lambda_1} \nabla^k \A^{n+1} \dr^n|^2_{L^2}) \\[1mm]
  & + (\mu_5 + \mu_6 + \tfrac{\lambda_2^2}{\lambda_1}) \sum_{k=0}^{s} |\nabla^k \A^{n+1} \dr^n|^2_{L^2} \, ,
\end{align*}
where
\begin{align*}
  \mathcal{N}_s (\rho^{n+1}) = \int_{\mathbb{R}^N} \tfrac{2 a}{\gamma -1} (\rho^{n+1})^\gamma \d x + \sum_{k=1}^{s} \int_{\mathbb{R}^N} \tfrac{p'(\rho^n)}{\rho^n} |\nabla^k \rho^{n+1}|^2 \d x \, ,
\end{align*}
and $| \u^{n+1} |^2_{H^s_{\rho^n}}$, $| \dr^{n+1} |^2_{H^s_{\rho^n}}$ are defined as $| \u |^2_{H^s_{\rho}}$, $| \dr |^2_{H^s_{\rho}}$.

Next, we are concerned with the following lemma which will give us an uniform existence time for the iterating solutions $(\vr^n, \u^n, \dr^n, \dot{\dr}^n)$. The lemma is articulated as
\begin{lemma} \label{lemma-5-2}
  Assume that $(\vr^{n+1}, \u^{n+1}, \dr^{n+1}, \dot{\dr}^{n+1})$ is the solution to the iterating approximate system \eqref{Approx-Equat}, for any fixed positive constant $M$, we define
  \begin{align*}
    T_{n+1} = \sup \Big\{ \tau \in [0, T^*_{n+1}); \sup_{t \in [0, \tau]} E_{n+1}(t) + \int_{0}^{\tau} D_{k+1} (t) \d t \leq M \Big\}\,,
  \end{align*}
  where $T^*_{n+1}>0$ is the existence time of the iterating approximating system \eqref{Approx-Equat}. Then for any fixed $M>E^{in}$, there is a constant time $T>0$ depending only on Leslie coefficients, $M$ and $E^{in}$, such that
  \begin{align*}
    T_{n+1} \geq T >0\,.
  \end{align*}
\end{lemma}
\begin{proof}
  It is obvious that $T_{n+1} >0$ by the continuity of the energy functional $E_{n+1}(t)$. We now consider the time sequence $\{T_n\}$, and prove that it admits a uniform positive lower bound. If the sequence $\{T_n\}$ is increasing, the conclusion immediately holds. Otherwise, if the time sequence is not increasing, we can choose a strictly increasing sequence $\{n_q\}_{q=1}^{\Lambda}$ as follows:
  \begin{align*}
    n_{q+1} = \min \{ n; n>n_q, T_n < T_{n_q} \} \ and \ n_1 = 1 \,.
  \end{align*}
  The number $\Lambda$ is finite or infinite. If $\Lambda<\infty$, the conclusion holds. Hence we only need to consider the case $\Lambda=\infty$. The definition of $n_q$ infers that the time sequence $\{ T_{n_q} \}$ is strictly decreasing, whose uniform lower bound is the same as the time sequence $\{T_n\}$, obviously. So the goal of us is to prove
  \begin{align*}
    \lim_{q \to \infty} T_{n_q} > 0\,.
  \end{align*}

  We now first to estimate the basic energy eatimate of the iterating approximating system \eqref{Approx-Equat}.
In order to get the $L^2$-estimate of the approximating system, multiplying the three equations of system \eqref{Approx-Equat} by $\tfrac{\gamma p(\rho^{n+1})}{(\gamma-1) \rho^{n+1}}$, $\rho^n \u^{n+1}$ and $\rho^n \dot{\dr}^{n+1}$, respectively, then similar calculation as the basic energy law in section \ref{Sect-2} gives us
  \begin{align}\label{AE-Basic}
    \no \tfrac{1}{2} \tfrac{\d}{\d t} & \int_{\mathbb{R}^N} \tfrac{2 a}{\gamma-1} (\rho^{n+1})^\gamma + \rho^n ( |\u^{n+1}|^2 + |\dot{\dr}^{n+1}|^2 ) + |\nabla \dr^{n+1}|^2 \d x \\
    \no & + \tfrac{1}{2} \mu_4 |\nabla \u^{n+1}|^2_{L^2} + ( \tfrac{1}{2} \mu_4 + \xi ) |\div \u^{n+1}|^2_{L^2} + \mu_1 |{\dr^n}^\top \A^{n+1} \dr^n|^2_{L^2} \\
    \no & - \lambda_1 ( |\dot{\dr}^{n+1}|^2_{L^2} + |\B^{n+1} \dr^n + \tfrac{\lambda_2}{\lambda_1} \A^{n+1} \dr^n|^2_{L^2} ) + (\mu_5 + \mu_6 + \tfrac{\lambda^2_2}{\lambda_1}) |\A^{n+1} \dr^n|^2_{L^2}\\
    \lesssim & \tilde{H} \, ,
  \end{align}
where
\begin{align*}
  \tilde{H} = & \big( \Q (\u^{n-2}) E_n^{\frac{1}{2}} + \Q (\u^{n-3}) E_{n-1}^{\frac{1}{2}} \big) \big( \Q (\u^{n-1}) + \Q (\u^{n-2}) E_n^{\frac{1}{2}} \big) \Q (\u^{n-1}) E_{n+1}\\[2mm]
  & + \Q (\u^{n-2}) E_n^{\frac{1}{2}} E_{n+1} + \Q (\u^{n-1}) E_{n+1} + \tfrac{1}{\sqrt{\mu_4}} \big[ D_n^{\frac{1}{2}} (t) E_{n+1} (t) + D_{n+1}^{\frac{1}{2}} (t) E_n (t) \big] \\[2mm]
  & + \tfrac{|\lambda_1| + |\lambda_2|}{\sqrt{\mu_4}} \big[ \Q (\u^{n-1}) D_n^{\frac{1}{2}} E_{n+1}^{\frac{1}{2}} + \Q (\u^{n-2}) D_{n+1}^{\frac{1}{2}} E_n^{\frac{1}{2}} \big] \, .
\end{align*}

We then turn to deal with the higher order estimate of the approximation system \eqref{Approx-Equat}. For $1 \leq k \leq s$, acting $\nabla^k$ on the three equations of the system \eqref{Approx-Equat}, multiplying them by $\tfrac{p'(\rho^n)}{\rho^n} \nabla^k \rho^{n+1}$, $\rho^n \nabla^k \u^{n+1}$ and $\rho^n \nabla^k \dot{\dr}^{n+1}$, respectively, then integrating the resulting equalities over $\mathbb{R}^n$ with respect to $x$ and adding up them together, one has
\begin{align}\label{AE-k}
  \no & \tfrac{1}{2} \tfrac{\d}{\d t} \int_{\mathbb{R}^N} \tfrac{p'(\rho^n)}{\rho^n} |\nabla^k \rho^{n+1}|^2 + \rho^n ( |\nabla^k \u^{n+1}|^2 + |\nabla^k \dot{\dr}^{n+1}|^2 ) + |\nabla^{k+1} \dr^{n+1}|^2 \d x \\[2mm]
  \no & + \tfrac{1}{2} \mu_4 |\nabla^{k+1} \u^{n+1}|^2_{L^2} + ( \tfrac{1}{2} \mu_4 + \xi ) |\nabla^k \div \u^{n+1}|^2_{L^2} - \lambda_1 |\nabla^k \dot{\dr}^{n+1}|^2_{L^2} + \mu_1 |{\dr^n}^\top \nabla^k \A^{n+1} \dr^n|^2_{L^2} \\[2mm]
  \no & - \lambda_1 |\nabla^k \B^{n+1} \dr^n + \tfrac{\lambda_2}{\lambda_1} \nabla^k \A^{n+1} \dr^n|^2_{L^2} + (\mu_5 + \mu_6 + \tfrac{\lambda^2_2}{\lambda_1}) |\nabla^k \A^{n+1} \dr^n|^2_{L^2} \\[2mm]
  & \lesssim \tilde{I} + \tilde{J}\,,
\end{align}
where
\begin{align}\label{tilde-I-dot}
  \no \tilde{I} = & \tfrac{1}{\sqrt{\mu_4}} \Q (\u^{n-1}) \big( 1 + \Q (\u^{n-2}) E_{n}^{\frac{1}{2}}(t) \big) E_{n+1}^{\frac{1}{2}}(t) D_{n+1}^\frac{1}{2}(t) + \tfrac{1}{\sqrt{\mu_4}} \Q (\u^{n-1}) \Q (\u^n) D_{n}^\frac{1}{2}(t) E_{n+1}^{\tfrac{1}{2}} (t) \\[2mm]
  \no & + \Q (\u^{n-1}) \Q (\u^{n-2}) E_n^{\frac{1}{2}}(t) ( 1 \!+\! E_n^{\frac{1}{2}}(t) ) E_{n+1}(t) \!+\! \Q (\u^{n-1}) \Big[ 1 + \Q (\u^{n-2}) E_n^{\frac{1}{2}}(t) \Big] E_{n+1}^2(t) \\[2mm]
  \no & + \tfrac{1}{\sqrt{\mu_4}} ( D_{n}^\frac{1}{2}(t) E_{n+1}(t) + D_{n+1}^\frac{1}{2}(t) E_n(t) ) + \tfrac{|\lambda_1| + |\lambda_2|}{\sqrt{\mu_4}} \Q (\u^{n-1}) D_{n}^{\frac{1}{2}} (t) \sum_{l=1}^{4} E_{n+1}^{\frac{l}{2}} (t) \\[2mm]
  & + \tfrac{\mu}{\sqrt{\mu_4}} \sum_{l=1}^{4} E_n^{\frac{l}{2}} (t) \Big[ \Q ( \u^{n-1}) E_{n+1}^{\frac{1}{2}} (t) + \Q (\u^{n-2}) E_n^{\frac{1}{2}}(t) \Big] D_{n+1}^{\frac{1}{2}} (t)
\end{align}
and
\begin{align}\label{tilde-J}
  \no \tilde{J} = & \Q(\u^{n-1}) \big( \Q(\u) E_n^{\frac{1}{2}}(t) + \Q(\u^{n-3}) E_{n-1}^{\frac{1}{2}}(t) \big) \big( 1 + ( \Q(\u^{n-2}) E_n^{\frac{1}{2}}(t)) \big) E_{n+1}(t) \\[2mm]
  \no & + \mu \Q ( \u^{n-1}) \P_s(|\rho^n|_{\dot{H}^s}) \Big[ \Q (\u^{n-2}) E_n^{\frac{1}{2}}(t) E_{n+1}^{\frac{1}{2}}(t) \big( 1 + E_n^{\frac{1}{2}}(t) + E_{n+1}^{\frac{1}{2}}(t) \big) + E_{n+1}(t) \Big] \\[2mm]
  \no & + \Q(\u^{n-1}) \P_s (|\rho^n|_{\dot{H}^s}) \Big[ E_n(t) E_{n+1}^{\frac{1}{2}}(t) + E_{n+1}(t) + E^{\frac{3}{2}}_{n+1}(t) + E_{n+1}^2(t) \Big] \\[2mm]
  \no & +  \mu \Q(\u^{n-1}) \P_s (|\rho^n|_{\dot{H}^s}) E_{n+1}(t) \sum_{l=0}^{4} E_n^{\frac{l}{2}}(t) \\[2mm]
  & + \tfrac{\mu + \xi}{\sqrt{\mu_4}} \Q(\u^{n-1}) \P_s (|\rho^n|_{\dot{H}^s}) \Big[ D_{n+1}^{\frac{1}{2}}(t) E_{n+1}^{\frac{1}{2}}(t) + D_n^{\frac{1}{2}}(t) \sum_{l=1}^{4} E_{n+1}^{\frac{l}{2}}(t) \Big] \, .
\end{align}
Consequently, summing up with $1 \leq k \leq s$ in \eqref{AE-k} and combining the result inequality with the $L^2$-estimate \eqref{AE-Basic} implies that
\begin{align} \label{Appro-Energy-Estimate}
  \tfrac{1}{2} \tfrac{\d}{\d t} E_{n+1}(t) + D_{n+1}(t) \lesssim \tilde{H} + \tilde{I}+ \tilde{J}
\end{align}
holds for all $t \in [0, T^*_{n+1})$. Recalling the definition of the sequence $\{ n_q \}$, we have that $T_N>T_{n_q}$ holds for any integer $N<n_q$. Taking $n=n_q-1$ in the inequality \eqref{Appro-Energy-Estimate}, then by the definition of $T_n$ we have
\begin{align*}
  E_k (t) + \int_{0}^{t} D_k(t) \d t \leq M
\end{align*}
holds for all the time $t \in [0, T_{n_q}]$ and $k \leq n_q$. Meanwhile, by the notation of $\Q(\u)$ we have that there exists a sufficiently large positive number $\Lambda$ and a constant $C$ such that
\begin{align*}
  \Q(\u^k) \leq \Q(M^{\frac{1}{2}}) \leq C \exp ( \Lambda M^{\frac{1}{2}} t )
\end{align*}
and
\begin{align*}
  \P_s(|\rho^k|_{\dot{H}^s}) \leq C \Q (\u^{k-2}) \P_s (E_k^{\frac{1}{2}} (t)) \leq C \Q (M) \P_s(M^{\frac{1}{2}}) \leq C \Q(M) \sum_{k=1}^{s} M^{\frac{k}{2}}\,.
\end{align*}
for all $t \in [0, T_{n_q}]$ and $k \leq n_q$. Then based on the inequality \eqref{Appro-Energy-Estimate} and the estimates of $\tilde{H}$, $\tilde{I}$ and $\tilde{J}$, we have
\begin{align}\label{Appro-Energy-Estimate-1}
  \tfrac{1}{2} \tfrac{\d}{\d t} E_{n_q}(t) + D_{n_q} (t)
  \leq  C \mathcal{F} (t) \,,
\end{align}
where
\begin{align*}
  \mathcal{F} (t) = G_1(M) (1\!+\!E_{n_q}^{\frac{1}{2}}) (t) D_{n_q}^{\frac{1}{2}} (t) \!+\! G_2(M) \sum_{k=1}^{4} E_{n_q}^{\frac{k}{2}} (t) D_{n_q-1}^{\frac{1}{2}} (t) \!+\! G_3(M) \!\sum_{k=0}^{4}\! E_{n_q}^{\frac{k}{2}} (t) \,.
\end{align*}
Here
\begin{align*}
  G_1(M) :=& ( 1+\Q(M) ) \sum_{k=0}^{s+2} M^{\frac{k}{2}} \,, \\
  G_2(M) :=& ( 1+\Q(M) ) \sum_{k=0}^{s} M^{\frac{k}{2}}\,, \\
  G_3(M) :=& \Q(M) \sum_{k=0}^{s+4} M^{\frac{k}{2}} \,.
\end{align*}
Integrating the inequality \eqref{Appro-Energy-Estimate-1} with respect to $t$ over $(0, t)$, and noticing the fact that for any $t \in (0, t_{n_q})$
\begin{align*}
  & C \int_{0}^{t} G_1(M) (1\!+\!E_{n_q}^{\frac{1}{2}}) (\tau) D_{n_q}^{\frac{1}{2}} (\tau) \d \tau \leq \tfrac{1}{2} \int_{0}^{t} D_{n_q} (\tau) \d \tau + \tfrac{1}{2} C^2 G_1^2(M) (1+M) t
\end{align*}
and
\begin{align*}
  \int_{0}^{t} \!\sum_{k=1}^{4}\! E_{n_q}^{\frac{k}{2}} (\tau) D_{n_q-1}^{\frac{1}{2}} (\tau) \d \tau
  \!\leq \!\sum_{k=1}^{4}\! M^{\frac{k}{2}} \Big(\! \int_{0}^{t} \!D_{n_q-1} (\tau) \d \tau \!\Big)^{\frac{1}{2}} t^{\frac{1}{2}}
  \!\leq\!\sum_{k=1}^{4}\! M^{\frac{k+1}{2}} t^{\frac{1}{2}}
\end{align*}
hold by using Cauchy inequality and H\"older inequality. Then we can infer that
\begin{align*}
  & E_{n_q} (t) + \int_{0}^{t} D_{n_q} (\tau) \d \tau \leq  \mathcal{H} (t) \,,
\end{align*}
where
\begin{align*}
  \mathcal{H} (t) := E^{in} + \Big( \tfrac{1}{2} C^2 G_1^2(M) (1+M) + C G_3 (M) \sum_{k=0}^{4} M^{\frac{k}{2}} \Big) t + C G_2(M) \sum_{k=1}^{4} M^{\frac{k+1}{2}} t^{\frac{1}{2}}\,.
\end{align*}
Obviously, $\mathcal{H}(t)$ is a strictly increasing and continuous function with respect to $t$, and $\mathcal{H}(0) = E^{in} \leq M$, then there exists a time $\hat{t}>0$, depending only on $M$, initial energy $E^{in}$ and Leslie coefficients such that
\begin{align*}
  E_{n_q} + \int_{0}^{t} D_{n_q} (\tau) \d \tau \leq M \,.
\end{align*}
holds for any $t \in [0, \hat{t}]$ and $p \in \mathbb{N}^+$. So we deduce that $T_{n_q} \geq \hat{t} > 0$ by the definition of $T_n$, hence $T=\lim_{q \to \infty} T_{n_q} \geq \hat{t} >0$. Consequently, we finish the proof of the Lemma \ref{lemma-5-2}.
\end{proof}

\textbf{Proof of theorem \ref{theorem-1}: Local well-posedness.}
Lemma \ref{lemma-5-2} implies that for any fixed $M>E^{in}$ there is a $T>0$ such that for all integer $n \geq 0$ and $t \in [0,T]$
\begin{align*}
  \sup_{t \in [0,T]} \!\Big( \!\mathcal{N}_{s} (\rho^{n+1}) \!+\! | \u^{n+1} |_{H^s_{\rho^n}}^2 \!+\! |\dot{\dr}^{n+1}|_{H^s_{\rho^n}}^2 \!+\! |\nabla \dr^{n+1}|_{H^s}^2 \!\Big) \!+\! \tfrac{1}{2} \mu_4 \!\int_{0}^{T}\! |\nabla \u^{n+1}|_{H^s}^2 (t) \d t \!\leq\! M \,.
\end{align*}
Then by compactness arguments and Lemma \ref{lemma-d}, we obtain that the system \eqref{compressible-Liquid-Crystal-Model} with the initial data \eqref{Initial-Data} admits a unique solution $(\rho, \u, \dr) \in \R \times \R^N \times \mathbb{S}^{N-1}$ satisfying
\begin{align*}
  \sup_{t \in [0,T]} \!\Big( \!\mathcal{N}_{s} (\rho) \!+\! | \u |_{H^s_{\rho}}^2 \!+\! |\dot{\dr}|_{H^s_{\rho}}^2 \!+\! |\nabla \dr|_{H^s}^2 \!\Big) \!+\! \tfrac{1}{2} \mu_4 \!\int_{0}^{T}\! |\nabla \u|_{H^s}^2 (t) \d t \!\leq\! M \,.
\end{align*}
Therefore the proof of the theorem \ref{theorem-1} is finished.

\section{Global classical solution} \label{Sect-6}

 The goal of this section is to prove existence of global classical solution $(\rho, \u, \dr)$ near the equilibrium $(1, 0 ,\dr_0)$ for the system \eqref{compressible-Liquid-Crystal-Model} in an additional coefficients constraint $\lambda_1 < 0$. Namely, we prove the global solution to \eqref{Expan-System} with small initial data. Before we proceed, what should be mentioned is that the main difficult in proving the well-posedness of system \eqref{Expan-System} near the equilibrium is lacking some dissipation by the observation of the $H^s$-estimate \eqref{A-Priori-Estimate-0}. So we need to find a new energy estimate which can give us enough dissipated terms. For any constants $\eta_1, \eta_2>0$, we present the energy functionals as follows:
 \begin{align*}
  \mathcal{E}_{\eta} (t) := & \int_{\R^N} \tfrac{p'(1+\vr)}{1+\vr} |\vr|^2 \d x + \sum_{k=1}^{s} \int_{\R^N} ( \tfrac{p'(1+\vr)}{1+\vr} - \eta_1 ) |\nabla^k \vr|^2 \d x \\
  & + \sum_{k=0}^{s-1} \int_{\R^N} ((1\!+\!\vr) - \eta_1) |\nabla^k \u|^2 + ((1\!+\!\vr) - \eta_2) |\nabla^{k+1} \dot{\dr}|^2 + (1-\eta_2) |\nabla^{k+1} \dr|^2 \d x \\
  & + \eta_1 |\u + \nabla \vr|_{H^{s-1}}^2 + \eta_2 |\dot{\dr} + \dr|_{\dot{H}^s}^2 + |\sqrt{1\!+\!\vr} \nabla^s \u|_{L^2}^2 + |\sqrt{1\!+\!\vr} \dot{\dr}|_{L^2}^2 + |\nabla^{s+1} \dr|_{L^2}^2 \\
  \mathcal{D}_{\eta} (t) := & \tfrac{3}{4} \eta_1 \sum_{k=1}^{s} \int_{\R^N} \tfrac{p'(1+\vr)}{1+\vr} |\nabla^k \vr|^2 \d x \!+\! \tfrac{3}{4} \eta_2 \sum_{k=1}^{s} \int_{\R^N} |\tfrac{1}{1+\vr} \nabla^{k+1} \dr|^2 \d x + ( \tfrac{1}{4} \mu_4 \!+\! \tfrac{1}{2} \xi ) |\div \u|_{H^s}^2 \\
  & + \tfrac{1}{4} \mu_4  |\nabla \u|_{H^s}^2 + (\mu_5+\mu_6+\tfrac{\lambda_2^2}{\lambda_1}) \sum_{k=0}^{s} |(\nabla^k \A) \dr|_{L^2}^2 + \mu_1 \sum_{k=0}^{s} |\dr^\top (\nabla^k \A) \dr|_{L^2}^2 \\
  & - \tfrac{\lambda_1}{2} \sum_{k=0}^{s} |\nabla^k \dot{\dr} + (\nabla^k \B) \dr + \tfrac{\lambda_2}{\lambda_1} (\nabla^k \A) \dr|_{L^2}^2 \,.
 \end{align*}
With the above energy functionals in hands, we then articulate the following lemma, which plays an key role in proving the global well-posedness of the system \eqref{Expan-System}.
\begin{lemma}\label{lemma-6-1}
  There exists a small constant $\eta_0>0$, depending only on Leslie coefficients, such that if $(1+\vr, \u, \dr)$ is the local solution constructed in Theorem \ref{theorem-1}, then for all $0<\eta_1, \eta_2 \leq \eta_0$
  \begin{align*}
    \tfrac{1}{2} \tfrac{\d}{\d t} \mathcal{E}_{\eta} (t) + \mathcal{D}_{\eta} (t) \leq C \mathcal{D}_{\eta} (t) \sum_{k=1}^{s+4} \mathcal{E}_{\eta}^{\frac{k}{2}} (t) \,,
  \end{align*}
  where the positive constant $C$ depends only on the Leslie coefficients.
\end{lemma}
\begin{proof}
This proof can be divided into four parts. Firstly, we get the $L^2$-estimate of \eqref{Expan-System}, which is different with the {\it basic energy law} in section \ref{Sect-2}. Secondly, the higher order estimate can be obtained by using the similar way as the {\em a priori} estimate in Section \ref{Sect-3}. Thirdly, we can get some dissipation of the density $\vr$ by multiplying $\nabla^k \eqref{Expan-System}_2$ with $\nabla^{k+1} \rho$ $(0 \leq k \leq s-1)$. Lastly, multiplying the $k$-order derivative of the third equation of \eqref{Expan-System} with $\nabla \dr$ $(1 \leq k \leq s)$ gives us more dissipation about $\nabla \dr$.

{\it Step 1. Basic energy estimate.} For the estimate of the $L^2$-estimate of the system \eqref{Expan-System}, multiplying the three equations of \eqref{Expan-System} with $\tfrac{p'(1+\vr)}{1+\vr} \vr$, $(1+\vr) \u$ and $(1+\vr) \dot{\dr}$, respectively, then adding the result identities together and taking $L^2$-inner product with respect to $x$ over $\R^N$, we have
\begin{align}\label{Expan-L-2-0}
  \no & \tfrac{1}{2} \tfrac{\d}{\d t} \int_{\R^N} \tfrac{p'(1+\vr)}{1+\vr} |\vr|^2 + (1+\vr) ( |\u|^2 + |\dot{\dr}|^2 ) + |\nabla \dr|^2 \d x + \tfrac{1}{2} \mu_4 |\nabla \u|^2_{L^2} + (\tfrac{1}{2} \mu_4 + \xi) |\div \u|^2_{L^2} \\
  \no & + \l p'(1+\vr) \nabla \vr, \u \r + \l p'(1+\vr) \vr, \div \u \r + \l \Delta \dr, \u \cdot \nabla \dr \r + \l \Sigma_2, \nabla \u \r + \l \Sigma_3, \nabla \u \r \\
  & - \lambda_1 |\dot{\dr}|^2_{L^2} - \l \lambda_1 \B \dr + \lambda_2 \A \dr, \dot{\dr} \r
  = \tfrac{1}{2} \int_{\R^N} \Big[ \p_t \Big( \tfrac{p'(1+\vr)}{1+\vr} \Big) + \div  \Big( \tfrac{p'(1+\vr)}{1+\vr} \u \Big) \Big] |\vr|^2 \d x\,.
\end{align}
We now deal with the above identity. Direct calculation enables us to get
\begin{align*}
  & \l p'(1+\vr) \vr, \div \u \r + \l p'(1+\vr) \nabla \vr, \u \r
  = - \l p''(1+\vr) \nabla \vr \vr, \u \r
  \lesssim |\nabla \vr|_{L^2} |\vr|_{L^4} |\u|_{L^4} \\
  \lesssim & |\nabla \vr|_{L^2} ( |\vr|_{L^2} |\u|_{L^2} )^{1-\frac{N}{4}} ( |\nabla \vr|_{L^2} |\nabla \u|_{L^2} )^{\frac{N}{4}}
  \leq |\nabla \vr|_{L^2} ( |\nabla \vr|_{L^2} |\nabla \u|_{L^2} )^{\frac{1}{2}} ( |\vr|_{H^1} |\u|_{H^1} )^{\frac{1}{2}} \,,
\end{align*}
where we take use of H\"older inequality and the interpolation inequality: $|f|_{L^4(\R^N)} \lesssim |f|^{1-\frac{N}{4}}_{L^2(\R^N)} |f|^{\frac{N}{4}}_{L^2(\R^N)}$ for any $f \in H^1 (\R^N) \ (N=2,3)$. For the term on the right-hand side of \eqref{Expan-L-2-0}, by using the first equation of \eqref{Expan-System}, we can easily get \begin{align*}
  \tfrac{1}{2} & \int_{\R^N} \! \Big[ \p_t \Big( \tfrac{p'(1+\vr)}{1+\vr} \Big) \!+\! \div  \Big( \tfrac{p'(1+\vr)}{1+\vr} \u \Big) \Big] |\vr|^2 \d x
  = \tfrac{1}{2} \int_{\R^N} \! \Big[ \tfrac{p'(1+\vr)}{1+\vr} \!-\! \Big( \tfrac{p'(1+\vr)}{1+\vr} \Big)' (1+\vr) \Big] \div \u |\vr|^2 \d x \\
   \lesssim & |\div \u|_{L^2} |\vr|_{L^4}^2 \lesssim |\nabla \u|_{L^2} |\nabla \vr|_{L^2} |\vr|_{H^1} \,.
\end{align*}
The other terms in \eqref{Expan-L-2-0} can be dealt with by using the same calculation of the basic energy law in Section 2. Hence we have
\begin{align}\label{Expan-L-2}
  \no & \tfrac{1}{2} \tfrac{\d}{\d t} \int_{\R^N} \tfrac{p'(1+\vr)}{1+\vr} |\vr|^2 + (1+\vr) ( |\u|^2 + |\dot{\dr}|^2 ) + |\nabla \dr|^2 \d x + \tfrac{1}{2} \mu_4 |\nabla \u|^2_{L^2} + (\tfrac{1}{2} \mu_4 + \xi) |\div \u|^2_{L^2} \\
  \no & + \mu_1 |\dr^\top \A \dr|_{L^2}^2 - \lambda_1 |\dot{\dr} + \B \dr + \tfrac{\lambda_2}{\lambda_1}|_{L^2}^2 + ( \mu_5 + \mu_6 + \tfrac{\lambda_2^2}{\lambda_1} ) |\A \dr|_{L^2}^2 \\
  \lesssim & |\nabla \u|_{L^2} |\nabla \vr|_{L^2} |\vr|_{H^1} + |\nabla \vr|_{L^2} ( |\nabla \vr|_{L^2} |\nabla \u|_{L^2} )^{\frac{1}{2}} ( |\vr|_{H^1} |\u|_{H^1} )^{\frac{1}{2}} \,.
\end{align}

 {\it Step 2. Higher order estimate for $\vr$, $\u$ and $\dot{\dr}$.} For the higher order estimate, we can use the same method as the {\em a priori} estimate in Section \ref{Sect-3}, the different is that here $1+\vr$ has the uniform bounds since $(\rho, \u, \dr)$ is near the equilibrium state $(1, 0, \dr_0)$. Hence similar calculation as the {\em a priori} estimate in Section \ref{Sect-3} tells we that
\begin{align}\label{Expan-H-k}
  \no & \tfrac{1}{2} \tfrac{\d}{\d t} \int_{\mathbb{R}^N} \tfrac{p'(1+\vr)}{1+\vr} |\nabla^k \vr|^2 + (1+\vr) ( |\nabla^k \u|^2 \!+\! |\nabla^k \dot{\dr}|^2 ) + |\nabla^{k+1} \dr|^2 \d x + \tfrac{1}{2} \mu_4 |\nabla^{k+1} \u|_{L^2}^2 \\
  \no & + ( \tfrac{1}{2} \mu_4 + \xi ) |\nabla^k \div \u|_{L^2}^2 + \mu_1 |\dr^\top (\nabla^k \A) \dr|^2_{L^2} - \lambda_1 |\nabla^k \dot{\dr} + (\nabla^k \B) \dr + \tfrac{\lambda_2}{\lambda_1} (\nabla^k \A) \dr |^2_{L^2} \\
  & + (\mu_5 + \mu_6 + \tfrac{\lambda_2^2}{\lambda_1}) |(\nabla^k \A) \dr|^2_{L^2}
  \lesssim \hat{I} + \hat{J} \,,
\end{align}
where
  \begin{align*}
    \hat{I} = & |\vr|^2_{\dot{H}^s} ( 1 + |\vr|_{\dot{H}^s} ) |\u|_{H^s} + |\nabla \dr|^2_{\dot{H}^s} |\nabla \u|_{H^s} + ( 1 + |\vr|_{\dot{H}^s} ) |\nabla \dr|_{H^s} |\dot{\dr}|^3_{H^s} \\
    & + |\nabla \dr|^2_{H^s} |\nabla \dr|_{\dot{H}^s} |\dot{\dr}|_{H^s} + \mu \sum_{l=1}^{4} |\nabla \dr|^l_{H^s} ( |\u|_{\dot{H}^s} + |\dot{\dr}|_{H^s} ) |\nabla \u|_{H^s} \, .
  \end{align*}
and
  \begin{align*}
    \hat{J} \lesssim & |\u|_{\dot{H}^s} (|\u|^2_{\dot{H}^s} + |\dot{\dr}|^2_{H^s}) + \P_s(|\vr|_{\dot{H}^s}) ( |\vr|_{\dot{H}^s} |\u|_{\dot{H}^s} + |\nabla \dr|_{\dot{H}^s} |\dot{\dr}|_{H^s} ) + ( 1 + |\vr|_{\dot{H}^s} ) |\vr|^2_{\dot{H}^s} |\u|_{H^s} \\[2mm]
    & + (\mu_4 + \xi) \P_s (|\vr|_{\dot{H}^s}) |\nabla \u|_{\dot{H}^s}  |\u|_{\dot{H}^s} + \P_s (|\vr|_{\dot{H}^s}) |\u|_{\dot{H}^s} |\nabla \dr|_{H^s} |\nabla \dr|_{\dot{H}^s} \\
    & + \mu \P_s (|\vr|_{\dot{H}^s}) |\u|_{\dot{H}^s} ( |\nabla \u|_{H^s} + |\dot{\dr}|_{H^s} ) \sum_{l=0}^{4} |\nabla \dr|^l_{H^s} + |\lambda_1| \P_s (|\vr|_{\dot{H}^s}) |\dot{\dr}|^2_{H^s} \\
    & + \P_s (|\vr|_{\dot{H}^s}) \Big\{ ( |\dot{\dr}|^2_{H^s} + |\nabla \dr|^2_{H^s} ) (1 + |\nabla \dr|_{H^s}) + |\lambda_2| |\u|_{\dot{H}^s} \sum_{l=0}^{3} |\nabla \dr|^l_{H^s} \Big\} \, .
  \end{align*}
  Then summing up with $1 \leq k \leq s$ to \eqref{Expan-H-k}, and combining the $L^2$-estimate \eqref{Expan-L-2} of \eqref{Expan-System}, we have
  \begin{align}\label{Expan-H-s}
  \no & \tfrac{1}{2} \Big[ |\vr|^2_{H^s_{\frac{p'(1+\vr)}{1+\vr}}} + ( |\nabla \u|^2_{H^s_{1+\vr}} \!+\! |\dot{\dr}|^2_{H^s_{1+\vr}} ) + |\nabla \dr|^2_{H^s} \Big] + \tfrac{1}{2} \mu_4 |\nabla \u|_{H^s}^2 + ( \tfrac{1}{2} \mu_4 + \xi ) |\div \u|_{H^s}^2 \\
  \no & + \mu_1 \sum_{k=0}^{s} |\dr^\top (\nabla^k \A) \dr|^2_{L^2} - \lambda_1 \sum_{k=0}^{s} |\nabla^k \dot{\dr} + (\nabla^k \B) \dr + \tfrac{\lambda_2}{\lambda_1} (\nabla^k \A) \dr |^2_{L^2} \\
  \no & + (\mu_5 + \mu_6 + \tfrac{\lambda_2^2}{\lambda_1}) \sum_{k=0}^{s} |(\nabla^k \A) \dr|^2_{L^2} \\
  \lesssim & \hat{I} + \hat{J} + |\nabla \u|_{L^2} |\nabla \vr|_{L^2} |\vr|_{H^1} + |\nabla \vr|_{L^2} ( |\nabla \vr|_{L^2} |\nabla \u|_{L^2} )^{\frac{1}{2}} ( |\vr|_{H^1} |\u|_{H^1} )^{\frac{1}{2}} \,.
\end{align}
{\it Step 3. The estimate for dissipation of density.} We should get some dissipation of $\vr$. Since $p'(1+\vr)>0$, then we can use the term $\tfrac{p'(1+\vr)}{1+\vr} \nabla \vr$ to get the dissipation. To do it, for all $1 \leq k \leq s-1$, taking $\nabla^k$ on the second equation of \eqref{Expan-System}, multiplying $\nabla^{k+1} \vr$ by integration by parts, we obtain
  \begin{align}\label{Dissip-rho-0}
    \no & \l \p_t \nabla^k \u, \nabla^{k+1} \vr \r + \l \tfrac{p'(1+\vr)}{1+\vr} \nabla^{k+1} \vr, \nabla^{k+1} \vr \r \\
    \no = & - \l \nabla^k (\u \cdot \nabla \u), \nabla^{k+1} \vr \r - \l [\nabla^k, \tfrac{p'(1+\vr)}{1+\vr} \nabla] \vr, \nabla^{k+1} \vr \r \\
    \no & + \l \nabla^k ( (\tfrac{1}{1 + \vr}) \div \Sigma_1 ), \nabla^{k+1} \vr \r + \l \nabla^k ( (\tfrac{1}{1 + \vr}) \div \Sigma_2 ),   \nabla^{k+1} \vr \r \\
    \no & + \l \nabla^k ( (\tfrac{1}{1 + \vr}) \div \Sigma_3 ), \nabla^{k+1} \vr \r \\
    \equiv & R_1 + R_2 + R_3 + R_4 + R_5 \, .
  \end{align}

For the first term on the left-hand side of \eqref{Dissip-rho-0}, by using the first equation of \eqref{Expan-System}, it is easy to get
\begin{align}\label{Dissip-rho-1}
   \no & \l \p_t \nabla^k \u, \nabla^{k+1} \vr \r = \p_t \l \nabla^k \u, \nabla^{k+1} \vr \r - \l \nabla^k \u, \nabla^{k+1} \p_t \vr \r \\
   \no = & \p_t \l \nabla^k \u, \nabla^{k+1} \vr \r + \l \nabla^k \u, \nabla^{k+1} (\u \cdot \nabla \vr + (1+\vr) \div \u) \r \\
   \no = & \p_t \l \nabla^k \u, \nabla^{k+1} \vr \r - \l \nabla^k \div \u, \u \cdot \nabla^{k+1} \vr + (1+\vr) \nabla^k \div \u \r \\
   & - \l [\nabla^k, \u \cdot \nabla] \vr + [\nabla^k, (1+\vr) \div] \u, \nabla^k \div \u \r \, .
\end{align}

Straightforward calculation tells us that
\begin{align}\label{Dissip-rho-2}
  \no & - \l \nabla^k \div \u, \u \cdot \nabla^{k+1} \vr + (1+\vr) \nabla^k \div \u \r \\
  \no \leq &  ( |\u|_{L^\infty} |\nabla^{k+1} \vr|_{L^2} + |1+\vr|_{L^\infty} |\nabla^k \div \u|_{L^2} ) |\nabla^k \div \u|_{L^2} \\
  \leq & | 1+\vr |_{L^\infty} | \div \u |_{H^s}^2 + C |\u|_{H^s} |\div \u|_{H^s} |\vr|_{\dot{H}^s}
\end{align}
and
\begin{align}\label{Dissip-rho-3}
  \no & - \l [\nabla^k, \u \cdot \nabla] \vr + [\nabla^k, (1+\vr) \div] \u, \nabla^k \div \u \r \\
  \no \leq & \big(  |\nabla \u|_{\infty} |\nabla^k \vr|_{L^2} + |\nabla^k \u|_{L^2} |\nabla \vr|_{L^\infty} \big) |\nabla^{k+1} \u|_{L^2} \\
  \leq & |\vr|_{\dot{H}^s} |\nabla \u|^2_{H^s} \, .
\end{align}

Then by using H\"older inequality and Sobolev embedding theory, we have the estimates of $R_i (i=1,...,5)$ in the following.

For the estimate of $R_1$ and $R_2$, and note that $1 \leq k \leq s-1$, we can get
\begin{align}\label{Dissip-rho-R-1}
  R_1 = - \l \nabla^k (\u \cdot \nabla \u), \nabla^{k+1} \vr \r
  \leq \sum_{a+b=k} |\nabla^a \u|_{L^4} |\nabla^{b+1} \u|_{L^4} |\nabla^{k+1} \vr|_{L^2} \lesssim |\u|_{H^s} |\nabla \u|_{H^s} |\vr|_{\dot{H}^s}
\end{align}
and
\begin{align}\label{Dissip-rho-R-2}
  \no R_2 = & - \l [\nabla^k, \tfrac{p'(1+\vr)}{1+\vr} \nabla] \vr, \nabla^{k+1} \vr \r  \\
  \no \lesssim & ( |\nabla (\tfrac{p'(1+\vr)}{1+\vr})|_{L^\infty} |\nabla^k \vr|_{L^2} + |\nabla^k ( \tfrac{p'(1+\vr)}{1+\vr} )|_{L^2} |\nabla \vr|_{L^\infty} ) |\nabla^{k+1} \vr|_{L^2} \\
  \no \lesssim & ( |\vr|_{\dot{H}^s} |\vr|_{\dot{H}^s} + \P_{s-1} (|\vr|_{\dot{H}^s}) |\vr|_{\dot{H}^s} ) |\vr|_{\dot{H}^s} \\
  \lesssim & \P_{s-1} (|\vr|_{\dot{H}^s}) |\vr|^2_{\dot{H}^s} \, .
\end{align}
by utilizing H\"older inequality, Sobolev embedding theory and Moser-type inequality.

 As to the estimate of $R_3$, we need only to estimate $\l \tfrac{\mu_4}{2} \nabla^k (\tfrac{1}{1+\vr} \Delta \u), \nabla^{k+1} \vr \r$. it is easy to derive that
\begin{align*}
  & \l \tfrac{\mu_4}{2} \nabla^k (\tfrac{1}{1 + \vr} \Delta \u), \nabla^{k+1} \vr \r
  =\tfrac 1 2 \mu_4 \l \tfrac{1}{1 + \vr} \nabla^k \Delta \u, \nabla^{k+1} \vr \r + \tfrac 1 2 \mu_4 \l [ \nabla^k, \tfrac{1}{1 + \vr} \Delta ] \u, \nabla^{k+1} \vr \r \\
  \lesssim & \mu_4 |\nabla^{k+2} \u|_{L^2} |\nabla^{k+1} \vr|_{L^2} + \mu_4 ( |\nabla (\tfrac{1}{1+\vr})|_{L^\infty} |\nabla^{k+1} \u|_{L^2} + |\nabla^k (\tfrac{1}{1+\vr})|_{L^2} |\Delta \u|_{L^\infty} ) |\nabla^{k+1} \vr|_{L^2} \\
  \lesssim & \mu_4 (1 + \P_{s-1} (|\vr|_{\dot{H}^{s}})) |\nabla \u|_{\dot{H}^s} |\vr|_{\dot{H}^s} \, .
\end{align*}
And another term in $R_3$ can be controlled by the same way, then we have
\begin{align}\label{Dissip-rho-R-3}
  R_3 \lesssim (\mu_4 + \xi)(1 + \P_{s-1} (|\vr|_{\dot{H}^{s}})) |\nabla \u|_{\dot{H}^s} |\vr|_{\dot{H}^s} \, .
\end{align}

For the estimate of the term $R_4$, since the two parts of it have the same structure, we need only to estimate $\l \nabla^k ( \tfrac{1}{1 + \vr} \div (\nabla \dr \odot \nabla \dr)), \nabla^{k+1} \vr \r$. Straightforward calculation gives us
\begin{align*}
  & \l \nabla^k (\tfrac{1}{1+\vr} \div (\nabla \dr \odot \nabla \dr) ), \nabla^{k+1} \vr \r \\
  = & \l \tfrac{1}{1+\vr} \nabla^k \div (\nabla \dr \odot \nabla \dr), \nabla^{k+1} \vr \r + \l [\nabla^k, \tfrac{1}{1+\vr} \div] (\nabla \dr \odot \nabla \dr), \nabla^{k+1} \vr \r \\
  \lesssim & 2 \l |\nabla \dr| |\nabla^{k+2} \dr|, |\nabla^{k+1} \vr| \r + \sum_{\substack{a+b=k+1,\\ a, b \geq 1}} \l |\nabla^{a+1} \dr| |\nabla^{b+1} \dr|, |\nabla^{k+1} \vr| \r \\
  & \!+\! ( |\nabla(\tfrac{1}{1+\vr})|_{L^\infty} |\nabla^{k-1} \div (\nabla \dr \!\odot\! \nabla \dr)|_{L^2} \!+\! |\nabla^k (\tfrac{1}{1+\vr})|_{L^2} |\div(\nabla \dr \!\odot\! \nabla \dr)|_{L^\infty} ) |\nabla^{k+1} \vr|_{L^2} \\
  \lesssim & |\nabla \dr|_{L^\infty} |\nabla^{k+2} \dr|_{L^2} |\nabla^{k+1} \vr|_{L^2} + \sum_{\substack{a+b=k+1,\\ a, b \geq 1}} |\nabla^{a+1} \dr|_{L^4} |\nabla^{b+1} \dr|_{L^4} |\nabla^{k+1} \vr|_{L^2} \\
  & + |\nabla \vr|_{L^\infty} ( |\nabla \dr|_{L^\infty} |\nabla^{k+1} \dr|_{L^2} + \sum_{\substack{a+b=k,\\ a, b \geq 1}} |\nabla^{a+1} \dr|_{L^4} |\nabla^{b+1} \dr|_{L^4} ) |\nabla^{k+1} \vr|_{L^2} \\
  & + \P_{s-1} (|\vr|_{\dot{H}^s}) |\nabla^{k+1} \vr|_{L^2} |\nabla \dr|_{L^\infty} |\nabla^2 \dr|_{L^\infty} \\
  \lesssim & |\vr|_{\dot{H}^s} |\nabla \dr|_{\dot{H}^s} |\nabla \dr|_{H^s} ( 1 + \P_{s-1} (|\vr|_{\dot{H}^s}) ) \, .
\end{align*}
Then we can infer that
\begin{align}\label{Dissip-rho-R-4}
  R_4 \lesssim |\vr|_{\dot{H}^s} |\nabla \dr|_{\dot{H}^s} |\nabla \dr|_{H^s} ( 1 + \P_{s-1} (|\vr|_{\dot{H}^s}) ) \, .
\end{align}

We now estimate the last term $R_5$. By using the Lemma \ref{Moser}, one has
\begin{align}\label{Dissip-rho-R-5-1}
  \no R_5 = & \l \tfrac{1}{1+\vr} \nabla^k \div \sigma, \nabla^{k+1} \vr \r + \l [ \nabla^k, \tfrac{1}{1+\vr} \div ] \sigma, \nabla^{k+1} \vr \r \\
  \no \leq & |\tfrac{1}{1+\vr}|_{L^\infty} |\nabla^k \div \sigma|_{L^2} |\nabla^{k+1} \vr|_{L^2} + ( |\nabla (\tfrac{1}{1+\vr})|_{L^\infty} |\nabla^{k-1} \div \sigma|_{L^2} \\
  & + |\nabla^k (\tfrac{1}{1+\vr})|_{L^2} |\div \sigma|_{L^\infty} ) |\nabla^{k+1} \vr|_{L^2} \, .
\end{align}
So if we want get the bound of $R_5$, we should get the bounds of $|\nabla^k \div \sigma|_{L^2}$ and $|\div \sigma|_{L^\infty}$. In order to get the bounds of $|\nabla^k \div \sigma|_{L^2}$ and $|\div \sigma|_{L^\infty}$, we just need to estimate the following two terms:
\begin{align*}
  & | \nabla^k (\dr_i \dr_j \dr_p \dr_q \p_i \u_j) |_{L^2} \\
  \leq & | \nabla^{k+1} \u |_{L^2} + \sum_{\substack{a+b = k,\\ a \geq 1}} \sum_{\substack{a_1 + a_2 + a_3 + a_4 = a, \\ a_1, a_2, a_3, a_4 \geq 1}} |\nabla^{a_1} \dr \nabla^{a_2} \dr \nabla^{a_3} \dr \nabla^{a_4} \dr \nabla^{b+1} \u |_{L^2} \\
  & + C^1_4 \sum_{\substack{a+b = k, \\ a \geq 1}} \sum_{\substack{a_1 + a_2 + a_3 = a, \\ a_1, a_2, a_3 \geq 1}} |\nabla^{a_1} \dr \nabla^{a_2} \dr \nabla^{a_3} \dr \nabla^{b+1} \u |_{L^2} \\
  & \!+\! C^1_4 C^1_3 \!\sum_{\substack{a+b = k, \\ a \geq 1}}\! \sum_{\substack{a_1 + a_2  = a, \\ a_1, a_2 \geq 1}} |\nabla^{a_1} \dr \nabla^{a_2} \dr \nabla^{b+1} \u|_{L^2} \!+\! C^1_4 C^1_3 C^1_2 \!\sum_{\substack{a+b = k, \\ a \geq 1}}\! |\nabla^a \dr \nabla^{b+1} \u|_{L^2} \\
  \lesssim & |\nabla \u|_{H^s} \sum^{4}_{l=0} |\nabla \dr|^l_{H^s}
\end{align*}
for $1 \leq k \leq s$, and
\begin{align*}
  |\div (\dr_i \dr_j \dr_p \dr_q \p_i \u_j)|_{L^\infty} \lesssim & |\nabla^2 \u|_{L^\infty} + |\nabla \dr|_{L^\infty} |\nabla \u|_{L^\infty}
  \lesssim |\nabla \u|_{H^s} (1 + |\nabla \dr|_{H^s}) \, .
\end{align*}
The other terms can be controlled by the same way, then we have
\begin{align*}
  |\nabla^k \div \sigma|_{L^2}
  \lesssim \mu \big( |\nabla \u|_{H^s} \sum_{l=0}^{4} |\nabla \dr|_{H^s}^l + |\dot{\dr}|_{H^s} ( |\nabla \dr|_{H^s} + 1 ) \big)
\end{align*}
and
\begin{align*}
  |\div \sigma|_{L^\infty} \leq \mu \big( |\nabla \u|_{H^s} + |\dot{\dr}|_{H^s} \big) \big( |\nabla \dr|_{H^s} + 1 \big) \, .
\end{align*}
Combining with the above two bounds and \eqref{Dissip-rho-R-5-1}, we obtain
\begin{align}\label{Dissip-rho-R-5}
  R_5 \lesssim & \mu |\vr|_{\dot{H}^s} (1+\P_{s-1} (|\vr|_{\dot{H}^s})) \big( |\nabla \u|_{H^s} \sum_{l=0}^{4} |\nabla \dr|^l_{H^s} + |\dot{\dr}|_{H^s} (1 + |\nabla \dr|_{H^s}) \big) \,.
\end{align}

For the case $k =0$, it is easy to derive that
\begin{align}\label{Dissip-rho-k-0}
  \no & \tfrac{1}{2} \tfrac{\d}{\d t} ( |\u + \nabla \vr|_{L^2}^2 - |\u|_{L^2}^2 - |\nabla \vr|_{L^2}^2 ) + \int_{\mathbb{R}^N} \tfrac{p'(1+\vr)}{1+\vr} |\nabla \vr|^2 \d x - |1+\vr|_{L^\infty} |\div \u|_{L^2}^2 \\
  \no \lesssim & |\nabla \u|_{L^2} |\u|_{L^\infty} |\nabla \vr|_{L^2} + ( \mu_4 + \xi ) |\nabla^2 \u|_{L^2} |\nabla \vr|_{L^2} + |\nabla \dr|_{L^\infty} |\nabla^2 \dr|_{L^2} |\nabla \vr|_{L^2} \\
  & + \mu |\nabla \vr|_{L^2} ( |\nabla \u|_{H^s} + |\dot{\dr}|_{H^s} ) ( 1 + |\nabla \dr|_{H^s} )\,.
\end{align}

Noticing that the first term on the right-hand side of \eqref{Dissip-rho-1} can be rewritten as
\begin{equation}\label{Dissop-rho-4}
  \p_t \l \nabla^k \u, \nabla^{k+1} \vr \r
  = \tfrac {1}{2} \tfrac{\d}{\d t} ( |\nabla^k (\u + \nabla \vr)|^2_{L^2} - |\nabla^k \u|^2_{L^2} - |\nabla^{k+1} \vr|^2_{L^2} )\, .
\end{equation}
Consequently, substituting \eqref{Dissip-rho-1}-\eqref{Dissip-rho-3}, \eqref{Dissip-rho-R-1}-\eqref{Dissip-rho-R-5} and \eqref{Dissop-rho-4} into \eqref{Dissip-rho-0}, then summing up with $1 \leq k \leq s-1$, and combining with \eqref{Dissip-rho-k-0} yields that
\begin{align}\label{Dissip-rho-1}
  & \tfrac 1 2 \tfrac{\d}{\d t} ( | \u + \nabla \vr|^2_{H^{s-1}} - |\u|^2_{H^{s-1}} - |\nabla \vr|^2_{H^{s-1}} ) + \sum_{k=1}^{s} \int_{\mathbb{R}^N} \tfrac{p'(1+\vr)}{1+\vr} |\nabla^k \vr|^2 \d x \\
  \no & - |1+\vr|_{L^\infty} |\div \u|_{H^s}^2 - \hat{C} (\mu+\xi) |\vr|_{\dot{H}^s} ( |\nabla \u|_{H^s} + |\dot{\dr}|_{H^s} ) \\
  \no \lesssim & |\nabla \u|_{H^s} |\vr|_{\dot{H}^s} ( |\u|_{H^s} \!+\! |\nabla \u|_{H^s} )
  \!+\! \P_{s-1} (|\vr|_{\dot{H}^s}) |\vr|_{\dot{H}^s} ( |\vr|_{\dot{H}^s} + |\nabla \u|_{H^s} + |\dot{\dr}|_{H^s} )\\
  \no & + \! |\vr|_{\dot{H}^s} (1\! +\! \P_{s-1} (|\vr|_{\dot{H}^s})) \Big[  (\mu \! + \! \xi) \Big( |\nabla \u|_{H^s} \sum_{l=1}^{4}\!  |\nabla \dr|^l_{H^s} \! + \! |\dot{\dr}|_{H^s} |\nabla \dr|_{H^s} \Big) \! +\!  |\nabla \dr|_{H^s} |\nabla \dr|_{\dot{H}^s} \Big]\, .
\end{align}
Noticing that
\begin{align*}
  & \hat{C} (\mu+\xi) |\vr|_{\dot{H}^s} ( |\nabla \u|_{H^s} + |\dot{\dr}|_{H^s} ) \\
  \leq & \tfrac{1}{4} \sum_{k=1}^{s} \int_{\mathbb{R}^N} \tfrac{p'(1+\vr)}{1+\vr} |\nabla^k \vr|^2 \d x + C (\mu^2 + \xi^2) |\tfrac{1+\vr}{p'(1+\vr)}|_{L^\infty} ( |\nabla \u|_{H^s}^2 + |\dot{\dr}|_{H^s}^2 ) \,,
\end{align*}
then we have
\begin{align}\label{Dissip-rho}
  \no & \tfrac 1 2 \tfrac{\d}{\d t} ( | \u + \nabla \vr|^2_{H^{s-1}} - |\u|^2_{H^{s-1}} - |\nabla \vr|^2_{H^{s-1}} ) + \tfrac{3}{4} \sum_{l=1}^{s} \int_{\mathbb{R}^N} \tfrac{p'(1+\vr)}{1+\vr} |\nabla^l \vr|^2 \d x \\
  \no & - |1+\vr|_{L^\infty} |\div \u|_{H^s}^2 - C (\mu^2+\xi^2) |\tfrac{1+\vr}{p'(1+\vr)}|_{L^\infty} ( |\nabla \u|_{H^s}^2 + |\dot{\dr}|_{H^s}^2 ) \\
  \no \lesssim & |\nabla \u|_{H^s} |\vr|_{\dot{H}^s} ( |\u|_{H^s} \!+\! |\nabla \u|_{H^s} )
  \!+\! \P_{s-1} (|\vr|_{\dot{H}^s}) |\vr|_{\dot{H}^s} ( |\vr|_{\dot{H}^s} + |\nabla \u|_{H^s} + |\dot{\dr}|_{H^s} )\\
  \no & + \! |\vr|_{\dot{H}^s} (1\! +\! \P_{s-1} (|\vr|_{\dot{H}^s})) \Big[  (\mu \! + \! \xi) \Big( |\nabla \u|_{H^s} \sum_{l=1}^{4}\!  |\nabla \dr|^l_{H^s} \! + \! |\dot{\dr}|_{H^s} |\nabla \dr|_{H^s} \Big) \! +\!  |\nabla \dr|_{H^s} |\nabla \dr|_{\dot{H}^s} \Big] \\
  \leq & C_2 \mathcal{D}_{\eta} (t) \sum_{k=1}^{s+3} \mathcal{E}_{\eta}^{\frac{k}{3}} (t) \,.
\end{align}

{\it Step 4. The estimate for the dissipation of $\dr$.} Taking $\nabla^k ( 1 \leq k \leq s )$ on the orientation equation of the system \eqref{Expan-System}, and multiplying $\nabla^k \dr$ by integrating by parts, then taking advantage of the following identities
\begin{align*}
  \l \nabla^k \ddot{\dr}, \nabla^k \dr \r
  = \p_t \l \nabla^k \dot{\dr}, \nabla^k \dr \r - | \nabla^k \dot{\dr} |_{L^2}^2 + \l \nabla^k (\u \cdot \nabla \dr), \nabla^k \dot{\dr} \r + \l \nabla^k (\u \cdot \nabla \dot{\dr}), \nabla^k \dr \r
\end{align*}
and
\begin{align*}
   & \l \nabla^k (\tfrac{1}{1+\vr} \Delta \dr), \nabla^k \dr \r
   = \l \tfrac{1}{1+\vr} \nabla^k \Delta \dr, \nabla^k \dr \r + \l [\nabla^k, \tfrac{1}{1+\vr} \Delta] \dr, \nabla^k \dr \r \\
   = & - \int_{\R^N} (\tfrac{1}{1+\vr}) |\nabla^{k+1} \dr|^2 \d x - \l \nabla (\tfrac{1}{1+\vr}) \nabla^{k+1} \dr, \nabla^k \dr \r + \l [\nabla^k, \tfrac{1}{1+\vr} \Delta] \dr, \nabla^k \dr \r
\end{align*}
meanwhile, noticing that
\begin{align*}
    & \lambda_1 \l \nabla^k (\tfrac{1}{1+\vr} \dot{\dr}), \nabla^k \dr \r\\
  = & -\lambda_1 \l \tfrac{1}{1+\vr} \nabla^{k+1} \dr + \nabla (\tfrac{1}{1+\vr}) \nabla^k \dr, \nabla^{k-1} \dot{\dr} \r + \lambda_1 \sum_{\substack{a+b=k,\\a \geq  1}} \l \nabla^a (\tfrac{1}{1+\vr}) \nabla^b \dot{\dr}, \nabla^k \dr \r \, ,
\end{align*}
we then have
\begin{align}\label{Dissip-d-0}
  \no & \tfrac{1}{2} \tfrac{\d}{\d t} \big( |\nabla^k \dot{\dr} + \nabla^k \dr|^2 - |\nabla^k \dot{\dr}|^2 - |\nabla^k \dr|^2 \big) + |\tfrac{1}{\sqrt{1+\vr}} \nabla^{k+1} \dr|_{L^2}^2 - |\nabla^k \dot{\dr}|_{L^2}^2 \\
  \no = & -\l \nabla^k (\u \cdot \nabla \dr), \nabla^k \dot{\dr} \r - \l \nabla^k (\u \cdot \nabla \dot{\dr}), \nabla^k \dr \r - \l \nabla (\tfrac{1}{1+\vr}) \nabla^{k+1} \dr, \nabla^k \dr \r \\
  \no & + \l [\nabla^k, \tfrac{1}{1+\vr} \Delta] \dr, \nabla^k \dr \r - \lambda_1 \l \tfrac{1}{1+\vr} \nabla^{k+1} \dr + \nabla (\tfrac{1}{1+\vr}) \nabla^k \dr, \nabla^{k-1} \dot{\dr} \r \\
  \no & + \!\lambda_1 \!\sum_{\substack{a+b=k,\\a \geq  1}} \!\l \nabla^a (\tfrac{1}{1+\vr}) \nabla^b \dot{\dr}, \nabla^k \dr \r \!+ \!\l \nabla^k (\tfrac{1}{1+\vr} (\lambda_1 \B \!+\! \lambda_2 \A) \dr ), \nabla^k \dr \r \! + \!\l \nabla^k (\tfrac{1}{1+\vr} \widetilde{\Gamma} \dr), \nabla^k \dr \r \\
  \equiv & \sum_{i=1}^{8} \tilde{R}_{i} \,.
\end{align}

We now estimate \eqref{Dissip-d-0} term by term.

The term $\tilde{R}_1$ can be easily controlled as
\begin{align}\label{d-R-1}
  \no \tilde{R}_1 = & \l \u \cdot \nabla^{k+1} \dr, \nabla^k \dot{\dr} \r + \sum_{\substack{a+b=k,\\a \geq 1}} \l \nabla^a \u \nabla^{b+1} \dr, \nabla^k \dot{\dr} \r \\
  \no \lesssim & |\u|_{L^\infty} |\nabla^{k+1} \dr|_{L^2} |\nabla^k \dot{\dr}|_{L^2} + \sum_{\substack{a+b=k,\\a \geq 1}} |\nabla^a \u|_{L^4} |\nabla^{b+1} \dr|_{L^4} |\nabla^k \dot{\dr}|_{L^2} \\
  \lesssim & ( |\u|_{H^s} |\nabla \dr|_{\dot{H}^s} + |\nabla \u|_{H^s} |\nabla \dr|_{H^s} ) |\dot{\dr}|_{H^s}\,.
\end{align}
by taking advantage of H\"older inequality and Sobolev embedding theory.

Similar as the estimate of $\tilde{R}_1$, $\tilde{R}_2$ can be controlled as
\begin{align}\label{d-R-2}
  \no \tilde{R}_2 = & \l \u \cdot \nabla^{k+1} \dot{\dr}, \nabla^k \dr \r + \sum_{\substack{a+b=k,\\a \geq 1}} \l \nabla^a \u \nabla^{b+1} \dot{\dr}, \nabla^k \dr \r \\
  \no \lesssim & - \l \nabla^k \dot{\dr}, \nabla \u \nabla^k \dr + \u \nabla^{k+1} \dr \r + \sum_{\substack{a+b=k,\\a \geq 1}} \l \nabla^a \u \nabla^{b+1} \dot{\dr}, \nabla^k \dr \r \\
  \no \lesssim & |\nabla^k \dot{\dr}|_{L^2} ( |\nabla \u|_{L^4} |\nabla^k \dr|_{L^4} + |\u|_{L^\infty} |\nabla^{k+1} \dr|_{L^2} ) + \sum_{\substack{a+b=k,\\a \geq 1}} |\nabla^a \u|_{L^4} |\nabla^{b+1} \dot{\dr}|_{L^2} |\nabla^k \dr|_{L^4} \\
  \lesssim & |\dot{\dr}|_{H^s} ( |\nabla \u|_{H^s} |\nabla \dr|_{H^s} + |\u|_{H^s} |\nabla \dr|_{\dot{H}^s} )\,.
\end{align}

For the estimate of $\tilde{R}_3$, we have
\begin{align}\label{d-R-3}
  \tilde{R}_3 \lesssim |\nabla \vr|_{L^4} |\nabla^{k+1} \dr|_{L^2} |\nabla^k \dr|_{L^4}
  \lesssim |\vr|_{\dot{H}^s} |\nabla \dr|_{\dot{H}^s} |\nabla \dr|_{H^s} \,.
\end{align}

By using the Lemma \ref{Moser} and Sobolev embedding theory, one can deduce that
\begin{align}\label{d-R-4}
  \no \tilde{R}_4 \lesssim & ( |\nabla (\tfrac{1}{1+\vr})|_{L^\infty} |\nabla^{k-1} \Delta \dr|_{L^2} + |\nabla^k (\tfrac{1}{1+\vr})|_{L^2} |\Delta \dr|_{L^\infty} ) |\nabla^k \dr|_{L^2}
  \lesssim \P_{s} (|\vr|_{\dot{H}^s}) |\nabla \dr|_{\dot{H}^s} |\nabla \dr|_{H^s} \,.
\end{align}

We then turn to estimate the term $\tilde{R}_5$, note that
\begin{align*}
  -\lambda_1 \l \nabla(\tfrac{1}{1+\vr}) \nabla^k \dr, \nabla^{k-1} \dot{\dr} \r \lesssim |\lambda_1| |\vr|_{\dot{H}^s} |\dot{\dr}|_{H^s} |\nabla \dr|_{H^s}
\end{align*}
and
\begin{align*}
  -\lambda_1 \l (\tfrac{1}{1+\vr}) \nabla^{k+1} \dr, \nabla^{k-1} \dot{\dr} \r \leq \tfrac{1}{8} |(\tfrac{1}{\sqrt{1+\vr}}) \nabla^{k+1} \dr|_{L^2}^2 + 2 |\lambda_1|^2 |(\tfrac{1}{\sqrt{1+\vr}}) \nabla^{k-1} \dot{\dr}|_{L^2}^2\,,
\end{align*}
so
\begin{align}
  \tilde{R}_5 \leq \tfrac{1}{8} |(\tfrac{1}{\sqrt{1+\vr}}) \nabla^{k+1} \dr|_{L^2}^2 + C |\lambda_1| |\vr|_{\dot{H}^s} |\dot{\dr}|_{H^s} |\nabla \dr|_{H^s} + 2 |\lambda_1|^2 |\tfrac{1}{1+\vr}|_{L^\infty} |\nabla^{k-1} \dot{\dr}|_{L^s}^2 \,.
\end{align}

The term $\tilde{R}_6$ can be estimate by using the Moser-type inequality \eqref{Moser-type} and Sobolev embedding theory as follows:
\begin{align}\label{d-R-6}
  \tilde{R}_6 \!\lesssim \!|\lambda_1| ( |\nabla(\tfrac{1}{1+\vr})|_{L^\infty} |\nabla^{k-1} \dot{\dr}|_{L^2} \!+ \!|\nabla^k (\tfrac{1}{1+\vr})|_{L^2} |\dot{\dr}|_{L^\infty} ) |\nabla^k \dr|_{L^2} \!\lesssim\! |\lambda_1| \P_s(|\vr|_{\dot{H}^s}) |\dot{\dr}|_{H^s} |\nabla \dr|_{H^s}\,.
\end{align}

As to the term $\tilde{R}_7$, H\"older inequality and Sobolev embedding theory implies the following two bounds:
\begin{align*}
  & \l (\tfrac{1}{1+\vr}) \nabla^{k} \big( (\lambda_1 \B + \lambda_2 \A) \dr \big), \nabla^k \dr \r \\
  = & \l \nabla^{k-1} \big( (\lambda_1 \B + \lambda_2 \A) \dr \big), (\tfrac{1}{1+\vr}) \nabla^{k+1} \dr + \nabla (\tfrac{1}{1+\vr}) \nabla^k \dr \r \\
  \lesssim & (|\lambda_1|+|\lambda_2|) ( |\nabla \dr|_{\dot{H}^s} + |\vr|_{\dot{H}^s} |\nabla \dr|_{H^s} ) |\nabla \u|_{H^s} ( 1 + |\nabla \dr|_{H^s} )
\end{align*}
and
\begin{align*}
  & \sum_{\substack{a+b=k, \\a \geq 1}} \l \nabla^a (\tfrac{1}{1+\vr}) \nabla^b \big((\lambda_1 \B + \lambda_2 \A) \dr\big), \nabla^k \dr \r \\
  \lesssim & (|\lambda_1|+|\lambda_2|) \P_s (|\vr|_{\dot{H}^s}) |\nabla \u|_{H^s} |\nabla \dr|_{H^s} ( 1 + |\nabla \dr|_{H^s} )\,,
\end{align*}
then we can infer that
\begin{align}\label{d-R-7}
  \no \tilde{R}_7 = & \l (\tfrac{1}{1+\vr}) \nabla^{k} \big( (\lambda_1 \B + \lambda_2 \A) \dr \big), \nabla^k \dr \r + \sum_{\substack{a+b=k, \\a \geq 1}} \l \nabla^a (\tfrac{1}{1+\vr}) \nabla^b \big((\lambda_1 \B + \lambda_2 \A) \dr\big), \nabla^k \dr \r \\
  \lesssim & (|\lambda_1|+|\lambda_2|) ( |\nabla \dr|_{\dot{H}^s} + \P_s (|\vr|_{\dot{H}^s}) |\nabla \dr|_{H^s} ) |\nabla \u|_{H^s} ( 1 + |\nabla \dr|_{H^s} ) \,.
\end{align}

We then estimate the last term $\tilde{R}_8$ on the right-hand side of \eqref{Dissip-d-0}. Based on the representation of the Lagrangian multiplier $\widetilde{\Gamma}$, $\tilde{R}_{8}$ can be divided into three parts as follows:
\begin{align*}
  & \l \nabla^k ( \tfrac{1}{1 + \vr} \widetilde{\Gamma} \dr ), \nabla^k \dr \r
  = - \l \nabla^k ( |\dot{\dr}|^2 \dr ), \nabla^k \dr \r + \l \nabla^k (\tfrac{1}{1+\vr} |\nabla \dr|^2 \dr), \nabla^k \dr \r \\
  & \qquad \qquad \qquad \qquad \qquad - \lambda_2 \l \nabla^k ( \tfrac{1}{1+\vr} ( \dr^\top \A \dr) \dr ), \nabla^k \dr \r
  \equiv \tilde{R}^1_{8} + \tilde{R}^2_{8} + \tilde{R}^3_{8}\, .
\end{align*}
We now estimate the three terms in the right-hand side of the above identity. For the term $\tilde{R}^1_{8}$, we have
\begin{align} \label{d-R-8-1}
  \no \tilde{R}^1_{8} = & -\l \nabla^k (|\dot{\dr}|^2) \dr, \nabla^k \dr \r - \l |\dot{\dr}|^2 \nabla^k \dr, \nabla^k \dr \r  - \sum_{\substack{a+b=k, \\ 1 \leq b \leq k - 1}} \sum_{a_1 + a_2 =a} \l \nabla^{a_1} \dot{\dr} \nabla^{a_2} \dot{\dr} \nabla^b \dr, \nabla^k \dr \r \\
  \no \lesssim & |\dot{\dr}|_{L^\infty} |\nabla^k \dot{\dr}|_{L^2} |\nabla^k \dr|_{L^2} + \sum_{\substack{a+b=k, \\ a, b \geq 1}} |\nabla^a \dot{\dr}|_{L^4} |\nabla^b \dot{\dr}|_{L^4} |\nabla^k \dr|_{L^2} + |\dot{\dr}|_{L^\infty}^2 |\nabla^k \dr|_{L^2}^2 \\
  \no & + \sum_{\substack{a+b=k, \\ 1 \leq b \leq k - 1}} \sum_{a_1 + a_2 =a} |\nabla^{a_1} \dot{\dr}|_{L^4} |\nabla^{a_2} \dot{\dr}|_{L^4} |\nabla^b \dr|_{L^4} |\nabla^k \dr|_{L^4} \\
  \lesssim & |\dot{\dr}|_{H^s}^2 |\nabla \dr|_{H^s} (1+|\nabla \dr|_{H^s}) \, .
\end{align}
The term $\tilde{R}^2_{8}$ can be divided into two parts as follows:
\begin{align*}
  \tilde{R}^2_{8} = \l (\tfrac{1}{1+\vr}) \nabla^k (|\nabla \dr|^2 \dr), \nabla^k \dr \r + \sum_{\substack{a+b=k, \\ a \geq 1}} \l \nabla^a (\tfrac{1}{1+\vr}) \nabla^b (|\nabla \dr|^2 \dr), (1+\vr) \nabla^k \dr \r \, .
\end{align*}
The first part of $\tilde{R}^2_{8}$ can be controlled as follows:
\begin{align*}
  & \l (\tfrac{1}{1+\vr}) \nabla^k (|\nabla \dr|^2 \dr), \nabla^k \dr \r \\
  \leq & \l |\nabla \dr|^2, |\nabla^k \dr|^2 \r + \l |\nabla^k(|\nabla \dr|^2)|, |\nabla^k \dr| \r + \sum_{\substack{a+b=k, \\ 1 \leq b \leq k-1}} \sum_{a_1 + a_2 =a} \l |\nabla^{a_1} \dr| |\nabla^{a_2} \dr| |\nabla^b \dr|, |\nabla^k \dr| \r \\
  \leq & |\nabla \dr|_{L^4}^2 |\nabla^k \dr|_{L^4}^2 + |\nabla \dr|_{L^4} |\nabla^k \dr|_{L^4} |\nabla^{k+1} \dr|_{L^2} + \sum_{\substack{a+b =k, \\ a,b \geq 1}} |\nabla^{a+1} \dr|_{L^4} |\nabla^{b+1} \dr|_{L^4} |\nabla^k \dr|_{L^2} \\
  & + \sum_{\substack{a+b=k, \\ 1 \leq b \leq k-1}} \sum_{\substack{a_1+a_2=a,\\a_1, a_2 \geq 1}} |\nabla^{a_1+1} \dr|_{L^2} |\nabla^{a_2+1} \dr|_{L^4} |\nabla^b \dr|_{L^\infty} |\nabla^k \dr|_{L^4} \\
  & + \sum_{\substack{a+b=k, \\ a, b \geq 1}} |\nabla \dr|_{L^\infty} |\nabla^{a+1} \dr|_{L^2} |\nabla^b \dr|_{L^4} |\nabla^k \dr|_{L^4} \\
  \lesssim & |\nabla \dr|_{\dot{H}^s}^2 |\nabla \dr|_{H^s} ( |\nabla \dr|_{H^s} + 1 )\, ,
\end{align*}
where we utilize H\"older inequality, Sobolev embedding and the following estimate
\begin{align*}
  |\nabla^a \dr|_{L^4} |\nabla^b \dr|_{L^4} \lesssim |\nabla^a \dr|_{L^2}^{1-\frac{N}{4}} |\nabla^{a+1} \dr|_{L^2}^{\frac{N}{4}} |\nabla^b \dr|_{L^2}^{1-\frac{N}{4}} |\nabla^{b+1} \dr|_{L^2}^{\frac{N}{4}} \lesssim |\nabla \dr|_{\dot{H}^s} |\nabla \dr|_{H^s}
\end{align*}
for any integer $a, b \ (1 \leq a, b \leq k)$, which implied by the interpolation inequality $|f|_{L^4(\R^N)} \leq |f|_{L^2(\R^N)}^{1-\frac{N}{4}} |\nabla f|_{L^2(\R^N)}^{\frac{N}{4}}$ for $N=2,3$. For the estimate of the second part of $\tilde{R}_{8}^2$, we have
\begin{align*}
  & \sum_{\substack{a+b=k \\ a \geq 1}} \l \nabla^a (\tfrac{1}{1+\vr}) \nabla^b (|\nabla \dr|^2 \dr), \nabla^k \dr \r \\
  \leq & \l |\nabla^k (\tfrac{1}{1+\vr})| |\nabla \dr|^2, |\nabla^k \dr| \r + \sum_{\substack{a+b=k,\\ 1 \leq a \leq k-1}} \sum_{b_1+b_1=b} \l |\nabla^a (\tfrac{1}{1+\vr})| |\nabla^{b_1} \dr| |\nabla^{b_2} \dr|, |\nabla^k \dr| \r \\
  & + \sum_{\substack{a+b+c=k,\\ a, c \geq 1}} \sum_{b_1+b_1=b} \l |\nabla^a (\tfrac{1}{1+\vr})| |\nabla^{b_1} \dr| |\nabla^{b_2} \dr| |\nabla^c \dr|, |\nabla^k \dr| \r \\
  \lesssim & \!|\nabla^k (\tfrac{1}{1+\vr})|_{L^2} \!|\nabla \dr|_{L^\infty} \!|\nabla \dr|_{L^4} \!|\nabla^k \dr|_{L^4} \!+\! \sum_{\substack{a+b=k,\\ 1 \leq a \leq k-1}} \!\sum_{b_1+b_1=b}\! |\nabla^a (\tfrac{1}{1+\vr})|_{L^4} |\nabla^{b_1} \dr|_{L^4} |\nabla^{b_2} \dr|_{L^4} |\nabla^k \dr|_{L^4} \\
  & + \sum_{\substack{a+b+c=k,\\ a, c \geq 1}} \sum_{b_1+b_1=b} |\nabla^a (\tfrac{1}{1+\vr})|_{L^4} |\nabla^{b_1} \dr|_{L^4} |\nabla^{b_2} \dr|_{L^\infty} |\nabla^c \dr|_{L^4} |\nabla^k \dr|_{L^4} \\
  \lesssim & \P_s (|\vr|_{\dot{H}^s}) |\nabla \dr|_{\dot{H}^s} |\nabla \dr|_{H^s} ( 1 + |\nabla \dr|_{H^s} ) \, .
\end{align*}
With the the above two estimates, we infer that
\begin{align} \label{d-R-8-2}
  \tilde{R}_{8}^2 \lesssim ( \P_s (|\vr|_{\dot{H}^s}) + |\nabla \dr|_{\dot{H}^s} ) |\nabla \dr|_{\dot{H}^s} |\nabla \dr|_{H^s} ( 1 + |\nabla \dr|_{H^s} ) \, .
\end{align}
For the estimate of $\tilde{R}_{8}^3$, we know that
\begin{align*}
  \tilde{R}_{8}^3 = - \lambda_2 \l (\tfrac{1}{1+\vr}) \nabla^k \big( (\dr^\top \A \dr \big) \dr), \nabla^k \dr \r - \lambda_2 \sum_{\substack{a+b=k,\\ a \geq 1}} \l \nabla^a (\tfrac{1}{1+\vr}) \nabla^b \big( (\dr^\top \A \dr) \dr \big), \nabla^k \dr \r \, .
\end{align*}
Then by the following two estimates
\begin{align*}
 \l (\tfrac{1}{1+\vr}) \nabla^k \big( (\dr^\top \A \dr) \dr \big), \nabla^k \dr \r
  = & \l \nabla^{k-1} \big( (\dr^\top \A \dr \big) \dr), (\tfrac{1}{1+\vr}) \nabla^{k+1} \dr + \nabla (\tfrac{1}{1+\vr}) \nabla^k \dr \r \\
  \lesssim & |\nabla \u|_{H^s} ( |\nabla \dr|_{\dot{H}^s} + |\vr|_{\dot{H}^s} |\nabla \dr|_{H^s} ) \sum_{l=0}^{3} |\nabla \dr|_{H^s}^l
\end{align*}
and
\begin{align*}
   \sum_{\substack{a+b=k\\ a \geq 1}} \l \nabla^a (\tfrac{1}{1+\vr}) \nabla^b( (\dr^\top \A \dr) \dr), \nabla^k \dr \r
  \lesssim  \P_s (|\vr|_{\dot{H}^s}) |\nabla \u|_{H^s} \sum_{l=1}^{4} |\nabla \dr|_{H^s}^l \, ,
\end{align*}
we have
\begin{align}\label{d-R-8-3}
  \tilde{R}_{8}^3 \lesssim |\lambda_2| |\nabla \u|_{H^s} |\nabla \dr|_{\dot{H}^s} \sum_{l=0}^{3} |\nabla \dr|_{H^s}^l + |\lambda_2| \P_s (|\vr|_{\dot{H}^s}) |\nabla \u|_{H^s} \sum_{l=1}^{4} |\nabla \dr|_{H^s}^l \, .
\end{align}
Therefore, combining with the estimate \eqref{d-R-8-1}, \eqref{d-R-8-2} and \eqref{d-R-8-3} yields to
\begin{align}\label{d-R-8}
  \no \tilde{R}_{8} \lesssim & |\dot{\dr}|_{H^s}^2 |\nabla \dr|_{H^s} (1+|\nabla \dr|_{H^s}) + ( \P_s (|\vr|_{\dot{H}^s}) + |\nabla \dr|_{\dot{H}^s} ) |\nabla \dr|_{\dot{H}^s} |\nabla \dr|_{H^s} ( 1 + |\nabla \dr|_{H^s} ) \\
  & + |\lambda_2| |\nabla \u|_{H^s} |\nabla \dr|_{\dot{H}^s} \sum_{l=0}^{3} |\nabla \dr|_{H^s}^l + |\lambda_2| \P_s (|\vr|_{\dot{H}^s}) |\nabla \u|_{H^s} \sum_{l=1}^{4} |\nabla \dr|_{H^s}^l \, .
\end{align}

Noticing that
\begin{align*}
  C ( |\lambda_1| \!+\! |\lambda_2| ) |\nabla \u|_{H^s} |\nabla \dr|_{\dot{H}^s}
  \leq \tfrac{1}{8} \sum_{k=1}^{s} |\tfrac{1}{\sqrt{1+\vr}} \nabla^{k+1} \dr|_{L^2}^2 + 2 C^2 ( |\lambda_1| + |\lambda_2| )^2 |1+\vr|_{L^\infty} |\nabla \u|_{H^s}^2 \, ,
\end{align*}
and by substituting the estimates \eqref{d-R-1}-\eqref{d-R-7} and \eqref{d-R-8} into \eqref{Dissip-d-0}, then summing up with $1 \leq k \leq s$, we have
\begin{align}\label{Dissip-d}
  \no \tfrac{1}{2} & \tfrac{\d}{\d t} ( |\dot{\dr} + \dr|_{\dot{H}^s}^2 - |\dot{\dr}|_{\dot{H}^s}^2 - |\dr|_{\dot{H}^s}^2 ) + \tfrac{3}{4} \sum_{k=1}^{s} |\tfrac{1}{\sqrt{1+\vr}} \nabla^{k+1} \dr|_{L^2}^2 - |\dot{\dr}|_{H^s}^2 - 2 |\lambda_1|^2 |\tfrac{1}{1+\vr}|_{L^\infty} |\dot{\dr}|_{H^s}^2 \\
  \no & - 2 C^2 ( |\lambda_1| + |\lambda_2| )^2 |1+\vr|_{L^\infty} |\nabla \u|_{H^s}^2 \\
  \no \lesssim & ( |\u|_{H^s} |\nabla \dr|_{\dot{H}^s} + |\nabla \u|_{H^s} |\nabla \dr|_{H^s} ) |\dot{\dr}|_{H^s} + (1+|\lambda_1|) \P_s(|\vr|_{\dot{H}^s}) ( |\nabla \dr|_{\dot{H}^s} + |\dot{\dr}|_{H^s} ) |\nabla \dr|_{H^s} \\
  \no & + ( \P_s (|\vr|_{\dot{H}^s}) + |\nabla \dr|_{\dot{H}^s} ) |\nabla \dr|_{\dot{H}^s} |\nabla \dr|_{H^s} ( 1 + |\nabla \dr|_{H^s} ) + |\dot{\dr}|_{H^s}^2 |\nabla \dr|_{H^s} (1+|\nabla \dr|_{H^s}) \\
  \no & + (|\lambda_1| + |\lambda_2|) |\nabla \u|_{H^s} |\nabla \dr|_{\dot{H}^s} \sum_{l=1}^{3} |\nabla \dr|_{H^s}^l + (|\lambda_1| + |\lambda_2|) \P_s (|\vr|_{\dot{H}^s}) |\nabla \u|_{H^s} \sum_{l=1}^{4} |\nabla \dr|_{H^s}^l \\
  \lesssim & C_3 \mathcal{D}_{\eta} (t) \sum_{k=1}^{s+3} \mathcal{E}_{\eta}^{\tfrac{k}{2}} (t)
\end{align}

If we denote
\begin{align*}
  C(\mu, \xi) := & C (\mu^2 + \xi^2) |\tfrac{1+\vr}{p'(1+\vr)}|_{L^\infty} \,, \\
  C(\lambda_1) := & 1 + 2 |\lambda_1|^2 |\tfrac{1}{1+\vr}|_{L^\infty} \,,
\end{align*}
and noticing that
\begin{align*}
  |\nabla^k \dot{\dr}|_{L^2}^2 \leq 2 |\nabla^k \dot{\dr} + (\nabla^k \B) \dr + \tfrac{\lambda_2}{\lambda_1} (\nabla^k \A) \dr|_{L^2}^2 + 2( 1 - \tfrac{|\lambda_2|}{\lambda_1} )^2 |\nabla^{k+1} \u|_{L^2}^2 \,,
\end{align*}
then we have
\begin{align*}
  C(\mu,\xi) (|\nabla \u|_{H^s}^2 + |\dot{\dr}|_{H^s}^2)
  \leq & C(\mu,\xi) \sum_{k=0}^{s} |\nabla^k \dot{\dr} + (\nabla^k \B) \dr + \tfrac{\lambda_2}{\lambda_1} (\nabla^k \A) \dr|_{L^2}^2\\
  & + C(\mu,\xi) \big( 1 + 2 (1 - \tfrac{|\lambda_2|}{\lambda_1})^2 \big) |\nabla \u|_{H^s}^2\,.
\end{align*}
Choosing
\begin{align*}
  \eta_0 = \tfrac{1}{2} \min \Big\{ & \frac{1}{ |\frac{1+\vr}{p'(1+\vr)}|_{L^\infty} }, \frac{\mu_4}{4 C_1(\mu,\xi)}, \frac{\mu_4}{4 C_2(\mu,\xi)}, \frac{-\lambda_1}{4 C(\mu,\xi)}, \frac{-\lambda_1}{4 C(\lambda_1)}, \frac{\frac{1}{2} \mu + \xi}{2 |1+\vr|_{L^\infty}}, 1 \Big\}
\end{align*}
and multiplying \eqref{Dissip-rho} and \eqref{Dissip-d} by $\eta_1, \eta_2 \in (0, \eta_0]$ respectively, and then adding them to the inequality \eqref{Expan-H-s}, we have
\begin{align}\label{Expan-Ener-Est-1}
  \no & \tfrac{1}{2} \tfrac{\d}{\d t} \Big\{ \int_{\R^N} \tfrac{p'(1+\vr)}{1+\vr} |\vr|^2 \d x + \sum_{k=1}^{s} \int_{\R^N} ( \tfrac{p'(1+\vr)}{1+\vr} - \eta_1 ) |\nabla^k \vr|^2 \d x \\
  \no & + \sum_{k=0}^{s-1} \int_{\R^N} ((1\!+\!\vr) - \eta_1) |\nabla^k \u|^2 + ((1\!+\!\vr) - \eta_2) |\nabla^{k+1} \dot{\dr}|^2 + (1-\eta_2) |\nabla^{k+1} \dr|^2 \d x \\
  \no & + \eta_1 |\u + \nabla \vr|_{H^{s-1}}^2 + \eta_2 |\dot{\dr} + \dr|_{\dot{H}^s}^2 + |\sqrt{1\!+\!\vr} \nabla^s \u|_{L^2}^2 + |\sqrt{1\!+\!\vr} \dot{\dr}|_{L^2}^2 + |\nabla^{s+1} \dr|_{L^2}^2 \Big\} \\
  \no & + \tfrac{3}{4} \eta_1 \sum_{k=1}^{s} \int_{\R^N} \tfrac{p'(1+\vr)}{1+\vr} |\nabla^k \vr|^2 \d x \!+\! \tfrac{3}{4} \eta_2 |\nabla \dr|_{\dot{H}^s_{\frac{1}{1+\vr}}}^2 + ( \tfrac{1}{2} \mu_4 \!+\! \xi \!-\! |1\!+\!\vr|_{L^\infty} \eta_1) |\div \u|_{H^s}^2 \\
  \no & + ( \tfrac{1}{2} \mu_4 - C_1(\mu, \xi) \eta_1 - C_2(\mu, \xi) \eta_2 ) |\nabla \u|_{H^s}^2 + (\mu_5+\mu_6+\tfrac{\lambda_2^2}{\lambda_1}) \sum_{k=0}^{s} |(\nabla^k \A) \dr|_{L^2}^2\\
  \no & \!+\!\mu_1 \!\sum_{k=0}^{s} \!|\dr^\top \!(\nabla^k \!\A) \dr|_{L^2}^2 \!+\! ( -\!\lambda_1 \!-\! 2 \eta_1 C(\mu,\xi) \!-\! 2 \eta_2 C (\lambda_1) ) \!\sum_{k=0}^{s}\! |\nabla^k \dot{\dr}\! +\! (\nabla^k \B) \dr \!+ \!\tfrac{\lambda_2}{\lambda_1} (\nabla^k \A) \dr|_{L^2}^2 \\
  & \lesssim C \mathcal{D}_{\eta} (t) \sum_{k=1}^{s+4} \mathcal{E}_{\eta}^{\frac{k}{2}} (t) \,,
\end{align}
where
\begin{align*}
  C_1(\mu, \xi) :=& C(\mu, \xi) + 2 (1-\tfrac{|\lambda_2|}{\lambda_1})^2 C(\mu, \xi),\\
  C_2(\mu, \xi) :=& 2 C^2 ( |\lambda_1| + |\lambda_2| )^2 |1+\vr|_{L^\infty} + 2 (1-\tfrac{|\lambda_2|}{\lambda_1})^2 C(\lambda_1) \,.
\end{align*}

With the estimate \eqref{Expan-Ener-Est-1} and the definition $\mathcal{E}_{\eta} (t)$ and $\mathcal{D}_{\eta} (t)$ in hand, we then derive that
\begin{align}\label{Global-Energy-Bound}
  \tfrac{1}{2} \tfrac{\d}{\d t} \mathcal{E}_{\eta} (t) + \mathcal{D}_{\eta} (t) \leq C \mathcal{D}_{\eta} (t) \sum_{k=1}^{s+4} \mathcal{E}_{\eta}^{\frac{k}{2}} (t) \,,
\end{align}
where the constant $C$ depends only on the Leslie coefficients. So we complete the proof of Lemma \ref{lemma-6-1}.
\end{proof}

{\textbf{Proof of Theorem \ref{theorem-2}: Global well-posedness.}}
As the end step, we use the continuum arguments to prove the global-in-time solutions of the compressible Ericksen-Leslie liquid crystal system. It is obvious that, there exists constants $C^{\#}$ and $C_\#$,
\begin{align*}
  C_\# =& \min \Big\{ \inf |\tfrac{p'(1+\vr)}{1+\vr}| - \eta_0, 1-\eta_0, 1 \Big\} \,, \\
  C^\# =& |\tfrac{p'(1+\vr)}{1+\vr}|_{L^\infty} + 4|1+\vr|_{L^\infty} + 4 \,,
\end{align*}
such that
\begin{align*}
  C_{\#} \tilde{E}(t) \leq \mathcal{E}_{\eta} (t) \leq C^{\#} \tilde{E}(t)
\end{align*}
and
\begin{align*}
  \mathcal{D}_{\eta} (t) \geq \tfrac{1}{4} \mu_4 |\nabla \u|_{H^s}^2 + (\tfrac{1}{4} \mu_4 + \tfrac{1}{2} \xi) |\div \u|_{H^s}^2 \,,
\end{align*}
where $\tilde{E}(t) = |\u|_{H^s}^2 + |\dot{\dr}|_{H^s}^2 + |\nabla \dr|_{H^s}^2 + |\vr|_{H^s}^2$. As a consequent, we have
\begin{align*}
  C_\# \tilde{E}^{in} \leq \mathcal{E}_{\eta} (0) \leq C^\# \tilde{E}^{in} \,.
\end{align*}
We make a definition as
\begin{align*}
  T^* = \sup \big\{ \tau>0; \sup_{t \in [0,\tau]} C \sum_{k=1}^{s+4} \mathcal{E}_{\eta}^{\frac{k}{2}} (t) \leq \tfrac{1}{2} \big\} \geq 0 \,,
\end{align*}
where the constant $C>0$ is mentioned as in Lemma 8.1. We then choose the sufficient small positive number $\eps_0 = \tfrac{1}{C^\#} \min \{1, \tfrac{1}{4 (s+4) C} \} >0 $. If the initial energy $\tilde{E}^{in} \leq \eps_0$, we can deduce that
\begin{align*}
  C \sum_{k=1}^{s+4} \mathcal{E}_{\eta}^{\frac{k}{2}} (0) \leq \tfrac{1}{4} < \tfrac{1}{2} \,.
\end{align*}
Then by taking advantage of the continuity of the energy functional $\mathcal{E}_{\eta} (t)$, one derives that $T^*>0$. Thus
\begin{align*}
  \tfrac{1}{2} \tfrac{\d}{\d t} \mathcal{E}_{\eta} (t) + \big[ 1 - C \sum_{k=1}^{s+4} \mathcal{E}_{\eta}^{\frac{k}{2}} (t) \big] \mathcal{D}_{\eta} (t) \leq 0 \,,
\end{align*}
holds for all $t \in [0, T^*]$, which also implies that we have $\mathcal{E}_{\eta} (t) \leq \mathcal{E}_{\eta} (0) \leq C^\# \tilde{E}^{in} $ for all $t \in [0, T^*]$. Then we can derive that
\begin{align*}
  \sup_{t \in [0, T^*]} \Big\{ C \sum_{k=1}^{s+4} \mathcal{E}_{\eta}^{\frac{k}{2}} (t) \Big\} \leq \tfrac{1}{4} \,.
\end{align*}
Based on the the above analysis, we claim that $T^* = + \infty$. Otherwise, the continuity of the energy $\mathcal{E}_{\eta} (t)$ implies that there exists a sufficiently small positive $\eps>0$ such that
\begin{align*}
  \sup_{t \in [0, T^*+\eps]} \Big\{ C \sum_{k=1}^{s+4} \mathcal{E}_{\eta}^{\frac{k}{2}} (t) \Big\} \leq \tfrac{3}{8} < \tfrac{1}{2} \,.
\end{align*}
which contradicts to the definition of $T^*$. As a consequence, there is a constant $C_1$ depends only on the Leslie coefficients, such that the following inequality hold
\begin{align*}
  & \sup_{t \geq 0} \big( |\u|_{H^s}^2 + |\vr|_{H^s}^2 + |\dot{\dr}|_{H^s}^2 + |\nabla \dr|_{H^s}^2 \big) (t) \\
   & + \tfrac{1}{2} \mu_4 \int_{0}^{\infty} |\nabla \u|_{H^s}^2 \d t + (\tfrac{1}{2} \mu_4 + \xi) \int_{0}^{\infty} |\div \u|_{H^s}^2 \d t\leq C_1 \tilde{E}^{in} \,.
\end{align*}
Thus we complete the proof of Theorem \ref{theorem-2}.
\appendix

\end{document}